\documentclass[a4paper,11pt]{amsart} 
\usepackage[T1]{fontenc}
\usepackage[utf8]{inputenc}
\usepackage{amsmath,amsthm, amssymb}
\usepackage{mathtools}
\usepackage{enumitem}
\usepackage{gensymb}
\usepackage[scr=boondoxo,scrscaled=1.05]{mathalfa}
\usepackage{lmodern}
\usepackage{float}
\usepackage{esint}
\usepackage[usenames,dvipsnames]{pstricks}
\usepackage{enumitem} 
\usepackage{chngcntr}
\usepackage{dsfont}
\usepackage[margin=2.5cm]{geometry}
\usepackage{color}
\usepackage{graphicx}
\usepackage[titletoc,toc,title]{appendix}
\usepackage[pdfusetitle=true,bookmarksdepth=3]{hyperref}
\usepackage{amsrefs}
\usepackage{tikz}
\usepackage{todonotes}
\usepackage{soul}
\usepackage{accents}
\usepackage{bbm}

\colorlet{jvs}{cyan!20}
\colorlet{bvv}{red!20}
\usepackage{IEEEtrantools}

\newtheorem{theorem}{Theorem}[section]
\newtheorem{lemma}[theorem]{Lemma}
\newtheorem{prop}[theorem]{Proposition}

\theoremstyle{definition}
\newtheorem{defn}[theorem]{Definition}
\newtheorem{cor}[theorem]{Corollary}
\newtheorem{rem}[theorem]{Remark}

\DeclareMathOperator*\tr{tr}

\DeclareMathOperator*\dist{dist}

\DeclareMathOperator*\loc{loc}
\DeclareMathOperator*\norm{\mathcal{N}}
\DeclareMathOperator{\sign}{sign}

\DeclarePairedDelimiterX{\intvc}[2]{[}{]}{#1,#2}
\DeclarePairedDelimiterX{\intvl}[2]{(}{]}{#1,#2}
\DeclarePairedDelimiterX{\intvr}[2]{[}{)}{#1,#2}
\DeclarePairedDelimiterX{\intvo}[2]{(}{)}{#1,#2}

\DeclarePairedDelimiterX\set[1]\{\}{%

#1
}

\makeatletter
\newcommand{\widerof}[3][c]{\mathpalette\widerof@{{#1}{#2}{#3}}}
\newcommand{\widerof@}[2]{\widerof@@{#1}#2}
\newcommand{\widerof@@}[4]{%
    \begingroup
    \sbox\z@{$\m@th#1#3$}%
    \sbox\tw@{$\m@th#1#4$}%
    \makebox[\ifdim\wd\z@>\wd\tw@ \wd\z@\else \wd\tw@\fi][#2]{$\m@th#1#3$}%
    \endgroup
}
\newcommand*\diff{\mathop{}\, d}

\makeatother

\numberwithin{equation}{section}
\numberwithin{figure}{section}

\author{Bohdan Bulanyi and Berardo Ruffini}

\keywords{Anisotropic mean curvature flow; approximation schemes; front propagation; level set approach; viscosity solution; normal velocity; Wulff shape.}

\subjclass[2020]{35B40, 35D40, 35G25, 35K10, 45L05, 53E10}

\date{\today}

\begin{document}

\address{
Universit\`{a} di Bologna, Dipartimento di Matematica, Piazza di Porta San Donato 5, 40126 Bologna, Italy}

\email{bohdan.bulanyi@unibo.it}

\address{
Universit\`{a} di Bologna, Dipartimento di Matematica, Piazza di Porta San Donato 5, 40126 Bologna, Italy}

\email{berardo.ruffini@unibo.it}

\address{
}

\email{}

    \title[Approximation schemes for anisotropic mean curvature flows] 
    {Threshold dynamics approximation schemes for anisotropic mean curvature flows with a forcing term}
   
      \begin{abstract}
 We establish the convergence of threshold dynamics-type approximation schemes to propagating fronts evolving according to an anisotropic mean curvature motion in the presence of a forcing term depending on both time and position, thus generalizing the consistency result obtained in \cite{Ishii-Pires-Souganidis} by extending the results obtained in \cite{Caffarelli-Souganidis} for $\alpha \in [1,2)$ to anisotropic kernels and in the presence of a driving force. The limit geometric evolution is of a variational type and can be approximated, at a large scale, by eikonal-type equations modeling dislocations dynamics. We prove that it preserves convexity under suitable convexity assumptions on the forcing term and that convex evolutions of compact sets are unique. If the initial set is bounded and sufficiently large, and the driving force is constant, then the corresponding generalized front propagation is asymptotically similar to the Wulff shape.
 \end{abstract}
  
  \thanks{B. Bulanyi and B. Ruffini were partially supported by the PRIN project 2022R537CS \emph{$NO^3$ - Nodal Optimization, NOnlinear elliptic equations, NOnlocal geometric problems, with a focus on regularity. The authors would like to warmly thank the referee for the careful reading of the paper and for the comments and suggestions which helped them to improve the paper.}}
  
   \maketitle

    \tableofcontents

\section{Introduction}
\subsection{General discussion}
We study the convergence of a class of threshold dynamics-type approximation schemes to hypersurfaces moving with normal velocity equal to the sum of a multiple of an anisotropic mean curvature and a multiple of a forcing term depending on both time and position. In order to describe the general scheme considered in this paper  and to present the results obtained, we begin with the following short story. In 1992, Bence, Merriman and Osher introduced a scheme to compute mean curvature motion by iterating the heat equation \cite{Bence-Merriman-Osher}. The proofs of the Bence, Merriman, and Osher algorithm were provided by Evans \cite{Evans_1993} and Barles and Georgelin \cite{Barles-Georgelin}. Another proof was given by Ishii for a more general isotropic symmetric kernel \cite{Ishii1995generalization}. This was generalized by Ishii, Pires, and Souganidis to the case of anisotropic schemes with kernels having finite second moment \cite{Ishii-Pires-Souganidis}. In \cite{Slepcev_2003}, Slep\v{c}ev proved the convergence of a class of nonlocal threshold dynamics. It is also worth noting that  Da~Lio, Forcadel and Monneau proved that the solution of the nonlocal Hamilton-Jacobi equation modeling dislocations dynamics converges, at a large scale, to the solution of the anisotropic mean curvature motion \cite{DaLio_Forcadel_Monneau}. The Bence, Merriman, and Osher scheme with kernels associated with the fractional heat equation with the fractional Laplacian of order $\alpha \in (0,2)$ (namely, the heat equations, where the usual Laplacian is replaced with the fractional one of order $\alpha \in (0,2)$) was considered by Caffarelli and Souganidis \cite{Caffarelli-Souganidis}. It was proved in \cite{Caffarelli-Souganidis} that for $\alpha \in (0,1)$ the resulting interface moves with normal velocity, which is nonlocal of ``fractional-type'', while for $\alpha \in [1,2)$ the resulting interface moves by weighted mean curvature. The consistency result of \cite{Caffarelli-Souganidis}  was extended in the nonlocal cases (i.e., when $\alpha \in (0,1)$) to the anisotropic case with the presence of an external driving force depending only on time by Chambolle, Novaga, and Ruffini \cite{Chambolle_Novaga_Ruffini_2017}. In the present paper, we extend and improve the algorithm considered in \cite{Chambolle_Novaga_Ruffini_2017} for the cases where $\alpha \in [1,2)$. Namely, we generalize the consistency result obtained in \cite{Ishii-Pires-Souganidis} by extending the results obtained in \cite{Caffarelli-Souganidis}  for $\alpha \in [1,2)$ to anisotropic kernels and
in the presence of a driving force depending on both time and position. Observe that the kernels $P_{\alpha} \in C(\mathbb{R}^{N})$ considered in our paper (see \eqref{Kforalpha12}) do not satisfy the assumption (3.4) on the measurable kernel $f:\mathbb{R}^{N}\to \mathbb{R}$ from \cite{Ishii-Pires-Souganidis}, used in the proofs of Lemmas~3.1 and 3.2 from \cite{Ishii-Pires-Souganidis} for the existence of  a decreasing function $\omega \in C([0,+\infty), [0,+\infty))$ such that $\omega(R)\to 0$ as $R\to +\infty$ and $\int_{\mathbb{R}^{N}\setminus B_{R(\varrho)}(0)}f(x)\diff x \leq \omega(R(\varrho))\varrho$, where $R(\varrho)\to +\infty$ and $\sqrt{\varrho} R(\varrho)\to 0$ as $\varrho \to 0+$. However, a careful analysis shows that if $\alpha \in (1,2)$, then the proofs of Lemmas 3.1 and 3.2 from \cite{Ishii-Pires-Souganidis} can be adapted to our kernels $P_{\alpha}$, and the constant external force can be replaced by a globally bounded external force depending on the position. Indeed, replacing the kernel $f$ by $P_{\alpha}$ in the proofs of Lemmas 3.1 and 3.2 from \cite{Ishii-Pires-Souganidis} and choosing $R(\varrho)=\varrho^{-\theta}$ for some $\theta \in (1/\alpha, 1)$, as we did in the proof of Proposition~\ref{propconsistency1} (see \eqref{estIIintegral}), and taking into account that $\int_{\mathbb{R}^{N}\setminus B_{R(\varrho)}(0)}P_{\alpha}(x)\diff x$ behaves like $\int_{\mathbb{R}^{N}\setminus B_{R(\varrho)}(0)}P_{\alpha}(x)\diff x \sim C\varrho^{\theta \alpha}$ as $\varrho\to 0+$, where $\theta\alpha>1$, one can circumvent (3.4) in \cite{Ishii-Pires-Souganidis} and replace the term $\omega(R(\varrho))\varrho$ by $C\varrho^{\theta\alpha}$. On the other hand, the assumptions (3.2), (3.3), (3.7) from \cite{Ishii-Pires-Souganidis} are strongly used in the proofs of Lemmas 3.1 and 3.2 from \cite{Ishii-Pires-Souganidis}, and our kernels $P_{\alpha}$ do not satisfy them in the case when $\alpha=1$.  Let us put it more precisely.
\subsection{Mathematical setting of the problem}
Given $\alpha \in [1,2)$, $N\geq 2$ and a norm $\norm$ on $\mathbb{R}^{N}$, for each $x \in \mathbb{R}^{N}$ and $t \in (0,+\infty)$, we define 
\begin{equation}\label{Kforalpha12} 
P_{\alpha}(x)=\frac{1}{1+\norm(x)^{N+\alpha}}\,\ \,\ \textup{and} \,\ \,\ p_{\alpha}(x, t)=t^{-\frac{N}{\alpha}}P_{\alpha}(x t^{-\frac{1}{\alpha}}).
\end{equation}
It is worth noting that locally uniformly in $\mathbb{R}^{N}\setminus \{0\}$ and hence in $L^{1}_{\loc}(\mathbb{R}^{N}\setminus \{0\})$, 
\begin{equation*}
\lim_{t\to 0+}t^{-1}p_{\alpha}(\cdot, t)=\norm(\cdot)^{-(N+\alpha)}.
\end{equation*}
 We denote by $h>0$ the size of the time step and choose $\sigma_{\alpha}(h)$ as follows
\begin{equation}\label{eq_definingthethresholding}
\begin{cases}
\sigma_{\alpha}(h)= h^{\frac{\alpha}{2}}   \,\ &\text{if}\,\  \alpha \in (1,2),\\
h=\sigma^{2}_{\alpha}(h)|\ln(\sigma_{\alpha}(h))| \,\  &\text{if}\,\  \alpha=1.
\end{cases}
\end{equation}
Let $\Omega_{0}$ be an open subset of $\mathbb{R}^{N}$ with boundary $\Gamma_{0}=\partial \Omega_{0}$ and $g \in C(\mathbb{R}^{N}\times [0,+\infty))$. For each $n \in \mathbb{N}\backslash \{0\}$, we define the functions $u_{h}(\cdot, nh): \mathbb{R}^{N} \to \{-1, 1\}$ by induction. In particular,
\[
u_{h}(\cdot, (n+1)h)= \sign(J_{h}*u_{h}(\cdot, nh)+g(\cdot, nh) \beta(\alpha, h))  \,\ \text{in} \,\ \mathbb{R}^{N},
\]
where $u_{h}(\cdot, 0)=\mathbbm{1}_{\Omega_{0}}-\mathbbm{1}_{\smash{\overline{\Omega}}_{0}^{c}}$, $\sign(t)=1$ if $t>0$ and $\sign(t)=-1$ if $t\leq 0$, $\mathbbm{1}_{A}$ denotes the characteristic function of $A\subset \mathbb{R}^{N}$,
\begin{equation}\label{defofkernelJh}
J_{h}(x)=p_{\alpha}(x, \sigma_{\alpha}(h)),
\end{equation}
(see \eqref{Kforalpha12} and \eqref{eq_definingthethresholding}) and 
\begin{equation}\label{eq_betathresholdalphatime}
\beta(\alpha, h)=
\begin{cases}
\sigma_{\alpha}(h)^{\frac{1}{\alpha}}=h^{\frac{1}{2}}\,\ &\text{if} \,\ \alpha \in (1,2),\\
\sigma_{\alpha}(h)|\ln(\sigma_{\alpha}(h))|\,\  &\text{if}\,\ \alpha=1.
\end{cases}
\end{equation}
This algorithm generates functions $u_{h}(\cdot, nh)$ and open sets $\Omega^{h}_{nh}$ defined by
\[
u_{h}(\cdot, nh)=\mathbbm{1}_{\Omega^{h}_{nh}}-\mathbbm{1}_{(\Omega^{h}_{nh})^{c}} \,\ \text{in} \,\ \mathbb{R}^{N}
\]
and 
\begin{equation}\label{eq_nsetdiscreteh}
\Omega^{h}_{nh}=\{x \in \mathbb{R}^{N}: J_{h}*u_{h}(\cdot, (n-1)h)(x)>-g(x, (n-1)h)\beta(\alpha, h)\}.
\end{equation}
We shall prove that, when $h\to 0+$, the discrete evolution $\Gamma_{0}\to \Gamma^{h}_{nh}=\partial \Omega^{h}_{nh}$ converges, in a suitable sense, to the motion $\Gamma_{0}\to \Gamma_{t}$ with normal velocity equal to the sum of a multiple of the anisotropic mean curvature (depending on $\norm$) and a multiple of the external force $g$.  
\subsection{Main results}
The anisotropic mean curvature motion in the presence of the external force $g \in C(\mathbb{R}^{N}\times [0,+\infty))$ that we shall obtain in the limit corresponds to the level set pde
\begin{equation}\label{weighted MC equation}
\partial_{t} u = \mu_{\alpha}(Du) (F_{\alpha} (D^{2} u, Du)+g|Du|)  \,\ \text{in} \,\ \mathbb{R}^{N}\times (0,+\infty),\\
\end{equation}
supplemented with the initial condition 
\begin{equation*}
u(\cdot, 0)=u_{0}(\cdot)\,\ \text{in} \,\ \mathbb{R}^{N}
\end{equation*}
for some uniformly continuous function $u_{0}:\mathbb{R}^{N}\to \mathbb{R}$, where for each $p \in \mathbb{R}^{N}\setminus\{0\}$ and for each $N \times N$ symmetric real matrix $M$,
\begin{equation}\label{eq_defofmobilitygeneralapproachintro}
\mu_{\alpha}\left(\frac{p}{|p|}\right)=\left(2\int_{\{x \in \mathbb{R}^{N}: \langle x, p\rangle=0\}} P_{\alpha}(x)\diff \mathcal{H}^{N-1}(x)\right)^{-1} 
\end{equation}
(see \eqref{Kforalpha12}) and 
\begin{equation}\label{eq_defofanisotropicmeancurvoppdevmodalpha}
F_{\alpha}(M,p)=\tr\left(M\mathcal{A}\left(\frac{p}{|p|}\right)\right)
\end{equation}
with
\begin{equation}\label{eq_defofanisotropyformc}
\mathcal{A}\left(\frac{p}{|p|}\right)=C_{N, \alpha}\int_{\mathbb{S}^{N-1}\cap \{x \in \mathbb{R}^{N}: \langle x, p\rangle=0\}}\theta \otimes \theta \frac{\diff \mathcal{H}^{N-2}(\theta)}{\norm(\theta)^{N+1}},
\end{equation}
where
\begin{equation}\label{eq_constcominganalphan}
C_{N,\alpha}=
\begin{cases}
\displaystyle \int_{0}^{+\infty}t^{N}(1+t^{N+\alpha})^{-1}\diff t \,\ & \text{if} \,\ \alpha \in (1,2),\\
1 \,\ &\text{if}\,\ \alpha=1.
\end{cases}
\end{equation}
\begin{rem}\label{rem geometric flow}
If $X$ and $Y$ are symmetric real $N\times N$ matrices such that $X\leq Y$ and $p \in \mathbb{R}^{N}\setminus \{0\}$, then $-F_{\alpha}(X, p)\geq -F_{\alpha}(Y,p)$, and hence $-\mu_{\alpha}(p)(F_{\alpha}(M,p)+g|p|)$ is degenerate elliptic. Also, $-\mu_{\alpha}(p)(F_{\alpha}(M,p)+g|p|)$ is geometric (the reader may consult \cite{Barles-Soner-Souganidis, Ishii-Pires-Souganidis, Chen-Giga-Goto, Ishii-Souganidis} for more details on the geometric equations), because $M\to F_{\alpha}(M,p)$ is linear and 
\[
 F_{\alpha}(M,p)=F_{\alpha}\left(\left(\mathrm{Id} - \frac{p}{|p|}\otimes \frac{p}{|p|}\right)M, \frac{p}{|p|}\right),
 \]
 which comes from \eqref{eq_defofanisotropicmeancurvoppdevmodalpha}, \eqref{eq_defofanisotropyformc} and the fact that $(\theta \otimes \theta)(p\otimes p)= \langle \theta, p \rangle \theta\otimes p$. 
\end{rem}
\begin{rem}
In the particular case where $\norm$ is the usual Euclidean norm, we obtain
\[
\mathcal{A}\left(\frac{p}{|p|}\right)=\frac{C_{N, \alpha}\mathcal{H}^{N-2}(\mathbb{S}^{N-2})}{N-1}\mathrm{Id}_{\{\langle x, p\rangle=0\}}
\]
and hence
\[
F_{\alpha}(M, p)=\frac{C_{N, \alpha}\mathcal{H}^{N-2}(\mathbb{S}^{N-2})}{N-1} \tr\left(\left(\mathrm{Id} - \frac{p}{|p|}\otimes \frac{p}{|p|}\right)M\right).
\]
We recover the classical mean curvature motion up to the factor $C_{N, \alpha}\mathcal{H}^{N-2}(\mathbb{S}^{N-2})\mu_{\alpha}/(N-1)$, where 
\[\mu_{\alpha}=\left(2\int_{\mathbb{R}^{N-1}} \frac{\diff \mathcal{H}^{N-1}(x)}{1+|x|^{N+\alpha}}\right)^{-1} .
\]
\end{rem}
Using the theory of viscosity solutions of Crandall, Ishii, and Lions \cite{Crandall-Ishii-Lions-1992}, one can give the precise meaning of a solution of the equation \eqref{weighted MC equation} (see Definition~\ref{def ofsubsolwmc} and Theorem~\ref{thm equivdefvs}). We point out that the equation \eqref{weighted MC equation}, supplemented with the initial condition $u(\cdot, 0)=u_{0}(\cdot)$ in $\mathbb{R}^{N}$ for some uniformly continuous function $u_{0}$, admits a unique viscosity solution (see, for instance, \cite{Barles-Soner-Souganidis, Chen-Giga-Goto, Evans-Spruck, Ishii-Souganidis}).

Next, we recall that, given a bounded sequence $(u_{h}(\cdot, nh))_{n\in \mathbb{N}}$ of bounded functions, the ``half-relaxed'' limits $\liminf_{*}u_{h}$ and $\limsup^{*}u_{h}$ are defined by
\begin{equation}\label{eq_lowuppsemlimits}
\begin{cases}
 \liminf_{*}u_{h}(x,t):=\displaystyle \liminf_{\substack{y\to x \\ nh\to t}}u_{h}(y, nh),\\
 \limsup^{*}u_{h}(x,t):=\displaystyle \limsup_{\substack{y\to x \\ nh \to t}}u_{h}(y, nh).
\end{cases}
\end{equation}
Then $\liminf_{*}u_{h}\leq \limsup^{*}u_{h}$. Furthermore, if $\widetilde{u}=\liminf_{*}u_{h}= \limsup^{*}u_{h}$, then $u_{h}\to\widetilde{u}$ locally uniformly as $h\to 0+$. 

Our main theorem is the following consistency result. 
\begin{theorem}\label{mainthm}
Let $u_{0}:\mathbb{R}^{N}\to \mathbb{R}$ be  uniformly continuous, $\Omega_{0}=\{x\in \mathbb{R}^{N}: u_{0}(x)>0\}$, $\Gamma_{0}=\{x\in \mathbb{R}^{N}: u_{0}(x)=0\}$, $\Omega_{t}=\{x\in \mathbb{R}^{N}: u(x, t)>0\}$ and $\Gamma_{t}=\{x \in \mathbb{R}^{N}: u(x, t)=0\}$, where $u$ is the unique viscosity solution of \eqref{weighted MC equation} satisfying the initial condition $u(\cdot, 0)=u_{0}(\cdot)$ in $\mathbb{R}^{N}$. Then 
\begin{equation*}
 \mathrm{lim\,inf}_{*}\, u_{h}=1 \,\ \text{in}\,\ \Omega_{t} \,\ \,\ \text{and}\,\ \,\,\  \mathrm{lim\,sup}^{*}\,u_{h}=-1 \,\ \text{in} \,\ (\Omega_{t}\cup \Gamma_{t})^{c}.
\end{equation*}
\end{theorem}
In particular, since $u_{h}:\mathbb{R}^{N}\to \{-1,1\}$, Theorem~\ref{mainthm} asserts that $u_{h}\to 1$ locally uniformly in $\Omega_{t}$ and $u_{h}\to -1$ locally uniformly in $(\Omega_{t} \cup \Gamma_{t})^{c}$ as $h\to 0+$. Namely, the scheme characterizes the evolution of the front $\Gamma_{0}\to \Gamma_{t}$ by assigning the values $1$ inside the region $\Omega_{t}$ and $-1$ outside the region $(\Omega_{t}\cup \Gamma_{t})^{c}$. Whether the regions where $u_{h}$ converges to $1$ and $-1$ are exactly the regions inside and outside the front, respectively, depends on whether the fattening phenomenon occurs or not (i.e., whether the front develops regions of positive measure where $u=0$; see \cite{Barles-Soner-Souganidis}). The answer is affirmative if and only if no fattening occurs. 
\begin{cor}
Let $u$ and $\Gamma_{t}$ be as in Theorem~\ref{mainthm}. Assume that
\[
\bigcup_{t\geq 0}\Gamma_{t}\times \{t\}=\partial \{(x,t): u(x,t)>0\}=\partial \{(x,t): u(x,t)<0\}.
\]
Then 
\[
\bigcup_{n \in \mathbb{N}}\Gamma^{h}_{nh}\times \{nh\} \to \bigcup_{t \geq 0}\Gamma_{t} \times \{t\}, \,\ \text{as}\,\ h \to 0+,
\]
in the Hausdorff distance. 
\end{cor}
Our strategy of the proof, as in \cite{Caffarelli-Souganidis}, is similar to the one in Barles and Georgelin \cite{Barles-Georgelin}, which relies on the general approach for proving convergence of numerical schemes by viscosity solution methods presented in \cite{Barles-Souganidis}. The novelty compared to previous results is that we consider anisotropic kernels that only have a prescribed power decay and establish the limit evolution in the presence of an external force that depends on both time and position, in contrast to \cite{Chambolle_Novaga_Ruffini_2017}, where only time-dependent forcing terms are considered. In particular, our kernels are not rotation-invariant, in contrast to those considered in \cite{Caffarelli-Souganidis}. They do not satisfy the assumptions (3.2), (3.3), (3.7) (if $\alpha=1$) and $(3.4)$ (if $\alpha \in [1,2)$) from \cite{Ishii-Pires-Souganidis}. Besides the consistency, we estimated the speed of the scheme applied to a ball (see Proposition~\ref{prop speed of balls}), illustrating the correctness of the chosen scales. We proved that our scheme is convexity preserving under suitable convexity assumptions on the forcing term (see Corollary~\ref{cor convexpresflowd}). Thus, the limit geometric evolution preserves convexity (see Corollary~\ref{cor convexflowlimitcontinuous}) under appropriate convexity assumptions on the external force $g$. We obtain the estimate (see Proposition~\ref{prop flowdistanceestim}) of the distance between two generalized evolutions with different external forces. Using this estimate, we provide a different proof of the uniqueness of the evolution of a convex bounded set than in \cite{Barles-Soner-Souganidis} (where the proof is based on the use of the comparison principle). In general, the inclusion principle and the uniqueness of evolutions follow from the scheme and the comparison principle (see Remark~\ref{remark about IP}). At a more technical point, under appropriate regularity assumptions, we established several stability results (see Theorems~\ref{thmconvergenceofviscsolstability1},~\ref{thmconvergenceofviscsolstability2} and Remark~\ref{rem fblalphato1plus}). In particular, the anisotropic fractional mean curvature operator defined in \cite{Chambolle_Novaga_Ruffini_2017} for $\alpha \in (0,1)$ multiplied by the factor $(1-\alpha)$ converges, as $\alpha \nearrow 1$, to our anisotropic mean curvature given for $\alpha=1$ (see Proposition~\ref{prop fromfractionaltoocalalpha=1}). Conversely, our anisotropic mean curvature multiplied by $(\alpha -1)$ converges as $\alpha \searrow 1$ to the anisotropic mean curvature that we obtain for the case where $\alpha=1$ (see Remark~\ref{rem ccmc}). In dimension 2, we characterized the norm $\norm$ by the mobility that we obtain in the limit, and vice versa (see Proposition~\ref{prop premise in 2d}). As a consequence, the unit ball of the mobility is as regular as the unit ball of the norm $\norm$, which has at least a Lipschitz regularity, since it is convex. In particular, the mobility can be a crystalline norm. The anisotropic mean curvature motion \eqref{weighted MC equation}, where $g=0$, is of a variational type and can be approximated, at a large scale, by eikonal-type equations modeling dislocations dynamics (see Theorem~\ref{thm variationaltypeorigin}). We also point out that if the
initial set is bounded and large enough, the corresponding front propagation, under appropriate assumptions, is asymptotically similar to the Wulff shape (see Theorem~\ref{theorem ltasymptotic behavior}). 

\section{Preliminaries}
\subsection{Conventions and Notation} \emph{Conventions}: in this paper, we say that a value is positive
if it is strictly greater than zero, and a value is nonnegative if it is greater than or equal to zero.
Euclidean spaces are endowed with the Euclidean inner product $\langle \cdot, \cdot \rangle$. We shall denote by $N$ an integer greater than or equal to $2$. The symbol $\norm$ will denote a norm on $\mathbb{R}^{N}$. A set will be called a domain whenever it is open and connected. The Hausdorff measures, which we shall use, coincide in terms of normalization with the appropriate outer Lebesgue measures. \\
\noindent \emph{Notation:} we denote the set of $N\times N$ symmetric real matrices by $\mathbb{M}^{N\times N}_{\mathrm{sym}}$. 
    We denote by $B_{r}(x)$, $\smash{\overline{B}}_{r}(x)$, and $\partial{B}_{r}(x)$, respectively, the open ball in $\mathbb{R}^{N}$, the closed ball in $\mathbb{R}^{N}$, and its boundary the $(N-1)$-sphere with center $x$ and radius $r$. 
    If the center is at the origin \(0\), we write $B_{r}$, $\overline{B}_{r}$ and $\partial B_{r}$ the corresponding balls and the $(N-1)$-sphere. 
    We shall denote by $\mathbb{S}^{N-1}$ and $\mathbb{S}^{N-2}$ the $(N-1)$-sphere and the $(N-2)$-sphere with center at the origin and radius $1$, respectively. 
    We denote by $\dist(x,A)$ and $\mathcal{H}^{l}(A)$, respectively, the Euclidean distance from $x\in \mathbb{R}^{N}$ to $A\subset \mathbb{R}^{N}$ and the $l$-dimensional Hausdorff measure of $A$. If $U\subset \mathbb{R}^{N}$ is Lebesgue measurable, then for $ p \in [1,+\infty)$, $L^{p}(U)$ will denote the space consisting of all real measurable functions on $U$ that are $p^{\mathrm{th}}$-power integrable on $U$. By $L^{1}_{\loc}(U)$ we denote the space of functions $u$ such that $u\in L^{1}(V)$ for all $V\Subset U$. We shall also write $\omega_{N-1}$ and $\omega_{N-2}$ instead of $\mathcal{H}^{N-1}(\mathbb{S}^{N-1})$ and $\mathcal{H}^{N-2}(\mathbb{S}^{N-2})$, respectively. The space of bounded continuous functions on $[0,+\infty)$ will be denoted by $C_{b}([0,+\infty))$. If $g \in C_{b}([0,+\infty))$, then $\|g\|_{\infty}$ will denote the supremum of $|g|$.  For each $p\in \mathbb{R}^{N}\setminus \{0\}$, $p^{\perp}$ will denote the orthogonal complement of  $\{p\}$, namely, $p^{\perp}=\{x \in \mathbb{R}^{N}: \langle x, p\rangle=0\}$. We use the standard notation for Sobolev spaces.       
\subsection{Definitions} We begin with the definition of a norm.
\begin{defn}\label{def of norm}
A norm on $\mathbb{R}^{N}$ is a function $\mathcal{N}: \mathbb{R}^{N} \to [0, +\infty)$ that satisfies the following properties
\begin{itemize}
\item $\norm(x)=0 \Leftrightarrow x=0$ ($\norm$ is positive definite);
\item $\norm(\lambda x)=|\lambda| \norm(x)$ for each $\lambda \in \mathbb{R}$ and $x \in \mathbb{R}^{N}$ ($\norm$ is positive 1-homogeneous and even);
\item $\norm(x+y)\leq \norm(x)+\norm(y)$ for each $x, y \in \mathbb{R}^{N}$ ($\norm$ is subadditive).
\end{itemize}
\end{defn}
It is well know that for each norm $\norm$ on $\mathbb{R}^{N}$ there exists a constant $C=C(\norm) \geq 1$ such that for each $x \in \mathbb{R}^{N}$,
\begin{equation}\label{equivofnormsoneucl}
C^{-1} |x| \le \norm(x) \le C |x|.
\end{equation}
Next, we recall the definition of a viscosity solution of \eqref{weighted MC equation}. We denote by $[F_{\alpha}]^{*}$ and $[F_{\alpha}]_{*}$ the upper and lower semicontinuous envelopes of $F_{\alpha}$, respectively. 
\begin{defn}\label{def ofsubsolwmc}
A locally bounded upper semicontinuous function $u:\mathbb{R}^{N}\times [0,+\infty)\to \mathbb{R}$ is a viscosity subsolution of \eqref{weighted MC equation} if for every $(x_{0}, t_{0})$ and every test function $\varphi \in C^{2}(\mathbb{R}^{N}\times (0,+\infty))$ such that $u-\varphi$ has a maximum at $(x_{0},t_{0})$, 
\begin{equation}\label{condition subsoldefn}
 \partial_{t} \varphi (x_{0},t_{0})\leq \mu_{\alpha}(D\varphi(x_{0},t_{0})) ([F_{\alpha}]^{*}(D^{2}\varphi(x_{0}, t_{0}), D\varphi(x_{0}, t_{0}))+g(x_{0},t_{0})|D\varphi(x_{0},t_{0})|). 
 \end{equation}
A  locally bounded lower semicontinuous function $u:\mathbb{R}^{N}\times [0,+\infty)\to \mathbb{R}$ is a viscosity supersolution of \eqref{weighted MC equation} if for every $(x_{0}, t_{0})$ and every test function $\varphi \in C^{2}(\mathbb{R}^{N}\times (0,+\infty))$ such that $u-\varphi$ has a minimum at $(x_{0},t_{0})$, 
\begin{equation}\label{condition supersoldefn} 
 \partial_{t} \varphi (x_{0}, t_{0})\geq \mu_{\alpha}(D\varphi(x_{0}, t_{0})) ([F_{\alpha}]_{*}(D^{2}\varphi(x_{0}, t_{0}), D\varphi(x_{0}, t_{0}))+g(x_{0}, t_{0})|D\varphi(x_{0}, t_{0})|).  \end{equation}
 A continuous function $u:\mathbb{R}^{N}\times [0,+\infty)\to \mathbb{R}$ is a viscosity solution of  \eqref{weighted MC equation}  if it is a subsolution and a supersolution of \eqref{weighted MC equation}.
\end{defn} 
For the theory of viscosity solutions, the reader may consult \cite{Crandall-Ishii-Lions-1992}. We shall use an equivalent definition which eliminates the difficulty related to the fact that $|D\varphi|$ may be equal to zero. 
\begin{theorem}\label{thm equivdefvs}
In Definition~\ref{def ofsubsolwmc}, the condition \eqref{condition subsoldefn} can be replaced by
\begin{equation}\label{def ofsubsolwmc new}
 \partial_{t} \varphi (x_{0}, t_{0})\leq \mu_{\alpha}(D\varphi(x_{0}, t_{0})) (F_{\alpha}(D^{2}\varphi(x_{0}, t_{0}), D\varphi(x_{0}, t_{0}))+g(x_{0}, t_{0})|D\varphi(x_{0}, t_{0})|) 
  \end{equation}
if $ |D\varphi(x_{0}, t_{0})|\neq 0$
 or 
 \begin{equation}\label{eq equivthm2sub}
  \partial_{t} \varphi (x_{0}, t_{0})\leq 0 \,\ \text{if}\,\ |D\varphi(x_{0}, t_{0})|= 0 \,\ \text{and}\,\ D^{2}\varphi(x_{0}, t_{0})=0,
\end{equation}
and the condition \eqref{condition supersoldefn} by
\begin{equation}\label{condition supersoldefn new}
 \partial_{t} \varphi (x_{0}, t_{0})\geq \mu_{\alpha}(D\varphi(x_{0}, t_{0})) (F_{\alpha}(D^{2}\varphi(x_{0}, t_{0}), D\varphi(x_{0}, t_{0}))+g(x_{0}, t_{0})|D\varphi(x_{0}, t_{0})|)
 \end{equation} if $|D\varphi(x_{0}, t_{0})|\neq 0$
 or 
 \begin{equation}\label{eq equivthm2sup}
  \partial_{t} \varphi (x_{0}, t_{0})\geq 0 \,\ \text{if}\,\ |D\varphi(x_{0}, t_{0})|= 0 \,\ \text{and}\,\ D^{2}\varphi(x_{0}, t_{0})=0,
\end{equation}
 and the definition remains equivalent. 
\end{theorem}
\begin{proof}
The reader may consult  the proof of \cite[Proposition~2.2]{Barles-Georgelin}, which adapts here without any difficulty.
\end{proof}
Next, we recall the definition of the generalized evolution corresponding to \eqref{weighted MC equation}. Let $\mathcal{F}$ and $\mathcal{O}$ denote, respectively, the collection of closed and open subsets of $\mathbb{R}^{N}$. Let $\Omega_{0}$ be an open subset of $\mathbb{R}^{N}$ and let $u_{0}:\mathbb{R}^{N}\to \mathbb{R}$  be a uniformly continuous function. Assume that $\Omega_{0}=\{x \in \mathbb{R}^{N}: u_{0}(x)>0\}$ and $\Gamma_{0}=\{x \in \mathbb{R}^{N}: u_{0}(x)=0\}$. Let $u$ be a unique viscosity solution of the equation \eqref{weighted MC equation} supplemented with the initial condition $u(\cdot, 0)=u_{0}(\cdot)$ in $\mathbb{R}^{N}$. We define $\Omega_{t}=\{x \in \mathbb{R}^{N}: u(x,t)>0\}$ and $\Gamma_{t}=\{x \in \mathbb{R}^{N}: u(x,t)=0\}$. We also define the maps  $X_{t}:\mathcal{F}\to \mathcal{F}$ and $O_{t}:\mathcal{O}\to \mathcal{O}$ by $X_{t}(\Omega_{0}\cup \Gamma_{0})=\Omega_{t}\cup \Gamma_{t}$ and $O_{t}(\Omega_{0})=\Omega_{t}$. 
\begin{defn}\label{def ofgenflow}
The collections $\{X_{t}\}_{t\geq 0}$ and $\{O_{t}\}_{t \geq 0}$ are called the generalized evolutions with normal velocity $v(-D(\frac{Du}{|Du|}), -\frac{Du}{|Du|}, x,t)=\mu_{\alpha}(-\frac{Du}{|Du|})(-\frac{1}{|Du|}F_{\alpha}(D^{2}u, Du)+g(x,t))$, where $\mu_{\alpha}$ is defined in \eqref{eq_defofmobilitygeneralapproachintro} and $F_{\alpha}$ is defined in \eqref{eq_defofanisotropicmeancurvoppdevmodalpha}.
\end{defn}
\section{Convergence of the discrete flows}
\subsection{The speed of balls}
In this subsection, we estimate the speed of the scheme applied to a ball. This provides us with a control on the (bounded) speed at which the balls decrease with the discrete flow.
\begin{prop}\label{prop speed of balls} Let $\alpha \in [1,2)$, $r>0$, $x_{0} \in \mathbb{R}^{N}$ and $g \in C_{b}([0,+\infty))$. There exist constants $A_{1}=A_{1}(\alpha, N, \norm)>0$, $A_{2}=A_{2}(\alpha, N, \norm)>0$ and $h_{0}=h_{0}(\alpha, r, \|g\|_{\infty}, N, \norm)>0$ such that for $h \in (0, h_{0})$ and $\tau=A_{1}/r + A_{2} \|g\|_{\infty}$, the following holds.
If $\alpha=1$, then 
\begin{equation}\label{eqestalpha1njg905j40990g045j0}
J_{h}*(\mathbbm{1}_{B_{r}(x_{0})}-\mathbbm{1}_{B^{c}_{r}(x_{0})})\geq \|g\|_{\infty}\sigma_{1}|\ln(\sigma_{1})|\,\ \text{in} \,\ B_{\smash{r-\tau h}}(x_{0})
\end{equation}
and 
\begin{equation}\label{eqextestc1k54gk045k40}
J_{h}*(\mathbbm{1}_{B_{r}(x_{0})}-\mathbbm{1}_{B^{c}_{r}(x_{0})})<- \|g\|_{\infty}\sigma_{1}|\ln(\sigma_{1})|\,\ \text{in} \,\ B_{\smash{r+\tau h}}(x_{0})\setminus \overline{B}_{r}(x_{0}).
\end{equation}
If $\alpha \in (1,2)$, then 
\begin{equation}\label{eqestalphaoirj095jg495j9g5j}
J_{h}*(\mathbbm{1}_{B_{r}(x_{0})}-\mathbbm{1}_{B^{c}_{r}(x_{0})})\geq \|g\|_{\infty} h^{\frac{1}{2}}\,\ \text{in} \,\ B_{\smash{r-\tau h}}(x_{0})
\end{equation}
and
\begin{equation}\label{eqestextballalpha12}
J_{h}*(\mathbbm{1}_{B_{r}(x_{0})}-\mathbbm{1}_{B^{c}_{r}(x_{0})}) <-\|g\|_{\infty} h^{\frac{1}{2}} \,\ \text{in} \,\ B_{\smash{r+\tau h}}(x_{0})\setminus \overline{B}_{r}(x_{0}).
\end{equation}
\end{prop}
\begin{rem}
In \cite{Chambolle_Novaga_Ruffini_2017}, for each $\alpha \in (0,1)$ and for each set $E\subset \mathbb{R}^{N}$ of class $C^{1,1}$, the authors define the anisotropic fractional mean curvature at $x \in \partial E$ by
\begin{equation}\label{nonlocalcurvature}
-\kappa_{\alpha}(x,E)=\int_{\mathbb{R}^{N}}\frac{\mathbbm{1}_{E}(y)-\mathbbm{1}_{E^{c}}(y)}{\norm(y-x)^{N+\alpha}} \diff y,
\end{equation}
where the role of the  ``$-$'' sign is to ensure that convex sets have nonnegative curvature (for a rigorous explanation of this definition, the reader may consult \cite[Subsection~2.2]{Imbert2009level} and, in particular, \cite[Lemma~1]{Imbert2009level}). We point out that this definition of the anisotropic mean curvature is no longer valid in the case where $\alpha \in [1,2)$, since \cite[Lemma~1]{Imbert2009level} is false in this case (see \cite[Remark~1]{Imbert2009level}), and the \emph{anisotropic} $\alpha$-stable Lévy measure $\frac{\diff z}{\norm(z)^{N+\alpha}}$ on $\mathbb{R}^{N}$ does not satisfy the assumption (A3) in \cite{Imbert2009level}. 
\end{rem}
\begin{proof}
To lighten the notation, denote $e=(1,0, \dotsc, 0) \in\mathbb{R}^{N}$ and $B=B_{r}(re)$. Up to a translation, we assume that $x_{0}=0$. A little later in Corollary~\ref{cor convexpresflowd} we shall prove that the superlevel sets of the function $J_{h}*(\mathbbm{1}_{B_{r}}-\mathbbm{1}_{B^{c}_{r}})$ are convex. Thus, to obtain the desired lower bound for $J_{h}*(\mathbbm{1}_{B_{r}}-\mathbbm{1}_{B^{c}_{r}})$ in the ball $\overline{B}_{r-t}$ for fairly small $t \in (0,1)$, it suffices to obtain the same estimate for $J_{h}*(\mathbbm{1}_{B_{r}}-\mathbbm{1}_{B^{c}_{r}})$ on $\partial B_{r-t}$.  Inasmuch as $\norm$ is even, $\norm(x)=\norm(-x)$  for each $x \in \mathbb{R}^{N}$, and in view of \eqref{equivofnormsoneucl}, to deduce the desired estimate for $J_{h}*(\mathbbm{1}_{B_{r}}-\mathbbm{1}_{B^{c}_{r}})$ on $\partial B_{r-t}$, it is enough to estimate $J_{h}*(\mathbbm{1}_{B}-\mathbbm{1}_{B^{c}})$ at the point $te$. According to \eqref{defofkernelJh},
\begin{equation}\label{formulaforkernelJh}
J_{h}(y)=\frac{\sigma_{\alpha}(h)}{\sigma_{\alpha}(h)^{\frac{N+\alpha}{\alpha}}+\norm(y)^{N+\alpha}}
\end{equation}
for each $y\in \mathbb{R}^{N}$. Since $\norm(y)=\norm(-y)$ for each $y\in \mathbb{R}^{N}$, 
\begin{equation}\label{eq_compsymnorm1}
\int_{B}\frac{\diff y}{\sigma_{\alpha}(h)^{\frac{N+\alpha}{\alpha}}+\norm(y)^{N+\alpha}}=\int_{-B}\frac{\diff y}{\sigma_{\alpha}(h)^{\frac{N+\alpha}{\alpha}}+\norm(y)^{N+\alpha}}.
\end{equation}
Using \eqref{formulaforkernelJh} and \eqref{eq_compsymnorm1}, we have
\begin{equation}\label{eqncinfi2fj0j349fj39huh2}
\begin{split}
\sigma_{\alpha}(h)^{-1}J_{h}*(\mathbbm{1}_{B}-\mathbbm{1}_{B^{c}})(0) & =-\int_{\mathbb{R}^{N}\setminus (-B \cup B)}\frac{\diff y}{\sigma_{\alpha}(h)^{\frac{N+\alpha}{\alpha}}+\norm(y)^{N+\alpha}}\\ & \geq- \int_{B_{\min\{1,r\}}\setminus (-B\cup B)}\frac{\diff y}{\sigma_{\alpha}(h)^{\frac{N+\alpha}{\alpha}}+\norm(y)^{N+\alpha}}-\int_{B^{c}_{\min\{1,r\}}}\frac{\diff y}{\norm(y)^{N+\alpha}} \\
&=- \int_{B_{\min\{1,r\}}\setminus (-B\cup B)}\frac{\diff y}{\sigma_{\alpha}(h)^{\frac{N+\alpha}{\alpha}}+\norm(y)^{N+\alpha}}-\frac{C^{N+\alpha}\omega_{N-1}}{\alpha \min\{1,r\}^{\alpha}},
\end{split}
\end{equation}
where $C=C(\norm)\geq 1$ is the constant coming from \eqref{equivofnormsoneucl}. To lighten the notation, hereinafter in this proof, we shall simply write  $\sigma$ instead of $\sigma_{\alpha}(h)$.  We consider the next cases. \\
\emph{Case 1:} $\alpha=1$. Then, defining $C_{\varrho}=(B_{\min\{1,r\}}\setminus (-B\cup B)) \cap \{x\in \mathbb{R}^{N}: \dist(x, \{t e: t \in \mathbb{R}\})=\varrho\}$, using the coarea formula (see \cite[Theorem~3.2.22~(3)]{Federer}) and \eqref{equivofnormsoneucl}, we obtain the following 
\begin{equation}\label{eqest007new3890h5083}
\begin{split}
\int_{B_{\min\{1,r\}}\setminus (-B\cup B)}\frac{\diff y}{\sigma^{N+1}+\norm(y)^{N+1}} & \leq \int_{0}^{\min\{1,r\}}\diff \varrho \int_{C_{\varrho}}\frac{\diff \mathcal{H}^{N-1}(y)}{\sigma^{N+1}+\norm(y)^{N+1}}\\
& \leq \frac{2}{r}C^{N+1} \omega_{N-2} \int_{0}^{\min\{1,r\}}\frac{\varrho^{N}\diff \varrho}{(C\sigma)^{N+1}+\varrho^{N+1}}\\
& \leq \frac{2C^{N+1}}{r(N+1)} \omega_{N-2} \ln\left(\frac{(C\sigma)^{N+1}+\min\{1,r\}^{N+1}}{(C\sigma)^{N+1}}\right)\\
& \leq \frac{3}{r}C^{N+1} \omega_{N-2} |\ln(\sigma)|
\end{split}
\end{equation}
provided that $\sigma> 0$ is small enough depending on $r$, $N$ and $C$, where $C=C(\norm)\geq 1$ is the constant coming from \eqref{equivofnormsoneucl}. Combining \eqref{eqncinfi2fj0j349fj39huh2} and \eqref{eqest007new3890h5083}, we get
\begin{equation}\label{eqestimation29042ijg05j340}
\sigma^{-1} J_{h}*(\mathbbm{1}_{B}-\mathbbm{1}_{B^{c}})(0) \geq -\frac{3}{r} C^{N+1}\omega_{N-2}|\ln(\sigma)|-\frac{C^{N+1}\omega_{N-1}}{\min\{1,r\}}
\end{equation}
provided that $\sigma> 0$ is small enough depending on $r$, $N$ and $C$. Next, we want to estimate $DJ_{h}*(\mathbbm{1}_{B}-\mathbbm{1}_{B^{c}})$ at the point $te$, where $|t|\ \leq f(\sigma)$ and $f(\sigma)$ is large enough with respect to $h$. Taking into account \eqref{eqestimation29042ijg05j340} and carefully performing preliminary computations, one can conclude that it is enough to consider $f(\sigma)=\sigma^{\theta}|\ln(\sigma)|$ for some $\theta \in (1,2)$. Let us fix $t=s \sigma^{\theta}|\ln(\sigma)|$, where $s \in [-1,1]$. We have 
\begin{equation*}
G:=\langle DJ_{h}*(\mathbbm{1}_{B}-\mathbbm{1}_{B^{c}})(te), \sigma^{-1} e\rangle=-2\int_{\partial B}\frac{\langle \nu, e \rangle \diff \mathcal{H}^{N-1}(x)}{\sigma^{N+1}+\norm(te-x)^{N+1}},
\end{equation*}
where $\nu: \partial B \to \mathbb{S}^{N-1}$ stands for the outward pointing unit normal vector field to $\partial B$. Performing the change of variables $x=\sigma y$ and denoting the ball $B_{\sigma^{-1}r}(\sigma^{-1} re)$ by $\sigma^{-1}B$, we deduce that
\begin{equation*}
\begin{split}
G&=-2 \int_{\partial(\sigma^{-1} B)} \frac{\sigma^{N-1} \langle \nu, e\rangle \diff \mathcal{H}^{N-1}(y)}{\sigma^{N+1}(1+\norm(s \sigma^{\theta-1}|\ln(\sigma)|e-y)^{N+1})} \\
&= -\frac{2}{\sigma^{2}} \int_{\partial(\sigma^{-1} B)} \frac{\langle \nu, e\rangle \diff \mathcal{H}^{N-1}(y)}{1+\norm(s \sigma^{\theta-1}|\ln(\sigma)|e-y)^{N+1}},
\end{split}
\end{equation*}
where $\nu$ is the outward pointing unit normal vector field to $\partial (\sigma^{-1}B)$. Let $R\geq 1$ and $B^{\sigma}$ be the ball with center at the origin and radius $R\sigma^{-\frac{N-1}{N+1}}$. If $y \in (B^{\sigma})^{c}$, then 
\begin{equation}\label{eqtrianglularinequality1}
|s \sigma^{\theta-1}|\ln(\sigma)|e-y| \geq R\sigma^{-\frac{N-1}{N+1}}-\sigma^{\theta-1}|\ln(\sigma)| \geq \frac{R}{2}\sigma^{-\frac{N-1}{N+1}}
\end{equation}
provided that $\sigma>0$ is small enough depending on $\theta$, $R$ and $N$. Thus, using \eqref{equivofnormsoneucl} and \eqref{eqtrianglularinequality1}, we obtain the following estimate 
\begin{equation*}
\begin{split}
\left|2\int_{\partial(\sigma^{-1} B)\setminus B^{\sigma}}  \frac{\langle \nu, e\rangle \diff \mathcal{H}^{N-1}(y)}{1+\norm(s \sigma^{\theta-1}|\ln(\sigma)|e-y)^{N+1}}\right| \leq \frac{2\sigma^{-(N-1)}r^{N-1}\omega_{N-1}}{1+ (\frac{R}{2C} \sigma ^{-\frac{N-1}{N+1}})^{N+1}}  =\frac{2r^{N-1} \omega_{N-1}}{\sigma^{N-1}+(\frac{R}{2C})^{N+1}},
\end{split}
\end{equation*}
which can be made arbitrarily small by choosing $R$ large enough depending on $r$, $N$ and $C$. On the other hand, if $\sigma>0$ is small enough (depending on $r$, $N$ and $C$) and $y \in \partial(\sigma^{-1}B) \cap B^{\sigma}$, then we can assume that $\sigma^{-1}r$ is large enough with respect to $R\sigma^{-\frac{N-1}{N+1}}$ so that $\langle \sigma, e\rangle \leq -1/2$. This implies that 
\begin{equation*}
\begin{split}
-2 \int_{\partial(\sigma^{-1} B)\cap B^{\sigma}}  \frac{\langle \nu, e\rangle  \diff \mathcal{H}^{N-1}(y)}{1+\norm(s \sigma^{\theta-1}|\ln(\sigma)|e-y)^{N+1}}& \geq  \int_{\partial(\sigma^{-1} B)\cap B^{\sigma}} \frac{\diff \mathcal{H}^{N-1}(y)}{1+\norm(s \sigma^{\theta-1}|\ln(\sigma)| e-y)^{N+1}}\\
& \to \int_{e^{\perp}} \frac{\diff \mathcal{H}^{N-1}(y)}{1+\norm(y)^{N+1}}=:2 \eta >0
\end{split}
\end{equation*}
as $h\to 0+$. Altogether, we have proved that if $R$ is large enough, there exists $h_{0}>0$ depending on $r$, $\theta$, $N$, $\norm$ and $C$ such that if $h \in (0,h_{0})$ (recall that $h=\sigma^{2}|\ln(\sigma)|$), then \eqref{eqestimation29042ijg05j340} holds and
\begin{equation}\label{eqmi0ijfi4gji45gi945j0}
\langle D J_{h}*(\mathbbm{1}_{B}-\mathbbm{1}_{B^{c}})(te), e\rangle \geq \frac{\eta}{\sigma}
\end{equation}
whenever $|t| \leq \sigma^{\theta} |\ln(\sigma)|$. Choosing $\theta=3/2$, $A_{1}=4C^{N+1}\omega_{N-2}/\eta$, $A_{2}=1/\eta$,
\[
\tau=\frac{A_{1}}{r}+A_{2}\|g\|_{\infty},\,\ \,\  t=\tau \sigma^{2}|\ln(\sigma)|
\]
and (possibly) reducing $h_{0}$, we have that $|t|\leq \sigma^{\theta} |\ln(\sigma)|$ and 
\begin{equation}\label{estder1stk39kg45igj}
J_{h}*(\mathbbm{1}_{B}-\mathbbm{1}_{B^{c}})(te)=J_{h}*(\mathbbm{1}_{B}-\mathbbm{1}_{B^{c}})(0)+\langle D J_{h}*(\mathbbm{1}_{B}-\mathbbm{1}_{B^{c}})(te), te\rangle+o(t)\geq \|g\|_{\infty} \sigma |\ln(\sigma)|
\end{equation}
if $h\in (0, h_{0})$, where we have used Taylor's expansion for $J_{h}*(\mathbbm{1}_{B}-\mathbbm{1}_{B^{c}})(0)$. Notice that $\tau$ and $h_{0}$ depend on $r$, $\|g\|_{\infty}$, $N$, $\norm$ and $C$. Inasmuch as $C$ depends only on $\norm$, we can assume that $\tau$ and $h_{0}$ depend only on $r$, $\|g\|_{\infty}$, $N$ and $\norm$. This implies \eqref{eqestalpha1njg905j40990g045j0}. On the other hand, taking into account that $J_{h}*(\mathbbm{1}_{B}-\mathbbm{1}_{B^{c}})(0)<0$ (see \eqref{eqncinfi2fj0j349fj39huh2}) and \eqref{eqmi0ijfi4gji45gi945j0}, we have
\[
J_{h}*(\mathbbm{1}_{B}-\mathbbm{1}_{B^{c}})(-te)=J_{h}*(\mathbbm{1}_{B}-\mathbbm{1}_{B^{c}})(0)-\langle D J_{h}*(\mathbbm{1}_{B}-\mathbbm{1}_{B^{c}})(-te),  te\rangle +o(t) < -\|g\|_{\infty}\sigma |\ln(\sigma)|.
\]
Thus, \eqref{eqextestc1k54gk045k40} is satisfied. 
\\
\emph{Case 2:} $\alpha \in (1,2)$. Let $C_{\varrho}$ be defined as in Case 1. We recall that $\sigma=h^{\frac{\alpha}{2}}$ and $C=C(\norm)\geq 1$ is the constant coming from \eqref{equivofnormsoneucl}. Using the coarea formula (see \cite[Theorem~3.2.22~(3)]{Federer}), \eqref{equivofnormsoneucl} and the facts that $(C^{2}h+\varrho^{2})^{\frac{N+\alpha}{2}}\leq 2^{\frac{N+\alpha}{2}-1}((C^{2}h)^{\frac{N+\alpha}{2}}+\varrho^{N+\alpha})$ (by Jensen's inequality), $C\geq 1$ and also $(h^{\frac{1}{2}}+\varrho)^{2}\leq 2(h+\varrho^{2})$ (by Jensen's inequality), we deduce that  
\begin{equation}\label{eqest007new3890h5083c2}
\begin{split}
\int_{B_{\min\{1,r\}}\setminus (-B\cup B)}\frac{\diff y}{\sigma^{\frac{N+\alpha}{\alpha}}+\norm(y)^{N+\alpha}} & \leq \int_{0}^{\min\{1,r\}}\diff \varrho \int_{C_{\varrho}}\frac{\diff \mathcal{H}^{N-1}(y)}{h^{\frac{N+\alpha}{2}}+\norm(y)^{N+\alpha}}\\
& \leq \frac{2}{r}C^{N+\alpha} \omega_{N-2} \int_{0}^{\min\{1,r\}}\frac{\varrho^{N}}{(C^{2}h)^{\frac{N+\alpha}{2}}+\varrho^{N+\alpha}}\diff \varrho\\
& \leq \frac{2^{\frac{N+\alpha}{2}}}{r}C^{N+\alpha} \omega_{N-2} \int^{1}_{0}\frac{\varrho^{N}}{(h+\varrho^{2})^{\frac{N+\alpha}{2}}} \diff \varrho \\
& \leq \frac{2^{\frac{N+\alpha}{2}}}{r}C^{N+\alpha} \omega_{N-2} \int^{1}_{0}(h+\varrho^{2})^{-\frac{\alpha}{2}}\diff \varrho \\
& \leq \frac{2^{\frac{N}{2}+\alpha}}{r}C^{N+\alpha} \omega_{N-2} \int^{1}_{0}(h^{\frac{1}{2}}+\varrho)^{-\alpha}\diff \varrho\\
&=\frac{2^{\frac{N}{2}+\alpha}}{r}C^{N+\alpha} \omega_{N-2}\frac{h^{\frac{1-\alpha}{2}}-(1+h^{\frac{1}{2}})^{1-\alpha}}{\alpha-1}.
\end{split}
\end{equation}
Combining \eqref{eqncinfi2fj0j349fj39huh2} and \eqref{eqest007new3890h5083c2}, we get
\begin{equation}\label{eqestimation29042ijg05j340c2}
\begin{split}
J_{h}*(\mathbbm{1}_{B}-\mathbbm{1}_{B^{c}})(0) & \geq -\frac{2^{\frac{N}{2}+\alpha}}{r}C^{N+\alpha}\omega_{N-2} \frac{h^{\frac{\alpha}{2}}((1+h^{\frac{1}{2}})^{\alpha-1}-h^{\frac{\alpha-1}{2}})}{(\alpha-1)(h^{\frac{1}{2}}+h)^{\alpha-1}}  -\frac{C^{N+\alpha}\omega_{N-1}h^{\frac{\alpha}{2}}}{\alpha \min\{1,r\}^{\alpha}}\\
& \geq -\frac{2^{\frac{N}{2}+\alpha}C^{N+\alpha}\omega_{N-2}}{(\alpha-1)r}h^{\frac{1}{2}}-\frac{C^{N+\alpha}\omega_{N-1}h^{\frac{\alpha}{2}}}{\alpha \min\{1,r\}^{\alpha}},
\end{split}
\end{equation}
where we have used that $(1+h^{\frac{1}{2}})^{\alpha-1}\leq 1+h^{\frac{\alpha-1}{2}}$ (by subadditivity, since $(\alpha-1) \in (0,1)$). It is worth noting that if $\alpha$ is close enough to $1$, then \[\frac{h^{\frac{\alpha}{2}}(h^{\frac{1-\alpha}{2}}-(1+h^{\frac{1}{2}})^{1-\alpha})}{\alpha-1} \approx h^{\frac{\alpha}{2}}|\ln(h^{\frac{1}{2}})|=\sigma|\ln(\sigma^{\frac{1}{\alpha}})| \approx \sigma |\ln(\sigma)|\] (this allows us to compare \eqref{eqestimation29042ijg05j340c2} with \eqref{eqestimation29042ijg05j340}  when $\alpha>1$ is close enough to $1$). Next, we want to estimate $DJ_{h}*(\mathbbm{1}_{B}-\mathbbm{1}_{B^{c}})$ at the point $te$, where $|t|\ \leq f(h)$ and $f(h)$ is large enough with respect to $h$. In view of \eqref{eqestimation29042ijg05j340c2} and our  preliminary computations, we conclude that it is enough to consider $f(h)=h^{\frac{2+\alpha}{4}}$ provided that $h>0$ is small enough. Let us fix $t=s h^{\frac{2+\alpha}{4}}$, where $s \in [-1,1]$. We have 
\begin{equation*}
G:=\langle DJ_{h}*(\mathbbm{1}_{B}-\mathbbm{1}_{B^{c}})(te), \sigma^{-1} e\rangle=-2\int_{\partial B}\frac{\langle \nu, e \rangle \diff \mathcal{H}^{N-1}(x)}{h^{\frac{N+\alpha}{2}}+\norm(te-x)^{N+\alpha}}.
\end{equation*}
Performing the change of variables $x=h^{\frac{1}{2}} y$ and denoting the ball $B_{h^{-\frac{1}{2}}r}(h^{-\frac{1}{2}} re)$ by $h^{-\frac{1}{2}}B$, we deduce that
\begin{equation*}
\begin{split}
G&=-2 \int_{\partial(h^{-\frac{1}{2}} B)} \frac{h^{\frac{N-1}{2}} \langle \nu, e\rangle \diff \mathcal{H}^{N-1}(y)}{h^{\frac{N+\alpha}{2}}(1+\norm(s h^{\frac{\alpha}{4}}e-y)^{N+\alpha})} \\
&= -\frac{2}{h^{\frac{1+\alpha}{2}}} \int_{\partial(h^{-\frac{1}{2}} B)} \frac{\langle \nu, e\rangle \diff \mathcal{H}^{N-1}(y)}{1+\norm(s h^{\frac{\alpha}{4}}e-y)^{N+\alpha}}.
\end{split}
\end{equation*}
Let $R\geq 1$ and $B^{h}$ be the ball with center at the origin and radius $Rh^{-\frac{N-1}{2(N+\alpha)}}$. If $y \in (B^{h})^{c}$, then 
\begin{equation}\label{eqtrianglularinequality1c2}
|s h^{\frac{\alpha}{4}}e-y| \geq Rh^{-\frac{N-1}{2(N+\alpha)}}-h^{\frac{\alpha}{4}} \geq \frac{R}{2}h^{-\frac{N-1}{2(N+\alpha)}}
\end{equation}
provided that $h>0$ is small enough depending on $\alpha$, $R$ and $N$. Thus, using \eqref{equivofnormsoneucl} and \eqref{eqtrianglularinequality1c2}, we obtain the following estimate 
\begin{equation*}
\begin{split}
\left|2\int_{\partial(h^{-\frac{1}{2}} B)\setminus B^{h}}  \frac{\langle \nu, e\rangle \diff \mathcal{H}^{N-1}(y)}{1+\norm(s h^{\frac{\alpha}{4}}e-y)^{N+\alpha}} \right| \leq \frac{2h^{-\frac{N-1}{2}}r^{N-1}\omega_{N-1}}{1+ (\frac{R}{2C} h ^{-\frac{N-1}{2(N+\alpha)}})^{N+\alpha}}=\frac{2r^{N-1} \omega_{N-1}}{h^{\frac{N-1}{2}}+(\frac{R}{2C})^{N+\alpha}},
\end{split}
\end{equation*}
which can be made arbitrarily small by choosing $R$ large enough depending on $\alpha$, $r$, $N$ and $C$. On the other hand, if $h>0$ is small enough (depending on $\alpha$, $r$, $N$ and $C$) and $y \in \partial(h^{-\frac{1}{2}}B) \cap B^{h}$, then we can assume that $h^{-\frac{1}{2}}r$ is large enough with respect to $Rh^{-\frac{N-1}{2(N+\alpha)}}$ so that $\langle \sigma, e\rangle \leq -1/2$. This implies that 
\begin{equation*}
\begin{split}
-2 \int_{\partial(h^{-\frac{1}{2}} B)\cap B^{h}}  \frac{\langle \nu, e\rangle  \diff \mathcal{H}^{N-1}(y) }{1+\norm(s h^{\frac{\alpha}{4}}e-y)^{N+\alpha}}& \geq  \int_{\partial(h^{-\frac{1}{2}} B)\cap B^{h}} \frac{\diff \mathcal{H}^{N-1}(y)}{1+\norm(s h^{\frac{\alpha}{4}}e-y)^{N+\alpha}}\\
& \to \int_{e^{\perp}} \frac{\diff \mathcal{H}^{N-1}(y)}{1+\norm(y)^{N+\alpha}}=:2 \eta >0
\end{split}
\end{equation*}
as $h\to 0+$. Altogether, if $R$ is large enough, there exists $h_{0}=h_{0}(\alpha, r, N, C)>0$ such that if $h \in (0,h_{0})$, then \eqref{eqestimation29042ijg05j340c2} holds and
\begin{equation}\label{eqi5jgi45ig549j90j}
\langle D J_{h}*(\mathbbm{1}_{B}-\mathbbm{1}_{B^{c}})(te), e\rangle \geq \eta h^{-\frac{1}{2}}
\end{equation}
whenever $|t| \leq h^{\frac{2+\alpha}{4}}$. Setting $A_{1}=2^{\frac{N}{2}+\alpha+1}C^{N+\alpha}\omega_{N-2}/((\alpha-1)\eta)$, $A_{2}=1/\eta$,
\[
\tau=\frac{A_{1}}{r}+A_{2}\|g\|_{\infty},\,\ \,\  t=\tau h
\]
and (possibly) reducing $h_{0}$, we have that $|t|\leq h^{\frac{2+\alpha}{4}}$ and 
\begin{equation*}
J_{h}*(\mathbbm{1}_{B}-\mathbbm{1}_{B^{c}})(te)=J_{h}*(\mathbbm{1}_{B}-\mathbbm{1}_{B^{c}})(0)+\langle D J_{h}*(\mathbbm{1}_{B}-\mathbbm{1}_{B^{c}})(te), te\rangle +o(t)\geq \|g\|_{\infty} h^{\frac{1}{2}}
\end{equation*}
if $h \in (0,h_{0})$, where we have used Taylor's expansion of $J_{h}*(\mathbbm{1}_{B}-\mathbbm{1}_{B^{c}})(0)$ at $te$. Let us point out that $\tau$ and $h_{0}$ depend on $\alpha$, $r$, $\|g\|_{\infty}$, $N$, $\norm$ and $C$. Inasmuch as $C$ depends only on $\norm$, we can assume that $\tau$ and $h_{0}$ depend only on $\alpha$, $r$, $\|g\|_{\infty}$, $N$ and $\norm$. We can also assume that $A_{1}$ and $A_{2}$ depend only on $\alpha$, $N$ and $\norm$. This yields \eqref{eqestalphaoirj095jg495j9g5j}. Furthermore, observing that  $J_{h}*(\mathbbm{1}_{B}-\mathbbm{1}_{B^{c}})(0)<0$ (see \eqref{eqncinfi2fj0j349fj39huh2}) and taking into account \eqref{eqi5jgi45ig549j90j}, we have
\[
J_{h}*(\mathbbm{1}_{B}-\mathbbm{1}_{B^{c}})(-te)=J_{h}*(\mathbbm{1}_{B}-\mathbbm{1}_{B^{c}})(0)-\langle D J_{h}*(\mathbbm{1}_{B}-\mathbbm{1}_{B^{c}})(-te), te\rangle + o(t)< -\|g\|_{\infty}h^{\frac{1}{2}},
\]
and hence \eqref{eqestextballalpha12} is satisfied. This completes our proof of Proposition~\ref{prop speed of balls}.
\end{proof}
\begin{cor} Let $\alpha \in [1,2)$, $r>0$, $x_{0} \in \mathbb{R}^{N}$, $g \in C_{b}([0,+\infty))$ and for each $x \in \mathbb{R}^{N}$, $u_{h}(x, 0)=\mathbbm{1}_{B_{r}(x_{0})}(x)-\mathbbm{1}_{B^{c}_{r}(x_{0})}(x)$. Let $A_{1}=A_{1}(\alpha, N, \norm)>0$, $A_{2}=A_{2}(\alpha, N, \norm)>0$ and $h_{0}=h_{0}(\alpha, r, \|g\|_{\infty}, N, \norm)>0$ be the constants of Proposition~\ref{prop speed of balls}.  If $h \in (0, h_{0})$, then \[u_{h}(\cdot, nh) \geq \mathbbm{1}_{B_{r/2}(x_{0})}(\cdot)-\mathbbm{1}_{B^{c}_{r/2}(x_{0})}(\cdot)\,\ \text{in} \,\ \mathbb{R}^{N}\] as long as 
\[
nh\leq \frac{r^{2}}{2 A_{1}+ A_{2}r\|g\|_{\infty}}.
\]
In particular, if $\varepsilon>0$ is small enough and $r \in (\varepsilon, 2\varepsilon)$, then $nh \leq r^{2}/(2A_{1}+A_{2} \varepsilon \|g\|_{\infty})$.
\end{cor}
 We shall denote by $\sign^{*}$ and $\sign_{*}$ the upper semicontinuous envelope and the lower semicontinuous envelope, respectively, of the $\sign$ function in $\mathbb{R}$, namely,
\[
\sign^{*}(t)=\begin{cases}
1 \,\ &\text{if}\,\ t\geq 0, \\ 
-1 \,\ &\text{otherwise}
\end{cases}
\]
and 
\[
\sign_{*}(t)=\begin{cases}
1 \,\ &\text{if} \,\ t>0,\\
-1 \,\ &\text{otherwise}.
\end{cases}
\]
The key ingredient in the proof of Theorem~\ref{mainthm} is the following 
\begin{prop}\label{prop visc subandsupersol}
The functions $\limsup^{*}u_{h}$ and $\liminf_{*}u_{h}$ defined in \eqref{eq_lowuppsemlimits} are, respectively, a viscosity subsolution and a viscosity supersolution of 
\eqref{weighted MC equation}.
\end{prop}
\begin{proof}
We only prove that $\limsup^{*}u_{h}$ is a subsolution, since the proof that $\liminf_{*}u_{h}$ is a supersolution follows similarly. Let $\varphi \in C^{2}(\mathbb{R}^{N}\times (0,+\infty))$. Assume that $(x_{0}, t_{0}) \in \mathbb{R}^{N}\times (0,+\infty)$ is a strict global maximum point of $\limsup^{*}u_{h}-\varphi$. Without loss of generality, we assume that 
\begin{equation}\label{eq_io304ij503jtcoerc}
\lim_{|x|+|t|\to +\infty} \varphi(x,t)=+\infty,
\end{equation}
which will eliminate technical difficulties coming from the unboundedness of the domain. Indeed, we can replace $\varphi$ by the function $\varphi_{\varepsilon}(x,t)=\varphi(x,t)+\varepsilon(|x-x_{0}|^{2}+|t-t_{0}|^{2})$ and prove the main inequality for $\varphi_{\varepsilon}$, which in the limit, as $\varepsilon\to 0+$, yields the same inequality for $\varphi$. 

If $\limsup^{*}u_{h}(x_{0}, t_{0})=-1$, then, since $\limsup^{*}u_{h}$ is upper semicontinuous and takes values in $\{-1,1\}$, $\limsup^{*}u_{h}=-1$ in a neighborhood of $(x_{0},t_{0})$, and hence
 $|D\varphi(x_{0},t_{0})|=0$ and $\partial_{t}\varphi(x_{0}, t_{0})=0$, which yields \eqref{condition subsoldefn}. Similarly, if $(x_{0}, t_{0})$ belongs to the interior of the set $\{\limsup^{*}u_{h}=1\}$, \eqref{condition subsoldefn} is satisfied at $(x_{0}, t_{0})$. Thus, assume that $(x_{0}, t_{0})$ belongs to the boundary of the set $\{\limsup^{*}u_{h}=1\}$. 
 
Notice that $\limsup^{*}u_{h}=\limsup^{*}u_{h}^{*}$, and hence we can replace $u_{h}$ with $u^{*}_{h}$, which is upper semicontinuous. In view of \eqref{eq_io304ij503jtcoerc} and \cite[Lemma~A.3]{Barles_Perthame_87}, there exists a subsequence $(x_{h}, n_{h}h)$ converging to $(x_{0}, t_{0})$ such that
\[
u^{*}_{h}(x_{h}, n_{h}h)-\varphi(x_{h}, n_{h}h)=\max_{\mathbb{R}^{N}\times \mathbb{N}}(u^{*}_{h}-\varphi)
\]
and 
\[
u^{*}_{h}(x_{h}, n_{h}h)\to 1,
\]
where we have used that $u^{*}_{h}$ is upper semicontinuous. The latter, together with the fact that $u^{*}_{h}$ takes values in $\{-1,1\}$, implies that $u^{*}_{h}(x_{h}, n_{h}h)=1$ for $h$ small enough. Furthermore, for such $h$, since $(x_{h}, n_{h}h)$ is a maximum point, 
\begin{equation}\label{eq_dychotomie}
u^{*}_{h}(x, nh)\leq 1+\varphi(x, nh)-\varphi(x_{h}, n_{h}h)
\end{equation}
for all $x \in \mathbb{R}^{N}$ and $n\in \mathbb{N}$. If $u^{*}_{h}(x,nh)=1$, then \eqref{eq_dychotomie} yields $\varphi(x,nh)-\varphi(x_{h},n_{h}h)\geq 0$ and hence
\begin{equation}\label{eq_estimustarm3ofj43f9}
u^{*}_{h}(x,nh)\leq \sign^{*}(\varphi(x,nh)-\varphi(x_{h}, n_{h}h)).
\end{equation}
Clearly, the above inequality also holds when $u^{*}_{h}(x, nh)=-1$. By definition, for each $x \in \mathbb{R}^{N}$,
\[
u_{h}(x, n_{h}h)= \sign(J_{h}*u_{h}(\cdot, (n_{h}-1)h) + g(\cdot, (n_{h}-1)h) \beta(\alpha, h))(x)
\]
(see \eqref{formulaforkernelJh} for the definition of $J_{h}$) and hence
\[
u_{h}(x, n_{h}h)\leq \sign^{*}(J_{h}*u^{*}_{h}(\cdot, (n_{h}-1)h)+ g(\cdot, (n_{h}-1)h) \beta(\alpha, h))(x).
\]
Since the right hand-side in the above inequality is upper semicontinuous, we observe that
\[
u^{*}_{h}(x, n_{h}h) \leq \sign^{*}(J_{h}*u^{*}_{h}(\cdot, (n_{h}-1)h)+ g(\cdot, (n_{h}-1)h) \beta(\alpha, h))(x).
\]
Using this property with $x=x_{h}$, \eqref{eq_estimustarm3ofj43f9} and the monotonicity of the $\sign^{*}$ function, we have 
\begin{align*}
1=u^{*}_{h}(x_{h}, n_{h}h)&\leq \sign^{*}(J_{h}*u^{*}_{h}(\cdot, (n_{h}-1)h)+ g(\cdot, (n_{h}-1)h) \beta(\alpha, h))(x_{h})\\
& \leq \sign^{*}(J_{h}*\sign^{*}(\varphi(\cdot, (n_{h}-1)h)-\varphi(x_{h}, n_{h}h))+ g(\cdot, (n_{h}-1)h) \beta(\alpha, h))(x_{h}),
\end{align*}
which is equivalent to the fact that
\[
J_{h}*[\mathbbm{1}^{+}-\mathbbm{1}^{-}](\varphi(\cdot, (n_{h}-1)h)-\varphi(x_{h}, n_{h}h))(x_{h})+g(x_{h}, (n_{h}-1)h) \beta(\alpha, h)\geq 0,
\]
namely
\begin{equation}\label{eq_finalbeforeproofthisimpliesmcf}
-g(x_{h}, (n_{h}-1)h) \beta(\alpha, h)\leq \int_{\mathbb{R}^{N}}[\mathbbm{1}^{+}-\mathbbm{1}^{-}](\varphi(y+x_{h}, (n_{h}-1)h)-\varphi(x_{h}, n_{h}h))J_{h}(y)\diff y,
\end{equation}
where $\mathbbm{1}^{+}$ and $\mathbbm{1}^{-}$ denote, respectively, the characteristic functions of $[0, +\infty)$ and $(-\infty, 0)$. To complete our proof of Proposition~\ref{prop visc subandsupersol}, we need to show that \eqref{eq_finalbeforeproofthisimpliesmcf} implies \eqref{def ofsubsolwmc new} if $|D\varphi(x_{0}, t_{0})|\neq 0$ or \eqref{eq equivthm2sub}  if $|D\varphi(x_{0}, t_{0})|=0$ and $D^{2}\varphi(x_{0}, t_{0})=0$. This is exactly the consistency of the scheme, which follows from Propositions~\ref{propconsistency1} and \ref{propconsistency2}.
\end{proof}

\begin{proof}[Proof of Theorem~\ref{mainthm}]
Let $u$ be as in the statement. Then $\sign^{*}(u(x,t))$ and $\sign_{*}(u(x,t))$ are the maximal upper semicontinuous subsolution and the minimal lower semicontinuous supersolution of \eqref{weighted MC equation} supplemented with the initial datum $\mathbbm{1}_{\smash{\overline{\Omega}}_{0}}-\mathbbm{1}_{\smash{\overline{\Omega}}_{0}^{c}}$ and $\mathbbm{1}_{\Omega_{0}}-\mathbbm{1}_{\Omega_{0}^{c}}$, respectively (see \cite{Barles-Soner-Souganidis} for the proof). This, together with Proposition~\ref{prop visc subandsupersol}, implies that
\begin{equation}\label{eq_km4oi5i0gji45n}
\mathrm{lim\, sup}^{*}u_{h}(x,t) \leq \sign^{*}(u(x,t)) \,\ \text{in} \,\ \mathbb{R}^{N}\times (0,+\infty)
\end{equation}
and
\begin{equation}\label{eq_bmpmko5p6ibm}
\mathrm{lim\, inf}_{*}u_{h}(x,t) \geq \sign_{*}(u(x,t)) \,\ \text{in} \,\ \mathbb{R}^{N} \times (0,+\infty).
\end{equation}
Since $u_{h}$ takes only values in $\{-1, 1\}$, \eqref{eq_km4oi5i0gji45n} implies that $\limsup^{*} u_{h}=-1$ in $(\Omega_{t} \cup \Gamma_{t})^{c}$, while \eqref{eq_bmpmko5p6ibm} yields $\liminf_{*}u_{h}=1$ in $\Omega_{t}$. This completes the proof of Theorem~\ref{mainthm}.
\end{proof}

\subsection{The consistency}
The next lemma will be used in the proof of the consistency result Proposition~\ref{propconsistency1}. 
\begin{lemma}\label{lemcomputderivative01}
Let $\alpha \in (1,2)$, $a \in \mathbb{R}$, $A \in \mathbb{M}^{N\times N}_{\textup{sym}}$, $E \in SO(N)$ and $\lambda:[0,1]\to \mathbb{R}$ be defined by 
\[
\lambda(s)=\int_{\mathbb{R}^{N}}[\mathbbm{1}^{+}-\mathbbm{1}^{-}](x_{1}+s\psi(x))P_{\alpha}(Ex)\diff x,
\]
where $\psi(x)=\langle A x, x\rangle -a$. Then 
\[
\lambda(s)=2s \left(\tr\left(\left(\int_{\mathbb{R}^{N-1}}(0, x^{\prime}) \otimes (0, x^{\prime}) P_{\alpha}(E(0, x^{\prime}))\diff x^{\prime}\right)A\right) - a\int_{\mathbb{R}^{N-1}}P_{\alpha}(E(0,x^{\prime}))\diff x^{\prime}\right)+o(s),
\]
where we denote $x=(x_{1}, x^{\prime})$ for each $x \in \mathbb{R}^{N}$.
\end{lemma}
\begin{proof}
Define $P(\cdot)=P_{\alpha}(E\cdot)$. We prove that  $\lambda^{\prime}(0)$ exists.  Since $\mathbbm{1}^{+}+\mathbbm{1}^{-}\equiv 1$ and $P \in L^{1}(\mathbb{R}^{N})$ (see \eqref{Kforalpha12}, \eqref{equivofnormsoneucl}), it is enough to prove that $\eta^{\prime}(0)$ exists, where 
\[
\eta(s)=\int_{\mathbb{R}^{N}}\mathbbm{1}^{+}(x_{1}+s\psi(x))P(x)\diff x,
\]
since, once the latter is satisfied, we have $\lambda^{\prime}(0)=2\eta^{\prime}(0)$. For each $\varepsilon>0$ small enough and for each $s \in [0,1]$, we define
\[
\eta_{\varepsilon}(s)=\int_{\mathbb{R}^{N}}\frac{1}{2}(1+\tanh)\left(\frac{x_{1}+s\psi(x)}{\varepsilon}\right)P(x)\diff x.
\]
Using \eqref{Kforalpha12}, \eqref{equivofnormsoneucl} and standard results, we observe that $\eta_{\varepsilon} \in C^{1}(0,1)$, $\eta_{\varepsilon} \to \eta$ pointwise in $[0, 1]$ and
\[
\eta^{\prime}_{\varepsilon}(s)=\int_{\mathbb{R}^{N}}\frac{1}{2\varepsilon}(1-\tanh^{2})\left(\frac{x_{1}+s\psi(x)}{\varepsilon}\right)\psi(x)P(x)\diff x.
\]
Applying Fubini's theorem, we deduce that $\eta_{\varepsilon}^{\prime} \in L^{1}(0,1)$, and hence $\eta_{\varepsilon}$ is absolutely continuous on $[0,1]$. Inasmuch as 
\begin{align*}
\frac{1}{\varepsilon}(1-\tanh^{2})\left(\frac{x_{1}+s\psi(x)}{\varepsilon}\right)&=\partial_{x_{1}}\left((1+\tanh)\left(\frac{x_{1}+s\psi(x)}{\varepsilon}\right)\right)\\
& \,\ \,\ -\frac{1}{\varepsilon}(1-\tanh^{2})\left(\frac{x_{1}+s\psi(x)}{\varepsilon}\right)s\partial_{x_{1}}\psi(x)
\end{align*}
and $\eta_{\varepsilon}(t)-\eta_{\varepsilon}(0)=\int_{0}^{t}\eta^{\prime}_{\varepsilon}(s)\diff s$, the following holds
\begin{equation}\label{eq_absolutecontinuityforetaequation}
\begin{split}
\eta_{\varepsilon}(t)-\eta_{\varepsilon}(0)
&=\int_{0}^{t}\int_{\mathbb{R}^{N}}\frac{1}{2}\partial_{x_{1}}\left((1+\tanh)\left(\frac{x_{1}+s\psi(x)}{\varepsilon}\right)\right)\psi(x)P(x)\diff x \diff s\\
& \,\ \,\ -\int_{0}^{t}\int_{\mathbb{R}^{N}}\frac{1}{2\varepsilon}(1-\tanh^{2})\left(\frac{x_{1}+s\psi(x)}{\varepsilon}\right)s(\partial_{x_{1}}\psi(x))\psi(x)P(x)\diff x \diff s.
\end{split}
\end{equation}
We integrate by parts with respect to  the variable $x_{1}$ in the first integral in \eqref{eq_absolutecontinuityforetaequation}, and then, using Fubini's theorem, we integrate by parts with respect to the variable $s$ in the second integral in \eqref{eq_absolutecontinuityforetaequation} to obtain
\begin{align*}
\eta_{\varepsilon}(t)-\eta_{\varepsilon}(0)&=-\int_{0}^{t}\int_{\mathbb{R}^{N}}\frac{1}{2}(1+\tanh)\left(\frac{x_{1}+s\psi(x)}{\varepsilon}\right)\partial_{x_{1}}(\psi(x)P(x))\diff x \diff s\\
& \,\ \,\ -t \int_{\mathbb{R}^{N}}\frac{1}{2}(1+\tanh)\left(\frac{x_{1}+t\psi(x)}{\varepsilon}\right)\partial_{x_{1}}\psi(x)P(x)\diff x\\
& \,\ \,\ +\int_{0}^{t}\int_{\mathbb{R}^{N}}\frac{1}{2}(1+\tanh)\left(\frac{x_{1}+s\psi(x)}{\varepsilon}\right)\partial_{x_{1}}\psi(x) P(x)\diff x \diff s.
\end{align*}
Letting $\varepsilon$ tend to $0+$, using Lebesgue's dominated convergence theorem, \eqref{Kforalpha12} and \eqref{equivofnormsoneucl}, yields
\begin{equation}\label{eq_ 34in959ug5n49gn}
\begin{split}
\eta(t)-\eta(0)&=-\int_{0}^{t}\int_{\mathbb{R}^{N}}\mathbbm{1}^{+}(x_{1}+s\psi(x))\partial_{x_{1}}(\psi(x)P(x))\diff x \diff s\\
& \,\ \,\ -t \int_{\mathbb{R}^{N}}\mathbbm{1}^{+}(x_{1}+t\psi(x))\partial_{x_{1}}\psi(x)P(x)\diff x\\
& \,\ \,\ +\int_{0}^{t}\int_{\mathbb{R}^{N}}\mathbbm{1}^{+}(x_{1}+s\psi(x))\partial_{x_{1}}\psi(x) P(x)\diff x \diff s.
\end{split}
\end{equation}
Applying Lebesgue's dominated convergence theorem again, one can see that all the integrals over $\mathbb{R}^{N}$ in \eqref{eq_ 34in959ug5n49gn} are continuous functions of $s$ or $t$. Thus, 
\begin{equation}\label{eq_computationdereta1}
\begin{split}
\eta^{\prime}(0)&=-\int_{\mathbb{R}^{N}}\mathbbm{1}^{+}(x_{1})\partial_{x_{1}}(\psi(x)P(x))\diff x\\
&=-\int_{\mathbb{R}^{N-1}}\int_{0}^{+\infty}\partial_{x_{1}}(\psi(x)P(x))\diff x_{1} \diff x^{\prime}\\
&=\int_{\mathbb{R}^{N-1}}\psi((0,x^{\prime}))P((0,x^{\prime}))\diff x^{\prime},
\end{split}
\end{equation}
where we have used \eqref{Kforalpha12}, \eqref{equivofnormsoneucl}. 
Since $P(x)=P(-x)$ for each $x\in \mathbb{R}^{N}$ (this comes from the fact that $\norm$ is even), we have $\lambda(0)=0$ and hence $\lambda(s)=\lambda^{\prime}(0)s+o(s)=2\eta^{\prime}(0)s+o(s)$. This, together with \eqref{eq_computationdereta1}, completes our proof of Lemma~\ref{lemcomputderivative01}.
\end{proof}

\begin{prop}\label{propconsistency1}
Let $(x_{0}, t_{0})$ and $\varphi \in C^{2}(\mathbb{R}^{N}\times (0, +\infty))$ be as in the proof of Proposition~\ref{prop visc subandsupersol}. Assume that $\alpha \in (1,2)$ and \eqref{eq_finalbeforeproofthisimpliesmcf} holds.  If $|D\varphi(x_{0}, t_{0})|\neq 0$, then \eqref{def ofsubsolwmc new} holds. If $|D\varphi(x_{0}, t_{0})|=0$ and $D^{2}\varphi(x_{0}, t_{0})=0$, then $\partial_{t} \varphi(x_{0}, t_{0})\leq 0$. 
\end{prop}
\begin{proof} We consider the next cases.\\
\noindent\emph{Case 1}: $|D\varphi(x_{0}, t_{0})|\neq 0.$ Setting $\sigma:=\sigma_{\alpha}(h)$, $t_{h}:=n_{h}h$, $\varphi_{h}(y, t):=\varphi(y+x_{h}, t)- \varphi(x_{h}, t_{h})$ and changing the variables (see \eqref{Kforalpha12} and \eqref{eq_definingthethresholding}), in view of \eqref{eq_finalbeforeproofthisimpliesmcf}, \eqref{defofkernelJh} and \eqref{eq_betathresholdalphatime}, we obtain
\begin{equation}\label{eq 001}
-g(x_{h}, t_{h}-h) \sigma^{\frac{1}{\alpha}}\leq \int_{\mathbb{R}^{N}}[\mathbbm{1}^{+}-\mathbbm{1}^{-}](\varphi_{h}(\sigma^{\frac{1}{\alpha}}y, t_{h}-h))P_{\alpha}(y) \diff y.
\end{equation}
Since $\varphi_{h} \in C^{2}(\mathbb{R}^{N}\times (0,+\infty))$, $\varphi_{h}(0, t_{h})=0$ and $\sigma = h^{\frac{\alpha}{2}}$, we expand $\varphi_{h}(\sigma^{\frac{1}{\alpha}}y, t_{h}-h)$ as follows
\begin{align*}
\varphi_{h}(\sigma^{\frac{1}{\alpha}}y, t_{h}-h)
= \sigma^{\frac{1}{\alpha}} \langle D\varphi(x_{h}, t_{h}), y\rangle &+\sigma^{\frac{2}{\alpha}}(-\partial_{t}\varphi(x_{h}, t_{h}) + \frac{1}{2}\langle D^{2}\varphi(x_{h}, t_{h}) y,y \rangle) \\  & \,\ \,\ \,\ + \sigma^{\frac{2}{\alpha}}(O((\sigma^{\frac{2}{\alpha}} + \sigma^{\frac{1}{\alpha}}|y|)(|y|^{2}+1))).
\end{align*}
Denote $p_{h}= D\varphi(x_{h}, t_{h})$, $a_{h}=\partial_{t}\varphi(x_{h}, t_{h})$ and $A_{h}=\frac{1}{2}D^{2}\varphi(x_{h}, t_{h})$. Then 
\[
\varphi_{h}(\sigma^{\frac{1}{\alpha}}y, t_{h}-h)=\sigma^{\frac{1}{\alpha}} \langle p_{h}, y\rangle + \sigma^{\frac{2}{\alpha}}(-a_{h} + \langle A_{h} y,y \rangle+O((\sigma^{\frac{2}{\alpha}}+\sigma^{\frac{1}{\alpha}}|y|)(|y|^{2}+1))).
\]
After a rotation and a change of variables,  we may assume that $p_{h}= \beta_{h} (1, 0,\dotsc, 0)$, where $\beta_{h}=|p_{h}|$. We denote by $\widetilde{A}_{h}$ the matrix we obtain from $A_{h}$ after the rotation. The integration in \eqref{eq 001} is taking place over the sets $C_{h}$ and $C_{h}^{c}$, where
\[
C_{h}=\{y \in \mathbb{R}^{N}: \sigma^{\frac{1}{\alpha}} \beta_{h} y_{1} + \sigma^{\frac{2}{\alpha}}(-a_{h}+ \langle \widetilde{A}_{h} y,y \rangle + O((\sigma^{\frac{2}{\alpha}} + \sigma^{\frac{1}{\alpha}} |y|)(|y|^{2}+1)))\geq 0\}.
\]
Since $\beta_{h} \to \beta_{0}:=|D\varphi(x_{0}, t_{0})|>0$ as $h\to 0+$, $\beta_{h}>0$ for each $h>0$ small enough. Thus,
\[
C_{h}=\{y \in \mathbb{R}^{N}: y_{1}+ \sigma^{\frac{1}{\alpha}} \beta_{h}^{-1}(-a_{h} + \langle \widetilde A_{h} y,y \rangle + O((\sigma^{\frac{2}{\alpha}}+\sigma^{\frac{1}{\alpha}}|y|)(|y|^{2}+1))) \geq 0\}.
\]
Using that $\widetilde{A}_{h}\to \widetilde{A}$ and $a_{h} \to a$ as $h \to 0+$, we get
\[
C_{h}=\{y \in \mathbb{R}^{N}: y_{1} + \sigma^{\frac{1}{\alpha}} \beta_{h}^{-1}(-a+\langle\widetilde{A}y,y\rangle +O((\sigma^{\frac{2}{\alpha}}+\sigma^{\frac{1}{\alpha}}|y|)(|y|^{2}+1))+o(1)(|y|^{2}+1))\geq 0\},
\]
where $\widetilde{A}$ is the rotated matrix $A$ and $a=\partial_{t} \varphi(x_{0}, t_{0})$, namely $\widetilde{A}=EAE^{\mathrm{T}}$ for some $E \in SO(N)$. After all, we deduce the following
\begin{equation}\label{eq_002}
\begin{split}
-g(x_{h}, t_{h}-h)\sigma^{\frac{1}{\alpha}} & \leq \int_{\mathbb{R}^{N}} [\mathbbm{1}^{+}-\mathbbm{1}^{-}](\Psi_{h}(y))P(y)\diff y\\
&=\int_{B_{R}}[\mathbbm{1}^{+}-\mathbbm{1}^{-}](\Psi_{h}(y))P(y)\diff y + \int_{B^{c}_{R}}[\mathbbm{1}^{+}-\mathbbm{1}^{-}](\Psi_{h}(y))P(y)\diff y,
\end{split}
\end{equation}
where for each $y \in \mathbb{R}^{N}$,
\[
\Psi_{h}(y)=y_{1}+\sigma^{\frac{1}{\alpha}}\beta_{h}^{-1}(\langle \widetilde{A} y,y \rangle -a+ O((\sigma^{\frac{2}{\alpha}}+\sigma^{\frac{1}{\alpha}}|y|)(|y|^{2}+1))+o(1)(|y|^{2}+1))
\]
and  $P(y)=P_{\alpha}(E^{\mathrm{T}}y)$. Let $\theta \in (1/\alpha,1)$ and $R=\sigma^{-\frac{\theta}{\alpha}}$. Taking into account \eqref{Kforalpha12} and \eqref{equivofnormsoneucl}, we observe that
\begin{equation}\label{estIIintegral}
\begin{split}
\sigma^{-\frac{1}{\alpha}}\int_{B^{c}_{R}}[\mathbbm{1}^{+}-\mathbbm{1}^{-}](\Psi_{h}(y))P(y)\diff y  \leq
\sigma^{-\frac{1}{\alpha}}\int_{B^{c}_{R}}P(y)\diff y & \leq C^{N+\alpha}\sigma^{-\frac{1}{\alpha}} \int_{B^{c}_{R}} |y|^{-N-\alpha} \diff y \\
& = \frac{C^{N+\alpha}}{\alpha} \omega_{N-1} \sigma^{-\frac{1}{\alpha}} R^{-\alpha}\\
& = \frac{C^{N+\alpha}}{\alpha}\omega_{N-1} \sigma^{\theta-\frac{1}{\alpha}}, 
\end{split} 
\end{equation}
which tends to $0$ as $h \to 0+$ (recall that $\sigma^{\frac{2}{\alpha}}=h$ and $C\geq 1$ is the constant coming from \eqref{equivofnormsoneucl}). 
Fix $\gamma>0$. Then for each $h>0$ small enough and for each $y \in B_{R}$, 
\begin{equation}\label{estimate by barrier CS}
\Psi_{h}(y)\leq \Phi_{h}(y), 
\end{equation}
where
\begin{equation}\label{CS barrier}
\Phi_{h}(y)=y_{1}+\sigma^{\frac{1}{\alpha}}\beta_{h}^{-1}(\langle (\widetilde{A}+\gamma\, \mathrm{Id}) y,y\rangle - a+\gamma),
\end{equation}
since for each $y \in B_{R}$,
\begin{equation}\label{eq 004}
O((\sigma^{\frac{2}{\alpha}}+\sigma^{\frac{1}{\alpha}}|y|)(|y|^{2}+1))+o(1)(|y|^{2}+1)= O(\sigma^{\frac{1-\theta}{\alpha}}(|y|^{2}+1))+o(1)(|y|^{2}+1)\leq \gamma (|y|^{2}+1).
\end{equation}
Inasmuch as $[\mathbbm{1}^{+}-\mathbbm{1}^{-}]$ is nondecreasing, \eqref{estimate by barrier CS} implies that 
\begin{equation}\label{estimate by barrier CS integral}
\int_{B_{R}}[\mathbbm{1}^{+}-\mathbbm{1}^{-}](\Psi_{h}(y))P(y) \diff y \leq \int_{B_{R}}[\mathbbm{1}^{+}-\mathbbm{1}^{-}](\Phi_{h}(y))P(y)\diff y,
\end{equation}
which, together with \eqref{eq_002}, \eqref{estIIintegral} and the facts that \eqref{estIIintegral} holds also with $\Psi_{h}$ replaced by $\Phi_{h}$ and $(x_{h}, t_{h})\to (x_{0}, t_{0})$ as $h \to 0+$, we deduce the following
\begin{equation}\label{eq_kmfo3jf30ji0}
-g(x_{0}, t_{0})\leq \lim_{h\to 0+} \sigma^{-\frac{1}{\alpha}} \int_{\mathbb{R}^{N}}[\mathbbm{1}^{+}-\mathbbm{1}^{-}](\Phi_{h}(y))P(y) \diff y.
\end{equation} 
Applying Lemma~\ref{lemcomputderivative01} with $\psi(y)=\langle (\widetilde{A}+\gamma \mathrm{Id}) y, y \rangle - (a-\gamma)$, $s=\sigma^{\frac{1}{\alpha}}\beta_{h}^{-1}$ and with $P_{\alpha}$ replaced by $P$, yields
\begin{equation}\label{eq_expressionfordiffindicatorint}
\begin{split}
\int_{\mathbb{R}^{N}}[\mathbbm{1}^{+}-\mathbbm{1}^{-}](\Phi_{h}(y))P(y)\diff y&=2\sigma^{\frac{1}{\alpha}}\beta_{h}^{-1}\tr\left(\left(\int_{\mathbb{R}^{N-1}}(0,x^{\prime})\otimes (0,x^{\prime})P((0,x^{\prime}))\diff x^{\prime}\right)(\widetilde{A}+\gamma \mathrm{Id})\right)\\ & \,\ \,\  - 2\sigma^{\frac{1}{\alpha}}\beta_{h}^{-1}(a-\gamma)\int_{\mathbb{R}^{N-1}}P((0,x^{\prime})) \diff x^{\prime} + o(\sigma^{\frac{1}{\alpha}}\beta^{-1}_{h}).
\end{split}
\end{equation}
Recalling that $\beta_{h}=|D\varphi(x_{h}, t_{h})|>0$ for each $h>0$ small enough, $\beta_{h}\to \beta_{0}>0$ and gathering together \eqref{eq_kmfo3jf30ji0} and the above equality, we get
\begin{align*}
-g(x_{0}, t_{0})\beta_{0} & \leq 2\tr\left(\left(\int_{\mathbb{R}^{N-1}}(0, x^{\prime})\otimes (0,x^{\prime}) P((0,x^{\prime}))\diff x^{\prime}\right)(\widetilde{A}+\gamma \mathrm{Id})\right)\\& \,\ \,\ \,\  -2 (a-\gamma)\int_{\mathbb{R}^{N-1}}P((0,x^{\prime})) \diff x^{\prime},
\end{align*}
where $x=(x_{1}, x^{\prime}) \in \mathbb{R}^{N}$. Letting $\gamma$ tend to $0+$, one has
\begin{equation}\label{eq_meancvokm340}
\begin{split}
a &\leq \mu_{\alpha} \left(\tr\left(\left(\int_{\mathbb{R}^{N-1}}(0,x^{\prime}) \otimes (0, x^{\prime}) P((0, x^{\prime})) \diff x^{\prime} \right)2\widetilde{A}\right)+g(x_{0}, t_{0}) \beta_{0}\right)\\&=:\mu_{\alpha}\eta+\mu_{\alpha}g(x_{0}, t_{0})\beta_{0},
\end{split}
\end{equation}
where $\mu_{\alpha}=(2\int_{\mathbb{R}^{N-1}}P((0,x^{\prime}))\diff x^{\prime})^{-1} \in (0,+\infty)$. 

Since $E^{\mathrm{T}}(1,0^{\prime})=\frac{D\varphi(x_{0}, t_{0})}{|D\varphi(x_{0}, t_{0})|}=:e$ and $E \in SO(N)$, changing the variables and recalling that $\widetilde{A}=EAE^{\mathrm{T}}$, $P(\cdot)=P_{\alpha}(E^{\mathrm{T}}\cdot)$, we obtain
\begin{equation*}\label{eq_1computtraceorhogonal}
\begin{split}
\eta&=\tr\left(\left(\int_{\mathbb{R}^{N-1}}E^{\mathrm{T}}(0,x^{\prime})\otimes E^{\mathrm{T}}(0, x^{\prime}) P_{\alpha}(E^{\mathrm{T}}(0,x^{\prime}))\diff x^{\prime}\right)2A\right)\\
&=\tr\left(\left(\int_{e^{\perp}}x \otimes x\, P_{\alpha}(x) \diff \mathcal{H}^{N-1}(x)\right)2A\right), 
\end{split}
\end{equation*}
where $A=\frac{1}{2}D^{2}\varphi(x_{0},t_{0})$. Using \eqref{Kforalpha12}, \cite[Theorem~3.2.22~(3)]{Federer} and changing the variables (namely, $t=r\norm(\theta)$), one has  
\begin{align*}
\int_{e^{\perp}}x \otimes x\, P_{\alpha}(x) \diff \mathcal{H}^{N-1}(x)&=\int_{0}^{+\infty} \frac{t^{N}}{1+t^{N+\alpha}} \diff t \int_{\mathbb{S}^{N-1} \cap e^{\perp}} \theta \otimes \theta \frac{\diff \mathcal{H}^{N-2}(\theta)}{\norm(\theta)^{N+1}}\\
&=C_{N, \alpha}\int_{\mathbb{S}^{N-1} \cap e^{\perp}} \theta \otimes \theta \frac{\diff \mathcal{H}^{N-2}(\theta)}{\norm(\theta)^{N+1}}.
\end{align*}
Thus, \eqref{eq_meancvokm340} yields the desired inequality \eqref{def ofsubsolwmc new} for $\varphi$ at $(x_{0}, t_{0})$.
\\

Next, we assume that $|D\varphi(x_{0}, t_{0})|=0$, $D^{2}\varphi(x_{0}, t_{0})=0$ and $\beta_{h} \to 0$ as $h\to 0+$. We need to distinguish between three further cases.\\
\noindent \emph{Case 2.1}: along some subsequence $\beta_{h}\neq 0$ and $\sigma^{\frac{1}{\alpha}}\beta^{-1}_{h}\to 0$. Then, in view of \eqref{eq_002},  \eqref{estIIintegral}, \eqref{estimate by barrier CS integral} and the fact that \eqref{estIIintegral} holds also with $\Psi_{h}$ replaced by $\Phi_{h}$, we have 
\[
0=\lim_{h\to 0+}-g(x_{h}, t_{h}-h)\beta_{h}\leq \lim_{h\to 0+} \sigma^{-\frac{1}{\alpha}}\beta_{h} \int_{\mathbb{R}^{N}}[\mathbbm{1}^{+}-\mathbbm{1}^{-}](\Phi_{h}(y))P(y) \diff y,
\]
which, together with \eqref{eq_expressionfordiffindicatorint} and the fact that $D^{2}\varphi(x_{0}, t_{0})=0$, yields that $\partial_{t} \varphi(x_{0}, t_{0}) \leq 0$.\\
\noindent \emph{Case 2.2}: along some subsequence $\beta_{h}=0$ or $\sigma^{\frac{1}{\alpha}}\beta_{h}^{-1}\to +\infty$. Assume that $a_{h}\to\partial_{t} \varphi(x_{0}, t_{0})>0$. Then the characteristic function of the set
\[
\{y \in \mathbb{R}^{N}: \beta_{h} y_{1} + \sigma^{\frac{1}{\alpha}}(-a_{h}+ \langle \widetilde{A}_{h} y,y \rangle + O((\sigma^{\frac{2}{\alpha}} + \sigma^{\frac{1}{\alpha}} |y|)(|y|^{2}+1)))\geq 0\},
\]
which is the same as the set
\[
\{y\in \mathbb{R}^{N}: \sigma^{-\frac{1}{\alpha}} \beta_{h} y_{1} -a_{h}+ \langle \widetilde{A}_{h} y,y \rangle + O((\sigma^{\frac{2}{\alpha}} + \sigma^{\frac{1}{\alpha}} |y|)(|y|^{2}+1))\geq 0\},
\]
pointwise converges to the constant function 0. Using this, \eqref{eq 001} and Lebesgue's dominated convergence theorem, we obtain
\[
\int_{\mathbb{R}^{N}} P_{\alpha}(y) \diff y \leq 2\int_{\mathbb{R}^{N}}\mathbbm{1}^{+}(\varphi_{h}(\sigma^{\frac{1}{\alpha}}y, t_{h}-h))P_{\alpha}(y)\diff y +g(x_{h}, t_{h}-h)\sigma^{\frac{1}{\alpha}} \to 0
\]
as $h\to 0+$, which leads to a contradiction with the fact that $\int_{\mathbb{R}^{N}}P_{\alpha}(y)\diff y>0$. \\
\noindent \emph{Case 2.3}: along some subsequence $\sigma^{\frac{1}{\alpha}}\beta_{h}^{-1}\to \beta>0$. Then the characteristic function of the set
\[
\{y \in \mathbb{R}^{N}: \beta_{h} y_{1} + \sigma^{\frac{1}{\alpha}}(-a_{h}+ \langle \widetilde{A}_{h} y,y \rangle + O((\sigma^{\frac{2}{\alpha}} + \sigma^{\frac{1}{\alpha}} |y|)(|y|^{2}+1)))\geq 0\},
\]
which is the same as the set
\[
\{y \in \mathbb{R}^{N}: y_{1} + \sigma^{\frac{1}{\alpha}}\beta_{h}^{-1}(-a_{h}+ \langle \widetilde{A}_{h} y,y \rangle + O((\sigma^{\frac{2}{\alpha}} + \sigma^{\frac{1}{\alpha}} |y|)(|y|^{2}+1)))\geq 0\},
\]
pointwise converges to the characteristic function of the set 
\[
\{y \in \mathbb{R}^{N}: y_{1}-\beta \partial_{t}\varphi(x_{0}, t_{0}) \geq 0\}.
\]
But we know that 
\[
\int_{\mathbb{R}^{N}}P_{\alpha}(y)\diff y - 2\int_{\mathbb{R}^{N}}\mathbbm{1}^{+}(\varphi_{h}(\sigma^{\frac{1}{\alpha}}y, t_{h}-h)) P_{\alpha}(y) \diff y \leq g(x_{h}, t_{h}-h) \sigma^{\frac{1}{\alpha}}
\]
(see \eqref{eq 001}). Thus, letting $h\to 0+$ and applying Lebesgue's dominated convergence theorem, we obtain
\[
\int_{\mathbb{R}^{N}}P_{\alpha}(y)\diff y - 2\int_{\mathbb{R}^{N}}\mathbbm{1}^{+}(y_{1}-\beta \partial_{t} \varphi(x_{0}, t_{0})) P_{\alpha}(y) \diff y \leq 0.
\]
Therefore,
\[
\partial_{t} \varphi(x_{0}, t_{0}) \leq 0.
\]
This completes our proof of Proposition~\ref{propconsistency1}.
\end{proof}
\begin{prop}\label{propconsistency2}
Let $(x_{0}, t_{0})$ and $\varphi \in C^{2}(\mathbb{R}^{N}\times (0, +\infty))$ be as in the proof of Proposition~\ref{prop visc subandsupersol}. Assume that $\alpha=1$ and \eqref{eq_finalbeforeproofthisimpliesmcf} holds.  If $|D\varphi(x_{0}, t_{0})|\neq 0$, then \eqref{def ofsubsolwmc new} holds with $\alpha=1$. If $|D\varphi(x_{0}, t_{0})|=0$ and $D^{2}\varphi(x_{0}, t_{0})=0$, then $\partial_{t} \varphi(x_{0}, t_{0})\leq 0$. 
\end{prop}
\begin{proof} We consider the next cases.\\
\noindent \emph{Case 1}: $|D\varphi(x_{0}, t_{0})|\neq 0$. Setting $\sigma:=\sigma_{1}(h)$, $t_{h}:=n_{h}h$, $\varphi_{h}(y, t):=\varphi(y+x_{h}, t)- \varphi(x_{h}, t_{h})$ and changing the variables (see \eqref{Kforalpha12} and \eqref{eq_definingthethresholding}), in view of \eqref{eq_finalbeforeproofthisimpliesmcf}, \eqref{defofkernelJh} and \eqref{eq_betathresholdalphatime}, we obtain
\begin{equation}\label{eq 001caseaplha1}
-g(x_{h}, t_{h}-h)\sigma|\ln(\sigma)|\leq \int_{\mathbb{R}^{N}}[\mathbbm{1}^{+}-\mathbbm{1}^{-}](\varphi_{h}(\sigma y, t_{h}-h))P_{1}(y) \diff y.
\end{equation}
Taking into account that $\varphi_{h} \in C^{2}(\mathbb{R}^{N}\times (0,+\infty))$, $\varphi_{h}(0, t_{h})=0$ and $h=\sigma^{2}|\ln(\sigma)|$, using Taylor's formula, we have
\[
\varphi_{h}(\sigma y, t_{h}-h)=\sigma \langle p_{h},  y\rangle+ \sigma(-\sigma|\ln(\sigma)|a_{h} + \sigma\langle A_{h}y, y\rangle + O((\sigma^{2}|\ln(\sigma)|+\sigma^{2}|y|^{2})(|y|+\sigma|\ln(\sigma)|))),
\]
where $p_{h}= D\varphi(x_{h}, t_{h})$, $a_{h}=\partial_{t}\varphi(x_{h}, t_{h})$ and $A_{h}=\frac{1}{2}D^{2}\varphi(x_{h}, t_{h})$. 
After a rotation and a change of variables,  we may assume that $p_{h}= \beta_{h} (1, 0,\dotsc, 0)$, where $\beta_{h}=|p_{h}|$. We denote by $\widetilde{A}_{h}$ the matrix we obtain from $A_{h}$ after the rotation. The integration in \eqref{eq 001caseaplha1} is taking place over the sets $C_{h}$ and $C_{h}^{c}$, where 
\[
C_{h}=\{y \in \mathbb{R}^{N}: \sigma \beta_{h} y_{1} + \sigma(-\sigma|\ln(\sigma)|a_{h}+ \sigma \langle \widetilde{A}_{h} y,y \rangle + O((\sigma^{2}|\ln(\sigma)|+\sigma^{2}|y|^{2})(|y|+\sigma|\ln(\sigma)|)))\geq 0\}.
\]
Since $\beta_{h} \to |D\varphi(x_{0}, t_{0})|>0$ as $h\to 0+$, $\beta_{h}>0$ for each $h>0$ small enough. Thus,
\[
C_{h}=\{y \in \mathbb{R}^{N}: y_{1}+ \beta_{h}^{-1}(-\sigma|\ln(\sigma)|a_{h}+ \sigma \langle \widetilde{A}_{h} y,y \rangle + O((\sigma^{2}|\ln(\sigma)|+\sigma^{2}|y|^{2})(|y|+\sigma|\ln(\sigma)|))) \geq 0\}.
\]
Using that $\widetilde{A}_{h}\to \widetilde{A}$ and $a_{h} \to a$ as $h \to 0+$, we deduce that $C_{h}$ consists of points $y \in \mathbb{R}^{N}$ such that
\begin{align*}
0&\leq y_{1} +  \beta_{h}^{-1}(-\sigma|\ln(\sigma)|a+ \sigma \langle \widetilde{A} y,y \rangle+o(1)(\sigma(|y|^{2}+|\ln(\sigma)|)))\\ & \,\ \,\ \,\ +\beta_{h}^{-1}O((\sigma^{2}|\ln(\sigma)|+\sigma^{2}|y|^{2})(|y|+\sigma|\ln(\sigma)|)),
\end{align*}
where $\widetilde{A}$ is the rotated matrix $A$ and $a=\partial_{t} \varphi(x_{0}, t_{0})$, namely $\widetilde{A}=EAE^{\mathrm{T}}$ for some $E \in SO(N)$. For each $y \in \mathbb{R}^{N}$ and for each $\gamma>0$, define $P(y)=P_{1}(E^{\mathrm{T}}y)$, 
\begin{align*}
\Psi_{h}(y) &= y_{1} +  \beta_{h}^{-1}(-\sigma|\ln(\sigma)|a+ \sigma \langle \widetilde{A} y,y \rangle+o(1)(\sigma(|y|^{2}+|\ln(\sigma)|)))\\ & \,\ \,\ \,\ +\beta_{h}^{-1}O((\sigma^{2}|\ln(\sigma)|+\sigma^{2}|y|^{2})(|y|+\sigma|\ln(\sigma)|))
\end{align*}
and
\[
\Phi_{h}(y)=y_{1}+\beta_{h}^{-1}(-\sigma|\ln(\sigma)|(a-\gamma) + \sigma \langle (\widetilde{A}+\gamma \mathrm{Id})y, y \rangle).
\]
Then, summing up the above considerations, in view of \eqref{eq 001caseaplha1}, for each $h>0$ small enough, we obtain
\begin{equation}\label{eq_estgfirstfunction1777907890}
-g(x_{h}, t_{h}-h)\sigma|\ln(\sigma)|\leq \int_{\mathbb{R}^{N}}[\mathbbm{1}^{+}-\mathbbm{1}^{-}](\Psi_{h}(y))P(y) \diff y.
\end{equation}
For each $\eta>0$, 
\begin{equation}\label{eq_uniformestimkernelplays1}
\int_{\mathbb{R}^{N}}[\mathbbm{1}^{+}-\mathbbm{1}^{-}](\Psi_{h}(y))P(y) \diff y \leq \int_{B_{\eta \sigma^{-1}}}[\mathbbm{1}^{+}-\mathbbm{1}^{-}](\Psi_{h}(y))P(y) \diff y+ \int_{B^{c}_{\eta \sigma^{-1}}}P(y) \diff y  
\end{equation}
and 
\begin{equation}\label{eq_estimoutofballcasealpha1}
 \int_{B^{c}_{\eta \sigma^{-1}}}P(y) \diff y  \leq C^{N+1}\int_{B^{c}_{\eta \sigma^{-1}}}\frac{\diff y}{|y|^{N+1}}=C^{N+1}\omega_{N-1} \eta^{-1}\sigma,
\end{equation}
where $C=C(\norm)\geq 1$ is the constant coming from \eqref{equivofnormsoneucl} and we have used \eqref{Kforalpha12}. Since $\Psi_{h} \leq \Phi_{h}$ in $B_{\eta \sigma^{-1}}$ (provided that $h>0$ is small enough and $\eta$ is sufficiently small with respect to $\gamma$) and  the function $[\mathbbm{1}^{+}-\mathbbm{1}^{-}]$ is nondecreasing,  \eqref{eq_estgfirstfunction1777907890}-\eqref{eq_estimoutofballcasealpha1} yield
\begin{equation*}
-g(x_{h}, t_{h}-h)\sigma|\ln(\sigma)|\leq \int_{B_{\eta \sigma^{-1}}}[\mathbbm{1}^{+}-\mathbbm{1}^{-}](\Phi_{h}(y))P(y) \diff y+C^{N+1}\omega_{N-1} \eta^{-1}\sigma.
\end{equation*}
Thus, we reach the inequality
\begin{equation}\label{eq_19i33i323ii20i30i09}
-g(x_{h}, t_{h}-h)\varrho \leq f(\varrho) +C^{N+1}\omega_{N-1} \eta^{-1}\sigma,
\end{equation}
where $\varrho=\sigma |\ln(\sigma)|$ and
\[
f(\varrho)=\int_{B_{\eta \sigma^{-1}}}[\mathbbm{1}^{+}-\mathbbm{1}^{-}](y_{1}+F(y, \varrho))P(y) \diff y,
\]
where $F(y, \varrho)= \beta_{h}^{-1} (\sigma \langle (\widetilde{A}+\gamma \mathrm{Id})y, y \rangle- \varrho (a-\gamma))$. Using \eqref{Kforalpha12} and the Lebesgue dominated convergence theorem, we get $f(0)=\int_{\mathbb{R}^{N}}[\mathbbm{1}^{+}-\mathbbm{1}^{-}](y_{1})P(y) \diff y=0$. We need to compute $f^{\prime}(0)$. Arguing as in the proof of Lemma~\ref{lemcomputderivative01}, namely, using a smooth approximation of $\mathbbm{1}^{+}$, we deduce that
\begin{align*}
f^{\prime}(\varrho)&=2 \int_{B_{\eta \sigma^{-1}}}\delta (y_{1}+F(y, \varrho)) \partial_{\varrho}F(y, \varrho) P(y) \diff y\\ & \,\ \,\ \,\  + \int_{\partial B_{\eta \sigma^{-1}}}[\mathbbm{1}^{+}-\mathbbm{1}^{-}](y_{1}+F(y, \varrho)) P(y)(\eta \sigma^{-1})^{\prime} \diff \mathcal{H}^{N-1}(y),
\end{align*}
where $\delta$ is the Dirac delta function and $(\eta \sigma^{-1})^{\prime}$ denotes the derivative of $\eta \sigma^{-1}$ with respect to $\varrho$. Then
\begin{equation}\label{eq_dofftwoint1}
f(\varrho) - f(0)=\mathrm{I}_{\varrho}+\mathrm{II}_{\varrho},
\end{equation}
where 
\[
\mathrm{I}_{\varrho}=2 \int_{0}^{\varrho}\int_{B_{\eta \sigma^{-1}(s)}}\delta(y_{1}+F(y,s))\partial_{s}F(y,s)P(y) \diff y \diff s
\]
and 
\begin{align*}
\mathrm{II}_{\varrho}&=\int_{0}^{\varrho}\int_{\partial B_{\eta \sigma^{-1}(s)}}[\mathbbm{1}^{+}-\mathbbm{1}^{-}](y_{1}+F(y,s))P(y) (\eta \sigma^{-1})^{\prime} \diff \mathcal{H}^{N-1}(y) \diff s\\& \leq \int_{0}^{\varrho}\int_{\mathbb{S}^{N-1}}P(\eta \sigma^{-1} (s) y) (\eta \sigma^{-1}(s))^{\prime}(\eta \sigma^{-1}(s))^{N-1} \diff \mathcal{H}^{N-1}(y) \diff s,
\end{align*}
where $s=\sigma(s)|\ln(\sigma(s))|$. 

Using \eqref{Kforalpha12} and the facts that $E\in SO(N)$ and $\sigma^{\prime}(s)=(|\ln(\sigma(s))|-1)^{-1}$, we obtain
\begin{equation}\label{eq_secondintegralrhotendsto0}
\begin{split}
\varrho^{-1}\mathrm{II}_{\varrho} & \leq C^{N+1}\omega_{N-1} \varrho^{-1} \int_{0}^{\varrho}\frac{(\eta \sigma^{-1}(s))^{N-1}}{C^{N+1}+(\eta \sigma^{-1}(s))^{N+1}}\left(\frac{\eta \sigma^{\prime}(s)}{\sigma^{2}(s)}\right) \diff s\\
& \leq C^{N+1}\omega_{N-1} (\eta\varrho)^{-1} \int_{0}^{\varrho} (|\ln(\sigma(s))|-1)^{-1} \diff s \to 0
\end{split}
\end{equation}
as $\varrho \to 0+$. Thus, $\varrho^{-1} \mathrm{II}_{\varrho} \to 0$ as $\varrho \to 0+$. 

Next, we analyze $\varrho^{-1} \mathrm{I}_{\varrho}$. Inasmuch as $\partial_{s}F(y, s)=\beta_{h}^{-1} (\langle (\widetilde{A}+\gamma \mathrm{Id})y, y\rangle \sigma^{\prime}-(a-\gamma))$, we have $\mathrm{I}_{\varrho}=\mathrm{I}^{1}_{\varrho}+\mathrm{I}^{2}_{\varrho}$, where
\[
\mathrm{I}^{1}_{\varrho}=-2\beta_{h}^{-1}(a-\gamma)\int_{0}^{\varrho}\int_{B_{\eta \sigma(s)^{-1}}}\delta(y_{1}+F(y,s))P(y) \diff y\diff s
\]
and
\[
\mathrm{I}^{2}_{\varrho}=2 \beta_{h}^{-1} \int_{0}^{\varrho}(|\ln(\sigma(s))|-1)^{-1}\int_{B_{\eta \sigma(s)^{-1}}}\delta(y_{1}+F(y,s))\langle (\widetilde{A}+\gamma \mathrm{Id}) y, y \rangle P(y) \diff y \diff s.
\]
Using the approximation of $\delta$ through $\frac{1}{2}(1-\tanh^{2})(\frac{\cdot}{\varepsilon})$, the properties of $P(\cdot)=P_{1}(E^{\mathrm{T}}\cdot)$ (see \eqref{Kforalpha12})  and changing the variables, one has
\begin{equation}\label{eq_estimlimitfirstintegralformulaexplicit}
\begin{split}
\lim_{\varrho\to 0+}\varrho^{-1}\mathrm{I}^{1}_{\varrho}& =-2(a-\gamma)\beta^{-1}_{0}\int_{\mathbb{R}^{N}}\delta(y_{1})P(y) \diff y =-2(a-\gamma)\beta^{-1}_{0}\int_{\mathbb{R}^{N-1}}P(0, y^{\prime}) \diff y^{\prime}\\ &= -2(a-\gamma)\beta^{-1}_{0}\int_{e^{\perp}}P_{1}(x) \diff \mathcal{H}^{N-1}(x),
\end{split}
\end{equation}
where $e=\frac{D\varphi(x_{0}, t_{0})}{|D\varphi(x_{0}, t_{0})|}$. On the other hand, using the properties of $P$ and changing the variables, we deduce the following chain of estimates
\begin{equation*}\label{eq_estimsecondintegralformulaexplicit}
\begin{split}
\lim_{\varrho\to 0+}\varrho^{-1} \mathrm{I}^{2}_{\varrho}&=\lim_{\varrho \to 0+}\frac{2}{\beta_{h}\varrho} \int_{0}^{\varrho}\frac{\ln(\eta \sigma^{-1})}{|\ln(\sigma)|-1} \frac{1}{\ln(\eta \sigma^{-1})}\int_{B_{\eta \sigma^{-1}}}\delta(y_{1}+F(y,s))\langle (\widetilde{A}+\gamma \mathrm{Id}) y, y \rangle P(y)\, d y\, d s\\
&=\lim_{R\to +\infty} \frac{2}{\beta_{0}\ln(R)} \int_{B_{R}}\delta(y_{1})\langle (\widetilde{A}+\gamma \mathrm{Id}) y, y \rangle P(y) \diff y \\
&=\lim_{R\to +\infty}\frac{2}{\beta_{0} \ln(R)}\int_{\mathbb{R}^{N-1}}\int_{0}^{+\infty}-\partial_{y_{1}}(\langle(\widetilde{A}+\gamma \mathrm{Id}) y, y\rangle P(y) \mathbbm{1}_{B_{R}}(y)) \diff y_{1} \diff y^{\prime}\\
&=\lim_{R\to +\infty}\frac{2}{\beta_{0} \ln(R)}\int_{|y^{\prime}|<R}\langle(\widetilde{A}+\gamma \mathrm{Id}) (0,y^{\prime}), (0,y^{\prime}) \rangle  P((0,y^{\prime})) \diff y^{\prime}\\
&=\lim_{R\to +\infty}\frac{2}{\beta_{0} \ln(R)}\int_{0}^{R}\int_{\mathbb{S}^{N-1} \cap \{0\}\times \mathbb{R}^{N-1}} \frac{r^{N}\langle(\widetilde{A}+\gamma \mathrm{Id}) (0,\theta), (0,\theta) \rangle \diff \mathcal{H}^{N-2}(\theta) }{1+r^{N+1}\norm(E^{\mathrm{T}}(0, \theta))^{N+1}}\diff r\\
&=\lim_{R\to +\infty}\frac{2}{\beta_{0} \ln(R)}\int_{0}^{R}\int_{\mathbb{S}^{N-1} \cap e^{\perp}} \frac{t^{N}}{1+t^{N+1}} \langle(A+\gamma \mathrm{Id}) \theta, \theta \rangle\frac{\diff \mathcal{H}^{N-2}(\theta)}{\norm(\theta)^{N+1}} \diff t\\
& = \frac{1}{\beta_{0}}\tr\left(\left(\int_{\mathbb{S}^{N-1} \cap e^{\perp}} \theta \otimes \theta \frac{\diff \mathcal{H}^{N-2}(\theta)}{\norm(\theta)^{N+1}}\right)2(A+\gamma \mathrm{Id})\right),
\end{split}
\end{equation*}
where $A=\frac{1}{2}D^{2}\varphi(x_{0}, t_{0})$. Using this, together with the fact that $f(0)=0$ and \eqref{eq_dofftwoint1}-\eqref{eq_estimlimitfirstintegralformulaexplicit}, in view of \eqref{eq_19i33i323ii20i30i09}, one has
\begin{align*}
-g(x_{0}, t_{0}) &\leq \lim_{\varrho \to 0+} \frac{f(\varrho)- f(0)}{\varrho}+ C^{N+1}\omega_{N-1}\frac{1}{\eta |\ln(\sigma(\varrho))|} = f^{\prime}(0) \\
& = -\frac{2(a-\gamma)}{\beta_{0}}\int_{e^{\perp}}P_{1}(x) \diff \mathcal{H}^{N-1}(x)+ \frac{1}{\beta_{0}}\tr\Biggl(\Biggl(\int_{\mathbb{S}^{N-1} \cap e^{\perp}} \theta \otimes \theta \frac{\diff \mathcal{H}^{N-2}(\theta)}{\norm(\theta)^{N+1}}\Biggr)2(A+\gamma \mathrm{Id})\Biggr).
\end{align*}
Letting $\gamma \to 0+$, we obtain
\[
\partial_{t} \varphi(x_{0}, t_{0}) \leq \mu_{1}(D\varphi(x_{0}, t_{0}))(F_{1}(D^{2}\varphi(x_{0}, t_{0}), D\varphi(x_{0}, t_{0}))+g(x_{0}, t_{0})|D\varphi(x_{0}, t_{0})|),
\]
which is \eqref{def ofsubsolwmc new} with $\alpha=1$.

Next, we assume that $|D\varphi(x_{0}, t_{0})|=0$, $D^{2}\varphi(x_{0}, t_{0})=0$ and $\beta_{h} \to 0$ as $h\to 0+$. We need to distinguish between three further cases. \\
\noindent \emph{Case 2.1}: along some subsequence $\beta_{h}\neq 0$ and $\sigma|\ln(\sigma)|\beta^{-1}_{h}\to 0$. Then, in view of \eqref{eq_19i33i323ii20i30i09} and the fact that $D^{2}\varphi(x_{0}, t_{0})=0$, we have 
\begin{align*}
0&=\lim_{h\to 0+}-g(x_{h}, t_{h}-h)\beta_{h}\leq \lim_{h \to 0+} \beta_{h} \frac{f(\sigma|\ln(\sigma)|)- f(0)}{\sigma|\ln(\sigma)|}=\lim_{\varrho\to 0+}\varrho^{-1}\beta_{h}\mathrm{I}^{1}_{\varrho} +\lim_{\varrho\to 0+}\varrho^{-1} \beta_{h} \mathrm{I}^{2}_{\varrho}\\
& = -2(\partial_{t} \varphi(x_{0}, t_{0})-\gamma)\int_{e^{\perp}}P_{1}(x) \diff \mathcal{H}^{N-1}(x)+ \tr\left(\left(\int_{\mathbb{S}^{N-1} \cap e^{\perp}} \theta \otimes \theta \frac{\diff \mathcal{H}^{N-2}(\theta)}{\norm(\theta)^{N+1}}\right)2(A+\gamma \mathrm{Id})\right).
\end{align*}
Since $A=0$, letting $\gamma \to 0+$, one has $\partial_{t} \varphi(x_{0}, t_{0}) \leq 0$.
\\
\noindent \emph{Case 2.2}: along some subsequence $\beta_{h}=0$ or $\sigma |\ln(\sigma)|\beta_{h}^{-1}\to +\infty$ as $h\to 0+$. Assume that $a_{h}\to\partial_{t} \varphi(x_{0}, t_{0})>0$. Then the characteristic function of the set
\[
\{y \in \mathbb{R}^{N}: \beta_{h} y_{1} -\sigma|\ln(\sigma)|a_{h}+ \sigma \langle \widetilde{A}_{h} y,y \rangle + O((\sigma^{2}|\ln(\sigma)|+\sigma^{2}|y|^{2})(|y|+\sigma|\ln(\sigma)|)) \geq 0\},
\]
which is the same as the set
\[
\{y \in \mathbb{R}^{N}: (\sigma |\ln(\sigma)|)^{-1}\beta_{h} y_{1} -a_{h}+ |\ln(\sigma)|^{-1}\langle \widetilde{A}_{h} y,y \rangle + O((\sigma + \sigma |\ln(\sigma)|^{-1} |y|^{2})(|y|+\sigma|\ln(\sigma)|))\geq 0\},
\]
pointwise converges to the constant function 0. Using this, \eqref{eq 001caseaplha1} and Lebesgue's dominated convergence theorem, we obtain
\[
\int_{\mathbb{R}^{N}} P_{1}(y) \diff y \leq 2\int_{\mathbb{R}^{N}}\mathbbm{1}^{+}(\varphi_{h}(\sigma y, t_{h}-h))P_{1}(y)\diff y +g(x_{h}, t_{h}-h)\sigma |\ln(\sigma)| \to 0
\]
as $h\to 0+$, which leads to a contradiction with the fact that $\int_{\mathbb{R}^{N}}P_{1}(y)\diff y>0$. \\
\noindent \emph{Case 2.3}: along some subsequence $\sigma |\ln(\sigma)|\beta_{h}^{-1}\to \beta>0$. Then the characteristic function of the set
\[
\{y \in \mathbb{R}^{N}: y_{1} + \beta_{h}^{-1}(-\sigma|\ln(\sigma)|a_{h}+ \sigma \langle \widetilde{A}_{h} y,y \rangle + O((\sigma^{2}|\ln(\sigma)|+\sigma^{2}|y|^{2})(|y|+\sigma|\ln(\sigma)|))) \geq 0\}
\]
pointwise converges to the characteristic function of the set 
\[
\{y \in \mathbb{R}^{N}: y_{1}-\beta \partial_{t}\varphi(x_{0}, t_{0}) \geq 0\}.
\]
To obtain a contradiction, we conclude as in the proof of Proposition~\ref{propconsistency1}. This completes our proof of Proposition~\ref{propconsistency2}.
\end{proof}
\subsection{Convexity of the mobility}
In the previous section we established the convergence of the anisotropic scheme of the Bence-Merriman-Osher type  to a viscosity solution of the equation, which can be written in the form 
\[
\partial_{t} u= \Phi_{\alpha}(Du)\left(\frac{1}{|Du|}F_{\alpha}(D^{2}u, Du)+g\right),
\]
where $\Phi_{\alpha}(p)$ is a 1-homogeneous function equal to $\mu_{\alpha}(p)|p|$. Moreover, it turns out that $\Phi_{\alpha}$ is a convex even 1-homogeneous function and, hence, a norm.
\begin{lemma}\label{lem convexitymobility}
For each $\alpha \in (0,2)$, the $1$-homogeneous function
\[
\Phi_{\alpha}(p)=\left(2\int_{p^{\perp}}\frac{\diff \mathcal{H}^{N-1}(x)}{1+\norm(x)^{N+\alpha}}\right)^{-1}|p|
\]
is convex in $\mathbb{R}^{N}$.
\end{lemma}
\begin{proof}
For a proof for the case where $\alpha \in (0,1)$, we refer to \cite[Lemma~3.7]{Chambolle_Novaga_Ruffini_2017}. The proof for the case where $\alpha \in [1,2)$ is similar.
\end{proof}
\begin{rem}\label{rem explicitmobil1234990495}
Let $\alpha \in (0,2)$ and $p \in \mathbb{R}^{N}\setminus \{0\}$. Then
\begin{align*}
\int_{p^{\perp}}\frac{\diff \mathcal{H}^{N-1}(x)}{1+\norm(x)^{N+\alpha}}&=\int_{0}^{+\infty}\int_{\mathbb{S}^{N-1} \cap\, p^{\perp}}\frac{r^{N-2}}{1+\norm(r \theta)^{N+\alpha}}  \diff \mathcal{H}^{N-2}(\theta)\diff r\\
&=\int_{0}^{+\infty}\int_{\mathbb{S}^{N-1} \cap\, p^{\perp}}\frac{r^{N-2}}{1+r^{N+\alpha}\norm(\theta)^{N+\alpha}} \diff \mathcal{H}^{N-2}(\theta)\diff r\\
&=\int_{0}^{+\infty}\frac{t^{N-2}}{1+t^{N+\alpha}} \diff t\int_{\mathbb{S}^{N-1} \cap\, p^{\perp}}\frac{\diff \mathcal{H}^{N-2}(\theta)}{\norm(\theta)^{N-1}}\\
&=:\lambda_{\alpha, N}^{-1}\int_{\mathbb{S}^{N-1} \cap \, p^{\perp}}\frac{\diff \mathcal{H}^{N-2}(\theta)}{\norm(\theta)^{N-1}},
\end{align*}
where we have made the change of variable $t=r\norm(\theta)$. Thus,
\begin{equation}\label{eq_explmobilitycomputah1}
\Phi_{\alpha}(p)=\lambda_{\alpha, N}\left(2\int_{\mathbb{S}^{N-1} \cap\, p^{\perp}}\frac{\diff \mathcal{H}^{N-2}(\theta)}{\norm(\theta)^{N-1}}\right)^{-1}|p|.
\end{equation}
\end{rem}
It is well known that every norm is uniquely determined by its unit ball. Then a natural question arises: knowing the unit ball of the norm $\Phi_{\alpha}$, is it possible to determine the unit ball of the norm $\norm$ ? We answer this question in dimension~2.
\begin{prop}\label{prop premise in 2d}
Let $\alpha \in (0,2)$, $N=2$ and $\lambda_{\alpha}:=\lambda_{\alpha, 2}>0$ be the constant defined in Remark~\ref{rem explicitmobil1234990495}. Then $\Phi_{\alpha}(p_{1}, p_{2})=\lambda_{\alpha}\norm(-p_{2}, p_{1})$.
\end{prop}
\begin{proof}
The proof is a direct consequence of the formula \eqref{eq_explmobilitycomputah1}.
\end{proof}
An immediate consequence of Proposition~\ref{prop premise in 2d} is the next corollary.
\begin{cor}
Let $\alpha \in (0,2)$, $N=2$, $a, b >0$ and $q \in [1,+\infty)\cup \{+\infty\}$. The convex set $\{\Phi_{\alpha}\leq 1\}$ is as regular as the convex set $\{\norm \leq 1\}$, which has at least a Lipschitz boundary. In particular, if $\{\Phi_{\alpha} \leq 1\}=\{(p_{1}, p_{2}): \frac{p^{2}_{1}}{a^{2}}+\frac{p^{2}_{2}}{b^{2}} \leq 1\}$ or $\{\Phi_{\alpha}\leq 1\}=\{\|p\|_{q} \leq 1\}$, then it holds $\{\norm \leq 1\}=\{(p_{1}, p_{2}): \frac{p_{1}^{2}}{b^{2}}+\frac{p_{2}^{2}}{a^{2}} \leq \lambda_{\alpha}^{2}\}$ or $\{\norm\leq 1\}=\{\|p\|_{q} \leq \lambda_{\alpha}\}$, respectively.
\end{cor}

\section{Anisotropic mean curvature motion}
\subsection{Several stability results}
In this subsection, we establish stability results that illustrate the links between the nonlocal anisotropic curvature \eqref{nonlocalcurvature} considered in the isotropic case in \cite{Imbert2009level} and the local anisotropic curvatures appearing in Propositions~\ref{propconsistency1} and \ref{propconsistency2}. For convenience, for each $\alpha \in (0,1)$, we define the measure 
\[
\nu^{\alpha}(d z)=(1-\alpha)\frac{\diff z}{\norm(z)^{N+\alpha}}.
\]
For each $\alpha \in (0,1)$ and for each function $u$ of class $C^{1,1}$ such that $|Du(x)|\neq 0$,  we define the quantities 
\[
\kappa^{\alpha}_{+}[x, u]=\nu^{\alpha}(\{z \in \mathbb{R}^{N}: u(x+z)\geq u(x),\,\ \langle Du(x), z \rangle \leq 0\}),
\]
\[
\kappa^{\alpha}_{-}[x, u]=\nu^{\alpha}(\{z \in \mathbb{R}^{N}: u(x+z)< u(x),\,\ \langle Du(x), z \rangle > 0\})
\]
and
\begin{equation}\label{eq_curvatureoperatoralpha}
\kappa^{\alpha}[x,u]=\kappa^{\alpha}_{+}[x, u]-\kappa^{\alpha}_{-}[x, u].
\end{equation}
According to \cite[Lemma~1]{Imbert2009level} the above quantities are finite. In particular, $\kappa^{\alpha}_{+}[x, u]$ measures how concave the curve $\{z \in \mathbb{R}^{N}: u(x+z)=u(x)\}$ is near $x$ and $\kappa^{\alpha}_{-}[x, u]$ how convexe it is. Moreover, it holds $-(1-\alpha)\kappa_{\alpha}(x,u)=\kappa^{\alpha}[x, u]$ (see \eqref{nonlocalcurvature} and Section~1.2 in \cite{Imbert2009level}). The next proposition is an anisotropic counterpart of \cite[Proposition~2]{Imbert2009level}.
\begin{prop}\label{prop fromfractionaltoocalalpha=1}
Assume that $u \in C^{2}(\mathbb{R}^{N})$ and $|Du(x)|\neq 0$. Then
\begin{equation}\label{}
\kappa^{\alpha}[x, u] \to \frac{1}{2|Du(x)|} \tr\left(\left(\int_{\mathbb{S}^{N-1} \cap \frac{Du(x)}{|Du(x)|}^{\perp}} \theta \otimes \theta \frac{\diff \mathcal{H}^{N-2}(\theta)}{\norm(\theta)^{N+1}}\right)D^{2}u(x)\right)
\end{equation}
as $\alpha \nearrow 1$. 
\end{prop}
\begin{proof}
Since $u \in C^{2}(\mathbb{R}^{N})$, for each $\eta>0$, there exists $\delta>0$ such that for each $z \in B_{\delta}(x)$, 
\begin{equation}\label{eq_secondder98989020948}
\left|u(x+z)-u(x)-\langle Du(x), z\rangle-\frac{1}{2}\langle D^{2} u(x) z, z \rangle\right|\leq \eta |z|^{2}.
\end{equation}
Denote $p=-Du(x)$ and $W(z)=u(x+z)-u(x)-\langle Du(x), z \rangle$. Then 
\begin{equation}\label{eq_operatordecomp23329480}
\begin{split}
\kappa^{\alpha}[x,u]&=\nu^{\alpha}(\{z \in \mathbb{R}^{N}: 0 \leq \langle p, z \rangle \leq W(z)\})-\nu^{\alpha}(\{z \in \mathbb{R}^{N}: W(z) < \langle p, z \rangle <0\})\\
&=(1-\alpha)\int_{\mathbb{R}^{N}}\left[\mathbbm{1}_{\{z \in B_{\delta}:  0 \leq \langle p, z \rangle \leq W(z)\}}-\mathbbm{1}_{\{z \in B_{\delta}: W(z)<\langle p, z \rangle<0\}}\right] \frac{\diff z}{\norm(z)^{N+\alpha}}\\ 
& \,\ \,\ \,\ +O(1-\alpha),
\end{split}
\end{equation}
since \[\nu^{\alpha}(B^{c}_{\delta})\leq \frac{(1-\alpha)C^{N+\alpha}\omega_{N-1}}{\alpha \delta^{\alpha}}.\] In view of \eqref{eq_secondder98989020948} and \eqref{eq_operatordecomp23329480}, it is enough to prove the result in the case where $W(z)=\langle Az, z \rangle$, where $A \in \mathbb{M}^{N\times N}_{\mathrm{sym}}$ (namely, $A=\frac{1}{2} D^{2}u(x)$). Indeed, rewriting \eqref{eq_operatordecomp23329480} with $W(z)$ replaced by $\frac{1}{2}D^{2}u(x)+\eta \mathrm{Id}$ and $\frac{1}{2}D^{2}u(x)-\eta \mathrm{Id}$ and performing computations similar to those given below, one obtains an upper and a lower bound for  the limit of $\kappa^{\alpha}[x,u]$. Then, letting $\eta \to 0+$, one obtains the same result as for the case where $W(z)=\langle A z, z\rangle$. Thus, we study the convergence of 
\[
 K^{\alpha}=(1-\alpha)\int_{\{z \in B_{\delta}:  0 \leq \langle p, z \rangle \leq \langle A z, z \rangle\}} \frac{\diff z}{\norm(z)^{N+\alpha}}- (1-\alpha)\int_{\{z \in B_{\delta}: \langle Az,z \rangle <\langle p, z \rangle<0\}}\frac{\diff z}{\norm(z)^{N+\alpha}}.
 \]
 For each $z \in \mathbb{R}^{N}$, we have $z=(z_{1}, z^{\prime}),$ where $z^{\prime} \in \mathbb{R}^{N-1}$. Let $E \in SO(N)$ be such that $E^{\mathrm{T}}(1, 0^{\prime})=\frac{Du(x)}{|Du(x)|}$. Performing the rotation and the change of variables, we can assume that $p=|p|(1,0,\dotsc, 0)$. We denote by $\widetilde{A}$ the matrix we obtain after the rotation, namely $\widetilde{A}=EAE^{\mathrm{T}}$. Then 
 \begin{equation}\label{eq_0j400tjt034ij034}
 \langle \widetilde{A} z, z\rangle =\widetilde{a}_{1}z^{2}_{1}+2z_{1}\langle \widetilde{a}^{\prime}, z^{\prime}\rangle +\langle \widetilde{A}^{\prime}z^{\prime}, z^{\prime}\rangle, 
 \end{equation}
where $\widetilde{a}=(\widetilde{a}_{1}, \widetilde{a}^{\prime})$ is the first column of $\widetilde{A}$ and  $\widetilde{A}^{\prime} \in \mathbb{M}^{(N-1)\times (N-1)}_{\mathrm{sym}}$. If $\delta>0$ is small enough, using \eqref{eq_0j400tjt034ij034}, for each $z \in B_{\delta}$, we have
 \[
 0\leq \langle p, z\rangle =|p|z_{1} \leq \langle \widetilde{A}z, z \rangle \Rightarrow 0 \leq z_{1} \leq (|p|-C\delta)^{-1} \langle \widetilde{A}^{\prime} z^{\prime}, z^{\prime}\rangle,
 \] 
 \[
 (|p|+C\delta)^{-1}\langle \widetilde{A}^{\prime} z^{\prime}, z^{\prime}\rangle <z_{1}< 0 \Rightarrow \langle \widetilde{A} z, z\rangle < \langle p, z \rangle= |p|z_{1} <0
 \]
 and
 \[
 0 \leq z_{1} \leq (|p|+C\delta)^{-1}\langle \widetilde{A}^{\prime} z^{\prime}, z^{\prime}\rangle \Rightarrow 0 \leq |p| z_{1}=\langle p, z\rangle \leq \langle \widetilde{A}z, z \rangle,
 \]
 \[
 \langle \widetilde{A} z, z \rangle < \langle p, z \rangle=|p|z_{1}  <0 \Rightarrow (|p|-C\delta)^{-1} \langle \widetilde{A}^{\prime} z^{\prime}, z^{\prime} \rangle < z_{1}<0,
 \]
 where $C=C(|Du(x)|, D^{2}u(x))>0$. This implies that 
 \[
 K^{\alpha}_{\delta,-}\leq K^{\alpha} \leq K^{\alpha}_{\delta, +},
 \]
 where 
 \begin{align*}
 K^{\alpha}_{\delta,+}&=\nu^{\alpha}(\{z \in B_{\delta}: 0 \leq z_{1} \leq (|p|-C\delta)^{-1}\langle \widetilde{A}^{\prime}z^{\prime}, z^{\prime}\rangle \}) \\ & \,\ \,\ \,\ -\nu^{\alpha}(\{z \in B_{\delta}: (|p|+C\delta)^{-1}\langle \widetilde{A}^{\prime} z^{\prime}, z^{\prime}\rangle <z_{1}< 0\})
 \end{align*}
 and 
 \begin{align*}
 K^{\alpha}_{\delta,-}&=\nu^{\alpha}(\{z \in B_{\delta}: 0 \leq z_{1} \leq (|p|+C\delta)^{-1}\langle \widetilde{A}^{\prime}z^{\prime}, z^{\prime}\rangle \})\\ & \,\ \,\ \,\ -\nu^{\alpha}(\{z \in B_{\delta}: (|p|-C\delta)^{-1}\langle \widetilde{A}^{\prime} z^{\prime}, z^{\prime}\rangle <z_{1}< 0\}).
 \end{align*}
Letting $\delta \to 0+$, we observe that it is enough to study the convergence of 
\[
K^{\alpha}_{\delta}=\nu^{\alpha}(\{z \in B_{\delta}: 0 \leq z_{1} \leq |p|^{-1}\langle \widetilde{A}^{\prime}z^{\prime}, z^{\prime}\rangle \})-\nu^{\alpha}(\{z \in B_{\delta}: |p|^{-1}\langle \widetilde{A}^{\prime} z^{\prime}, z^{\prime}\rangle <z_{1}< 0\}).
\]
Using \cite[Theorem~3.2.22~(3)]{Federer}, if $\delta>0$ is small enough, we obtain 
\begin{align*}
K^{\alpha}_{\delta}&=(1-\alpha)\int_{\{(z_{1}, z^{\prime}): |z^{\prime}|<\delta,\, 0 \leq z_{1} \leq |p|^{-1}\langle \widetilde{A}^{\prime}z^{\prime}, z^{\prime}\rangle \}}\frac{\diff z}{\norm(z_{1}, z^{\prime})^{N+\alpha}}\\ & \,\ \,\ \,\  -(1-\alpha)\int_{\{(z_{1}, z^{\prime}): |z^{\prime}|<\delta,\, |p|^{-1}\langle \widetilde{A}^{\prime} z^{\prime}, z^{\prime}\rangle <z_{1}< 0\}}\frac{\diff z}{\norm(z_{1},z^{\prime})^{N+\alpha}}\\
& = (1-\alpha)\int_{\mathbb{S}^{N-2} \cap\{\langle \widetilde{A}^{\prime} \theta, \theta \rangle \geq 0\}}  \int_{0}^{\delta} \int_{0}^{|p|^{-1}r^{2}\langle \widetilde{A}^{\prime} \theta, \theta \rangle} \frac{r^{N-2}}{\norm(z_{1}, r\theta)^{N+\alpha}}\diff z_{1} \diff r \diff \mathcal{H}^{N-2}(\theta) \\ & \,\  \,\  \,\ \,\  -(1-\alpha)\int_{\mathbb{S}^{N-2} \cap\{\langle \widetilde{A}^{\prime} \theta, \theta \rangle < 0\}}  \int_{0}^{\delta} \int^{0}_{|p|^{-1}r^{2}\langle \widetilde{A}^{\prime} \theta, \theta \rangle} \frac{r^{N-2}}{\norm(z_{1}, r\theta)^{N+\alpha}}\diff z_{1} \diff r \diff \mathcal{H}^{N-2}(\theta) \\
&=(1-\alpha)\int_{\mathbb{S}^{N-2} \cap\{\langle \widetilde{A}^{\prime} \theta, \theta \rangle \geq 0\}}  \int_{0}^{\delta} \int_{0}^{|p|^{-1}\langle \widetilde{A}^{\prime} \theta, \theta \rangle} \frac{r^{N}}{\norm(r^{2}\tau, r\theta)^{N+\alpha}} \diff \tau \diff r \diff \mathcal{H}^{N-2}(\theta) \\ & \,\  \,\  \,\   -(1-\alpha)\int_{\mathbb{S}^{N-2} \cap\{\langle \widetilde{A}^{\prime} \theta, \theta \rangle < 0\}}  \int_{0}^{\delta} \int^{0}_{|p|^{-1}\langle \widetilde{A}^{\prime} \theta, \theta \rangle} \frac{r^{N}}{\norm(r^{2}\tau, r\theta)^{N+\alpha}} \diff \tau \diff r\diff \mathcal{H}^{N-2}(\theta) \\
&=(1-\alpha)\int_{\mathbb{S}^{N-2} \cap\{\langle \widetilde{A}^{\prime} \theta, \theta \rangle \geq 0\}}  \int_{0}^{\delta} r^{-\alpha} \int_{0}^{|p|^{-1}\langle \widetilde{A}^{\prime} \theta, \theta \rangle} \frac{1}{\norm(r\tau, \theta)^{N+\alpha}} \diff \tau \diff r \diff \mathcal{H}^{N-2}(\theta) \\ & \,\  \,\  \,\  -(1-\alpha)\int_{\mathbb{S}^{N-2} \cap\{\langle \widetilde{A}^{\prime} \theta, \theta \rangle < 0\}}  \int_{0}^{\delta} r^{-\alpha} \int^{0}_{|p|^{-1}\langle \widetilde{A}^{\prime} \theta, \theta \rangle} \frac{1}{\norm(r\tau, \theta)^{N+\alpha}} \diff \tau \diff r\diff \mathcal{H}^{N-2}(\theta) .
\end{align*}
We observe that for each $r \in (0, \delta)$,
\[
\int_{0}^{|p|^{-1}\langle \widetilde{A}^{\prime} \theta, \theta \rangle} \frac{1}{\norm(r\tau, \theta)^{N+\alpha}} \diff \tau \to \frac{ |p|^{-1}\langle \widetilde{A}^{\prime} \theta, \theta \rangle}{\norm(0, \theta)^{N+\alpha}}
\]
as $\delta \to 0+$. In particular, for each $\varepsilon \in (0,1)$, there exists $\delta>0$ such that
\[
(1-\varepsilon)\frac{ |p|^{-1}\langle \widetilde{A}^{\prime} \theta, \theta \rangle}{\norm(0, \theta)^{N+\alpha}}
 \leq  \int_{0}^{|p|^{-1}\langle \widetilde{A}^{\prime} \theta, \theta \rangle} \frac{1}{\norm(r\tau, \theta)^{N+\alpha}} \diff \tau \leq (1+\varepsilon) \frac{ |p|^{-1}\langle \widetilde{A}^{\prime} \theta, \theta \rangle}{\norm(0, \theta)^{N+\alpha}}.
 \]
 Using this, together with the fact that
 \[
 (1-\alpha)\int_{0}^{\delta}r^{-\alpha} \diff r= \delta^{1-\alpha}
 \]
 and letting $\alpha \nearrow 1$ and then $\delta \searrow 0$, we deduce that
 \[
 K^{\alpha}_{\delta}\to \int_{\mathbb{S}^{N-2}} \frac{|p|^{-1}\langle \widetilde{A}^{\prime} \theta, \theta \rangle}{\norm(0, \theta)^{N+1}} \diff \mathcal{H}^{N-2}(\theta) .
 \]
Altogether, recalling that $p=-Du(x)$ and $A=\frac{1}{2}D^{2}u(x)$ and $\widetilde{A}=EAE^{\mathrm{T}}$, we obtain
\[
\kappa^{\alpha}[x,u]\to \frac{1}{2|Du(x)|} \int _{\mathbb{S}^{N-1} \cap \frac{Du(x)}{|Du(x)|}^{\perp}} \frac{\langle D^{2}u(x) \theta, \theta \rangle}{\norm(\theta)^{N+1}} \diff \mathcal{H}^{N-2}(\theta)
\]
as $\alpha \nearrow 1$. This completes our proof of Proposition~\ref{prop fromfractionaltoocalalpha=1}.
\end{proof}
\begin{rem}\label{rem ccmc}
It is worth noting that \[(\alpha-1)C_{N, \alpha}=(\alpha-1)\int_{0}^{+\infty} \frac{t^{N}}{1+t^{N+\alpha}} \diff t \to 1\] as $\alpha \searrow 1$, where $C_{N, \alpha}$ is defined in \eqref{eq_constcominganalphan}. Thus, $(\alpha-1)F_{\alpha}\to F_{1}$ as $\alpha \searrow 1$, where $F_{\alpha}$ is defined in \eqref{eq_defofanisotropicmeancurvoppdevmodalpha}.
\end{rem}
Next, we state two convergence results demonstrating how one can recover the anisotropic local mean curvature flow in the limit. The first of these appeared in \cite{DaLio_Forcadel_Monneau}, where it was shown that the solution of the nonlocal (eikonal) Hamilton-Jacobi equation modeling dislocation dynamics converges, at a large scale, to the solution of a local anisotropic mean curvature motion. The second result can be proved using Proposition~\ref{prop fromfractionaltoocalalpha=1}.
\begin{theorem}\label{thmconvergenceofviscsolstability1}
Given a Lipschitz function $u_{0}:\mathbb{R}^{N} \to \mathbb{R}$ and an even nonnegative function $c_{0} \in W^{1,1}(\mathbb{R}^{N})$ such that $c_{0}(z)=\frac{1}{|z|^{N+1}}\norm(\frac{z}{|z|})$ if $|z|\geq 1$, we consider the viscosity solution $u^{\varepsilon}$ of the problem
\begin{equation*}
\partial_{t}u = \kappa^{\varepsilon}[x,u]|Du| \,\ \text{in}\,\ \mathbb{R}^{N}\times (0,+\infty)\\
\end{equation*}
supplemented with the initial condition $u(0,x)=u_{0}(x)$ in $\mathbb{R}^{N}$,
where $\kappa^{\varepsilon}[x,u]$ is defined by
\begin{align*}
\kappa^{\varepsilon}[x,u]&=\nu^{\varepsilon}(\{z \in \mathbb{R}^{N}: u(x+z)\geq u(x),\, \langle Du(x), z\rangle \leq 0\})\\ & \,\ \,\ \,\ -\nu^{\varepsilon}(\{z \in \mathbb{R}^{N}: u(x+z)< u(x),\, \langle Du(x), z\rangle > 0\})
\end{align*}
with $\nu^{\varepsilon}(dz)=\frac{1}{\varepsilon^{N+1}|\ln(\varepsilon)|}c_{0}(\frac{z}{\varepsilon}) \diff z$. Then $u^{\varepsilon}$ converges locally uniformly on compact sets in $\mathbb{R}^{N}\times [0, +\infty)$ to the unique solution $u^{0}$ of \[\partial_{t} u=F_{1}(D^{2}u, Du) \,\ \text{in}\,\ \mathbb{R}^{N}\times (0,+\infty)\] supplemented with the initial condition $u(x,0)=u_{0}(x)$ in $\mathbb{R}^{N}$ as $\varepsilon \searrow 0$.
\end{theorem}
\begin{proof}
For a proof, the reader may consult \cite[Theorem~1.4]{DaLio_Forcadel_Monneau} and \cite[Lemma~2]{Imbert2009level}.
\end{proof}
\begin{theorem}\label{thmconvergenceofviscsolstability2}
Given $\alpha \in (0,1)$, a Lipschitz function $u_{0} :\mathbb{R}^{N} \to \mathbb{R}$, we consider the viscosity solution $u_{\alpha}$ of the problem
\[
\partial_{t}u= \mu_{\alpha}(Du)|Du| \kappa^{\alpha}[x,u] \,\ \text{in}\,\ \mathbb{R}^{N}\times (0,+\infty)
\]
supplemented with the initial condition $u(x,0)=u_{0}(x)$ in $\mathbb{R}^{N}$, where $\kappa^{\alpha}$ is defined in \eqref{eq_curvatureoperatoralpha} and 
$\mu_{\alpha}$ is defined as in \eqref{eq_defofmobilitygeneralapproachintro} with $\alpha \in (0,1)$. Then $u_{\alpha}$ converges locally uniformly on compact sets in $\mathbb{R}^{N}\times [0, +\infty)$ to a unique solution $u_{1}$ of \eqref{weighted MC equation}, where $\alpha=1$ and $g=0$, supplemented with the initial condition $u(x,0)=u_{0}(x)$ in $\mathbb{R}^{N}$ as $\alpha \nearrow 1$.
\end{theorem}
\begin{rem}\label{rem fblalphato1plus}
In view of Remark~\ref{rem ccmc}, the viscosity solution $u_{\alpha}$ of the problem 
\[
\partial_{t} u = \mu_{\alpha}(Du)(\alpha-1)F_{\alpha}(D^{2}u, Du) \,\ \text{in}\,\ \mathbb{R}^{N} \times (0,+\infty)
\]
supplemented with the initial condition $u(x,0)=u_{0}(x)$ in $\mathbb{R}^{N}$ for some Lipschitz function $u_{0}:\mathbb{R}^{N} \to \mathbb{R}$, converges locally uniformly to a unique solution $u_{1}$ of \eqref{weighted MC equation}, where $\alpha=1$ and $g=0$, supplemented with the initial condition $u(x,0)=u_{0}(x)$ in $\mathbb{R}^{N}$, as $\alpha \searrow 1$.
\end{rem}
\subsection{Variational origin of anisotropic mean curvature motion}
The anisotropic mean curvature motion \eqref{weighted MC equation}, where $g=0$, is of a variational type. To state the result, we associate to $\alpha \in [1,2)$ and $\norm$ a tempered distribution $L_{\alpha, \norm}$ defined by 
\[
\langle L_{\alpha, \norm}, \varphi \rangle = 2C_{N,\alpha}\int_{\mathbb{R}^{N}}(\varphi(x)-\varphi(0)-\langle D\varphi(0), x \rangle \mathbbm{1}_{B_{1}}(x))\frac{\diff x}{\norm(x)^{N+1}},
\]
for $\varphi \in \mathcal{S}(\mathbb{R}^{N})$, where $C_{N, \alpha}>0$ is the constant defined in \eqref{eq_constcominganalphan} and $\mathcal{S}(\mathbb{R}^{N})$ is the Schwartz space of test functions. We define the Fourier transform of $\varphi \in \mathcal{S}(\mathbb{R}^{N})$ by 
\[
\mathcal{F}(\varphi)(\xi)=\int_{\mathbb{R}^{N}}\varphi(x)e^{-i\langle \xi, x\rangle} \diff x.
\]
\begin{theorem}\label{thm variationaltypeorigin}
For each $\alpha \in [1,2)$, there exists a unique $\psi_{\alpha} \in C(\mathbb{R}^{N})\cap C^{2}(\mathbb{R}^{N}\setminus \{0\})$ such that $\psi_{\alpha}(-p)=\psi_{\alpha}(p)$, $\psi_{\alpha}(0)=0$ and 
\[
F_{\alpha}(M, p)=\tr(MD^{2}\psi_{\alpha}(p)),
\]
where $F_{\alpha}(M,p)$ is defined in \eqref{eq_defofanisotropicmeancurvoppdevmodalpha}.
Moreover, $\psi_{\alpha}$ is convex, $\psi_{\alpha}(\lambda p)=|\lambda|\psi_{\alpha}(p)$ for all $\lambda \in \mathbb{R}\setminus \{0\}$ and $\psi_{\alpha}=-\frac{1}{2\pi}\mathcal{F}(L_{\alpha, \norm})$, where $\mathcal{F}(L_{\alpha, \norm})$ is the Fourier transform of $L_{\alpha,\norm}$. If $u \in C^{2}(\mathbb{R}^{N})$ with $|Du|\neq 0$, then
\[
 F_{\alpha}(D^{2}u, Du)=|Du|\mathrm{div}\left(\nabla \psi_{\alpha}\left(\frac{Du}{|Du|}\right)\right),
\]
which means that the anisotropic mean curvature motion derives from the energy $\int \psi_{\alpha}(Du)$.
\end{theorem}
\begin{proof}
For a proof, we refer to Section~7 of \cite{DaLio_Forcadel_Monneau} and in particular to the proof of \cite[Theorem~1.7]{DaLio_Forcadel_Monneau}.
\end{proof}
\section{Evolution of convex sets}
In this section, under some convexity assumptions on the external force $g$, we show that during the  anisotropic mean curvature flow, which we obtain at the limit of our anisotropic version of the Bence-Merriman-Osher type scheme, the convexity of the set $\Omega_{0}$ is preserved. Namely, at each step of the discrete approximation, the convexity is preserved (see Corollary~\ref{cor convexpresflowd}), and hence it is preserved at the limit.

If $q \in \mathbb{R}\setminus \{0\}$, $a, b \geq 0$ and $\lambda \in [0,1]$, we define
\begin{equation}\label{eq_pconcaveqquant}
M_{q}(a,b,\lambda)=((1-\lambda)a^{q}+\lambda b^{q})^{\frac{1}{q}}
\end{equation}
if $a, b>0$ and $M_{q}(a, b, \lambda)=0$ if $ab=0$. We also define
\begin{equation}\label{eq_pconcaveqquant0}
M_{0}(a,b, \lambda)=a^{1-\lambda} b^{\lambda}.
\end{equation}
For convenience, we recall the following definition. 
\begin{defn}\label{def pconncave}
A nonnegative function $f$ on $\mathbb{R}^{N}$ is called $q$-concave on a convex set $E$ if 
\[
f((1-\lambda)x+\lambda y)\geq M_{q}(f(x), f(y), \lambda)
\]
for all $x, y \in E$ and $\lambda \in [0,1]$, where  $M_{q}(a, b, \lambda)$ is defined in \eqref{eq_pconcaveqquant} and \eqref{eq_pconcaveqquant0}.
\end{defn}
It is worth noting that if $q>0$ (respectively, $q<0$), then $f$ is $q$-concave if and only if $f^{q}$ is concave (respectively, convex), and in particular, $1$-concave is just concave in the usual sense (see \cite[Section~9]{Gardner}). If $q=0$ and $f$ is positive, then $f$ is 0-concave if and only if  $\ln(f)$ is concave. If $f$ is positive, then $f$ is $-1$-concave if and only if $f^{-1}$ is convex. 

We recall the following result, which can be proved using \cite[Corollary~11.2]{Gardner} (see \cite[p.~379]{Gardner}).
\begin{prop}\label{prop pconcaveconv}
Let $q \geq -1/N$, $f \in L^{1}(\mathbb{R}^{N})$ be a $q$-concave function on $\mathbb{R}^{N}$ and $K \subset \mathbb{R}^{N}$ be a convex set with nonempty interior. Then $f*\mathbbm{1}_{K}$ is $q/(Nq +1)$-concave on $\mathbb{R}^{N}$.
\end{prop}
\begin{cor}\label{cor convexpresflowd}
Let $\Omega_{0} \subset \mathbb{R}^{N}$ be an open convex set, $\alpha \in [1,2)$ and $g \in C([0,+\infty))$. Then for each $h>0$ and for each $n \in \mathbb{N}$, the set $\Omega^{h}_{nh}$ defined in \eqref{eq_nsetdiscreteh} is convex. 
\end{cor}
\begin{rem}\label{rem convexity of extforceg}
Let us comment on the importance of the assumption that the external force $g$ depends only on time for the proof of Corollary~\ref{cor convexpresflowd}. Assume that $g \in C(\mathbb{R}^{N}\times [0,+\infty))$. Then to prove Corollary~\ref{cor convexpresflowd} we need the fact that $\max\{\|J_{h}\|_{L^{1}(\mathbb{R}^{N})}-g_{h}\beta(\alpha,h), 0\}^{-\frac{1}{\alpha}}$ is a concave function on $\mathbb{R}^{N}$ for each $h \in [0,+\infty)$, where $J_{h}$ is defined in \eqref{defofkernelJh}, \eqref{formulaforkernelJh}, $\beta(\alpha, h)$ is defined in \eqref{eq_betathresholdalphatime} and $g_{h}(\cdot)=g(\cdot, h)$. Since a nonnegative concave function on $\mathbb{R}^{N}$ is a constant, we deduce that for each fixed $h\in [0,+\infty)$,  $g_{h}:\mathbb{R}^{N}\to \mathbb{R}$ is a constant function. This implies our condition on $g$, namely, the assumption that $g \in C([0,+\infty))$. A proof in the case where $g \in C(\mathbb{R}^{N}\times [0,+\infty))$ would require stronger convexity properties of the kernel $J_{h}$.
\end{rem}
\begin{proof}[Proof of Corollary~\ref{cor convexpresflowd}]
Since the function $\theta\mapsto (\sigma_{\alpha}(h)^{\frac{N+\alpha}{\alpha}}+\theta^{N+\alpha})^{\frac{1}{N+\alpha}}$ is nondecreasing and convex on $[0, +\infty)$ and since $\norm$ is convex, the function $x \mapsto (\sigma_{\alpha}(h)^{\frac{N+\alpha}{\alpha}}+\norm(x)^{N+\alpha})^{\frac{1}{N+\alpha}}$ is convex on $\mathbb{R}^{N}$, and hence the function $J_{h}$ is $-1/(N+\alpha)$-concave. Then, according to Proposition~\ref{prop pconcaveconv}, $J_{h}*\mathbbm{1}_{\Omega_{0}}$ is $-1/\alpha$-concave. In view of the facts that $J_{h}*[\mathbbm{1}_{\Omega_{0}}-\mathbbm{1}_{\smash{\overline{\Omega}}^{c}_{0}}]=2J_{h}*\mathbbm{1}_{\Omega_{0}}-\|J_{h}\|_{L^{1}(\mathbb{R}^{N})}$ and $\Omega^{h}_{h}=\{J_{h}*[\mathbbm{1}_{\Omega_{0}}-\mathbbm{1}_{\smash{\overline{\Omega}}^{c}_{0}}] \geq -g(0)\beta(\alpha, h)\}$, we have \[\Omega^{h}_{h}=\{(2J_{h}*\mathbbm{1}_{\Omega_{0}})^{-\frac{1}{\alpha}}\leq \max\{\|J_{h}\|_{L^{1}(\mathbb{R}^{N})}-g(h)\beta(\alpha,h), 0\}^{-\frac{1}{\alpha}}\}.\] Observing that the function $(2J_{h}*\mathbbm{1}_{\Omega_{0}})^{-\frac{1}{\alpha}}- \max\{\|J_{h}\|_{L^{1}(\mathbb{R}^{N})}-g(h)\beta(\alpha, h), 0\}^{-\frac{1}{\alpha}}$ is convex (as a sum of convex functions), we deduce that $\Omega^{h}_{h}$ is an open convex set. Iterating this procedure, we deduce that $\Omega^{h}_{nh}$ is an open convex set, which completes our proof of Corollary~\ref{cor convexpresflowd}.
\end{proof}
In view of Corollary~\ref{cor convexpresflowd} and the convergence of the front propagation, we obtain the following result. 
\begin{cor}\label{cor convexflowlimitcontinuous}
Let $\alpha \in [1,2)$, $g \in C([0,+\infty))$ and $u_{0}:\mathbb{R}^{N}\to \mathbb{R}$ be  uniformly continuous. Assume that for each $s \in \mathbb{R}$, the set $\{u_{0}(\cdot)>s\}$ is convex, namely, $u_{0}$ is quasiconcave. Then if $u$ is a unique  viscosity solution of \eqref{weighted MC equation} supplemented with the initial condition $u(\cdot, 0)=u_{0}(\cdot)$ in $\mathbb{R}^{N}$, then the sets $\{u(\cdot, t)\geq s\}$ are convex for each $t \geq 0$. 
\end{cor}
\begin{proof}[Proof of Corollary~\ref{cor convexflowlimitcontinuous}]
If $t=0$, then $\{u(\cdot, 0)\geq s\}=\{u_{0}(\cdot) \geq s\}$ is convex,  since $u_{0}$ is quasiconcave. Let $t>0$, $x, y \in \{u(\cdot, t) \geq s\}$. Assume that there exists $z \in [x,y]$ such that $u(z,t)<s$. Let $\delta>0$ be such that $u(z, t)< s-\delta$. Define $\Omega_{0}=\{u_{0}(\cdot)>s-\delta/2\}$. Then $\Omega_{0}$ is convex by assumption. According to Corollary~\ref{cor convexpresflowd}, $\Omega^{h}_{nh}$ is convex for each $h>0$ and for each $n\in \mathbb{N}$. In view of Theorem~\ref{mainthm}, $u_{h}\to 1$ locally uniformly in $\Omega_{t}$ as $h \to 0+$. Since $x, y \in \Omega_{t}$, $u_{h}(\cdot, nh)=1$ locally around $x$ and $y$ for each $h>0$ small enough, where $nh\to t$ as $h\to 0+$. This, in view of the convexity of $\Omega^{h}_{nh}$, implies that there exists an open convex set containing $[x,y]$ whose closure is a subset of  $\Omega^{h}_{nh}$ for each $h>0$ small enough, where $nh\to t$. This yields that $u_{h}(\cdot,nh)=1$ locally around $z$ for each $h>0$ small enough, where $nh \to t$ as $h \to 0+$. Since $\sign^{*}(u(\cdot,\cdot)-s+\delta/2)$ is the maximal upper semicontinuous subsolution of \eqref{weighted MC equation} supplemented with the initial datum $\mathbbm{1}_{\smash{\overline{\Omega}}_{0}}-\mathbbm{1}_{\smash{\overline{\Omega}}_{0}^{c}}$ (see \cite{Barles-Soner-Souganidis}), using Proposition~\ref{prop visc subandsupersol}, we deduce that $1=\limsup^{*}u_{h}(z,t)\leq \sign^{*}(u(z,t)-s+\delta/2)$. Thus, $u(z,t)-s+\delta/2\geq 0$ and $u(z,t)\geq s-\delta/2>s-\delta$, which leads to a contradiction with the fact that $u(z, t)<s-\delta$. This completes our proof of Corollary~\ref{cor convexflowlimitcontinuous}. 
\end{proof}
\section{Splitting the flow}\label{Section 6}
In this section, we show that a local anisotropic mean curvature flow with a forcing term depending only on time can be obtained
by alternating local anisotropic mean curvature flows without a forcing term and evolutions with only a forcing term. Using this,
we shall see how the distance between two sets evolves under the action of a forced local anisotropic mean curvature flow. A similar result was obtained in \cite{Chambolle_Novaga_Ruffini_2017} for a \emph{nonlocal} anisotropic mean curvature flow.

For each $\varepsilon>0$, we consider the sets $E_{\varepsilon}=\bigcup_{n \in \mathbb{N}} (2n\varepsilon, (2n+1)\varepsilon]$ and $O_{\varepsilon}=(0,+\infty)\setminus E_{\varepsilon}$. Given $\alpha \in [1,2)$, $g \in C([0,+\infty))$, $t>0$, $p \in \mathbb{R}^{N} \setminus \{0\}$ and $s \in \mathbb{R}$, we define
\[
H^{\varepsilon}_{\alpha}(t, p, s)=2\mathbbm{1}_{E_{\varepsilon}}(t)\Phi_{\alpha}(p)c_{\varepsilon}(t)+2\mathbbm{1}_{O_{\varepsilon}}(t)\mu_{\alpha}(p)s,
\]
where $\Phi_{\alpha}(p)=\mu_{\alpha}(p)|p|$ (see also \eqref{eq_explmobilitycomputah1}) and $c_{\varepsilon}:[0,+\infty)\to \mathbb{R}$ is defined by
\[
c_{\varepsilon}(t)=\frac{1}{2\varepsilon}\int_{2n\varepsilon}^{2(n+1)\varepsilon} g(\tau) \diff \tau
\]
if $t \in (2n\varepsilon, 2(n+1)\varepsilon]$ for each $n \in \mathbb{N}$. We also define
\[
H_{\alpha}(t, p, s)=\mu_{\alpha}(p)s+\Phi_{\alpha}(p)g(t).
\]
For fixed $p$ and $s$, we observe that $t \mapsto \int_{0}^{t}(H^{\varepsilon}_{\alpha}(\tau,p,s)-H_{\alpha}(\tau, p, s)) \diff \tau \to 0$ locally uniformly on $[0, +\infty)$ as $\varepsilon \to 0+$. Let $u_{0}:\mathbb{R}^{N}\to \mathbb{R}$ be a uniformly continuous function and construct the function $u_{\varepsilon}:\mathbb{R}^{N}\to \mathbb{R}$ as follows. Let $u_{\varepsilon}(\cdot, 0)=u_{0}(\cdot)$ in $\mathbb{R}^{N}$ and for each $n\in \mathbb{N}$, define $u_{\varepsilon}$ on $\mathbb{R}^{N}\times [n\varepsilon, (n+1)\varepsilon]$ as the unique viscosity solution of the equation
\begin{equation}\label{eq_viscsolutionsplit}
\partial_{t}u=H^{\varepsilon}_{\alpha}(t, Du, F_{\alpha}(D^{2}u, Du)) \,\ \text{in}\,\ \mathbb{R}^{N}\times [n\varepsilon, (n+1)\varepsilon]
\end{equation}
supplemented with the initial condition $u(\cdot, n\varepsilon)=u_{\varepsilon}(\cdot, n\varepsilon)$, where $F_{\alpha}(D^{2}u, Du)$ is defined in \eqref{eq_defofanisotropicmeancurvoppdevmodalpha}. 
\begin{prop}\label{prop approachingmcf}
Let $u_{0}:\mathbb{R}^{N} \to \mathbb{R}$ be uniformly continuous and $u$ be a unique viscosity solution of the equation \eqref{weighted MC equation} supplemented with the initial condition $u(\cdot, 0)=u_{0}(\cdot)$ in $\mathbb{R}^{N}$. Let $u_{\varepsilon}:\mathbb{R}^{N}\to \mathbb{R}$ be the uniformly continuous function such that $u_{\varepsilon}(\cdot, 0)=u_{0}(\cdot)$ in $\mathbb{R}^{N}$ and $u_{\varepsilon}$ is a unique viscosity solution of the equation \eqref{eq_viscsolutionsplit} for each $n \in \mathbb{N}$. Then $u_{\varepsilon}\to u$ locally uniformly on $\mathbb{R}^{N}\times [0,+\infty)$ as $\varepsilon \to 0+$. 
\end{prop}
\begin{proof}
Since $u_{\varepsilon}$ is uniformly continuous, up to a subsequence (not relabeled), we can assume that $u_{\varepsilon}(x,t)\to v(x,t)$ locally uniformly as $\varepsilon \to 0+$  for some $v \in C(\mathbb{R}^{N} \times [0,+\infty))$. We shall prove that $v$ is a viscosity solution of the equation \eqref{weighted MC equation} supplemented with the initial condition $v(\cdot, 0)=u_{0}(\cdot)$ in $\mathbb{R}^{N}$. Since the latter admits a unique viscosity solution $u$, $v$ does not depend on a subsequence and $u_{\varepsilon}\to u=v$ locally uniformly as $\varepsilon \to 0+$.

We shall only prove that $v$ is a subsolution of \eqref{weighted MC equation} (see Definition~\ref{def ofsubsolwmc} and Theorem~\ref{thm equivdefvs}), since the proof that $v$ is a supersolution of \eqref{weighted MC equation} is similar. We follow the strategy of \cite{Barles_2006}. 

Let $\varphi$ be a smooth test function and $(x_{0}, t_{0}) \in \mathbb{R}^{N}\times (0,+\infty)$ be a strict global maximum of $v-\varphi$. Assume that $|D\varphi(x_{0}, t_{0})|\neq 0$. We define
\begin{align*}
\psi_{\varepsilon}(t)&=H_{\alpha}^{\varepsilon}(t, D\varphi(x_{0}, t_{0}), F_{\alpha}(D^{2}\varphi(x_{0}, t_{0}),D\varphi(x_{0}, t_{0})))\\ & \,\ \,\ \,\ -H_{\alpha}(t, D\varphi(x_{0}, t_{0}), F_{\alpha}(D^{2}\varphi(x_{0}, t_{0}),D\varphi(x_{0}, t_{0})))
\end{align*}
and observe that $\int^{t}_{0}\psi_{\varepsilon}(\tau) \diff \tau\to 0$ locally uniformly as $\varepsilon \to 0+$. This, since $u_{\varepsilon}\to v$ locally uniformly as $\varepsilon \to 0+$, implies that $u_{\varepsilon}(x,t)-\int_{0}^{t}\psi_{\varepsilon}(\tau)\diff \tau \to v(x,t)$ uniformly on compact subsets of $\mathbb{R}^{N}\times [0, +\infty)$. Then there exist points $(x_{\varepsilon}, t_{\varepsilon}) \in \mathbb{R}^{N}\times (0,+\infty)$ of global maximum of the function $u_{\varepsilon}(x,t)-\int_{0}^{t}\psi_{\varepsilon}(\tau)\diff \tau -\varphi(x,t)$ such that $(x_{\varepsilon}, t_{\varepsilon}) \to (x_{0}, t_{0})$ as $\varepsilon \to 0+$. If $t_{\varepsilon}/\varepsilon \in \mathbb{N}$, then, since $u_{\varepsilon}$ is a viscosity solution and $|D\varphi(x_{\varepsilon}, t_{\varepsilon})|\neq 0$ for each sufficiently small $\varepsilon>0$, 
\begin{equation}\label{eq_uhguhuh3ht9783h8934}
\partial_{t}\varphi(x_{\varepsilon}, t_{\varepsilon})+\psi_{\varepsilon}(t_{\varepsilon}) \leq H_{\alpha}^{\varepsilon}(t_{\varepsilon}, D\varphi(x_{\varepsilon}, t_{\varepsilon}), F_{\alpha}(D^{2}\varphi(x_{\varepsilon}, t_{\varepsilon}),D\varphi(x_{\varepsilon}, t_{\varepsilon}))).
\end{equation}
If $t_{\varepsilon}/\varepsilon \in \mathbb{N}$, then replacing $\psi_{\varepsilon}$ and $H_{\alpha}^{\varepsilon}$ by their left limits and taking into account \cite{Ishii1985}, we observe that \eqref{eq_uhguhuh3ht9783h8934} still holds. 
Thus, we have
\begin{align*}
\partial_{t}\varphi(x_{\varepsilon}, t_{\varepsilon}) &\leq H_{\alpha}(t_{\varepsilon}, D\varphi(x_{0}, t_{0}), F_{\alpha}(D^{2}\varphi(x_{0}, t_{0}), D\varphi(x_{0}, t_{0}))) \\ & \,\ \,\ \,\ + H_{\alpha}^{\varepsilon}(t_{\varepsilon}, D\varphi(x_{\varepsilon}, t_{\varepsilon}), F_{\alpha}(D^{2}\varphi(x_{\varepsilon}, t_{\varepsilon}), D\varphi(x_{\varepsilon}, t_{\varepsilon})))\\ & \,\ \,\ \,\ \,\ \,\ \,\ -H_{\alpha}^{\varepsilon}(t_{\varepsilon}, D\varphi(x_{0}, t_{0}), F_{\alpha}(D^{2}\varphi(x_{0}, t_{0}), D\varphi(x_{0}, t_{0}))).
\end{align*}
Letting $\varepsilon \to 0+$ and observing that 
\begin{align*}
&H_{\alpha}^{\varepsilon}(t_{\varepsilon}, D\varphi(x_{\varepsilon}, t_{\varepsilon}), F_{\alpha}(D^{2}\varphi(x_{\varepsilon}, t_{\varepsilon}), D\varphi(x_{\varepsilon}, t_{\varepsilon})))\\ & \,\ \,\ \,\ -H_{\alpha}^{\varepsilon}(t_{\varepsilon}, D\varphi(x_{0}, t_{0}), F_{\alpha}(D^{2}\varphi(x_{0}, t_{0}), D\varphi(x_{0}, t_{0}))) \to 0, 
\end{align*}
we obtain
\begin{equation}\label{eq_ineqsubsol11}
\partial_{t}\varphi(x_{0}, t_{0})\leq H_{\alpha}(t_{0}, D\varphi(x_{0}, t_{0}), F_{\alpha}(D^{2}\varphi(x_{0}, t_{0}), D\varphi(x_{0}, t_{0}))),
\end{equation}
since $\varphi$ is smooth and $H_{\alpha}(\cdot, D\varphi(x_{0}, t_{0}), F_{\alpha}(D^{2}\varphi(x_{0}, t_{0}), D\varphi(x_{0}, t_{0})))$ is continuous. 

Next, assume that $|D\varphi(x_{0}, t_{0})|=0$ and $D^{2}\varphi(x_{0}, t_{0})=0$. Let $(x_{\varepsilon}, t_{\varepsilon})$ be a global maximum of $u_{\varepsilon}-\varphi$ such that $(x_{\varepsilon}, t_{\varepsilon})\to  (x_{0}, t_{0})$ as $\varepsilon \to 0+$. Since $u_{\varepsilon}$ is a viscosity solution, we have
\[
\partial_{t} \varphi(x_{\varepsilon}, t_{\varepsilon}) \leq [H^{\alpha}_{\varepsilon}]^{*}(t_{\varepsilon}, D\varphi(x_{\varepsilon}, t_{\varepsilon}), F_{\alpha}(D^{2}\varphi(x_{\varepsilon}, t_{\varepsilon}),  D\varphi(x_{\varepsilon}, t_{\varepsilon}))) \to 0
\]
as $\varepsilon \to 0+$, in view of the definition of $H^{\alpha}_{\varepsilon}$. Thus, $\partial_{t}\varphi(x_{0}, t_{0})\leq 0$ as desired. This, together with \eqref{eq_ineqsubsol11} and Theorem~\ref{thm equivdefvs}, implies that $v$ is a viscosity subsolution of \eqref{weighted MC equation} and completes our proof of Proposition~\ref{prop approachingmcf}.
\end{proof}
\section{Geometric uniqueness in the convex case}
In this section, we obtain the estimate (see Proposition~\ref{prop flowdistanceestim}) of the distance between two generalized evolutions with different external forces. Using this estimate, we prove that if the initial set $\Omega_{0}$ is convex and bounded, then the evolution is unique. The proof of the uniqueness is based on \cite[Theorem~8.4]{Bellettini_Caselles_Chambolle_Novaga}. Thus, we provide a different proof of the uniqueness of the evolution of a convex bounded set than in \cite{Barles-Soner-Souganidis}, namely, where the proof is based on the use of the comparison principle. In general, the inclusion principle and the uniqueness of evolutions follow from the scheme and the comparison principle (see Remark~\ref{remark about IP}). It is worth noting that a counterpart of Proposition~\ref{prop flowdistanceestim} was established earlier in \cite{Chambolle_Novaga_Ruffini_2017} in the context of a \emph{nonlocal} anisotropic mean curvature flow and with a different initial condition (see \cite[Proposition~6.2]{Chambolle_Novaga_Ruffini_2017}).  

We recall that the mobility $\Phi_{\alpha}$ is a norm on $\mathbb{R}^{N}$ (see Lemma~\ref{lem convexitymobility}). For each $\alpha \in [1,2)$, we shall consider the distance $\dist_{\Phi^{\circ}_{\alpha}}$ on $\mathbb{R}^{N}$ induced by the dual norm \[\Phi^{\circ}_{\alpha}(x)=\sup\{\langle \xi, x \rangle: \Phi_{\alpha}(\xi)\leq 1\}\] of $\Phi_{\alpha}$. Given $\eta>0$, $x \in \mathbb{R}^{N}$ and $E\subset \mathbb{R}^{N}$, we define \[d^{\eta}_{\partial E}(x)=-\eta \vee(\eta\wedge(-\mathrm{dist}_{\Phi^{\circ}_{\alpha}}(x,E)+\mathrm{dist}_{\Phi^{\circ}_{\alpha}}(x, E^{c}))),\] so that $d^{\eta}_{\partial E}(x)=\eta \wedge \dist_{\Phi^{\circ}_{\alpha}}(x, E^{c})$ if $x \in E$ and $d^{\eta}_{\partial E}(x)=-\eta \vee -\dist_{\Phi^{\circ}_{\alpha}}(x, E)$ if $x \in E^{c}$. In particular, $d^{\eta}_{\partial E}(x)=0$ if $x \in \partial E$.
\begin{lemma}\label{lem estimondistdual1}
Let $\alpha \in [1,2)$ and $E_{1}\subset E_{2}$ be two nonempty subsets of $\mathbb{R}^{N}$. Let $c_{i} \in \mathbb{R}$ and $E_{i}(t)$ be the evolution of the flow $v_{i}(p)=c_{i}\Phi_{\alpha}(p)$ such that $E_{i}(0)=E_{i}$ for each $i \in \{1,2\}$, which means that $E_{i}(t)=\{x \in \mathbb{R}^{N}: u_{i}(x,t)\geq 0\}$, where $u_{i}:\mathbb{R}^{N}\to \mathbb{R}$ is a unique viscosity solution to the problem
\[
\begin{cases}
\partial_{t}u_{i}= c_{i}\Phi_{\alpha}(Du_{i}),\\
u_{i}(x,0)=d^{\eta}_{\partial E_{i}}(x).
\end{cases}
\]
Assume that $\dist_{\Phi^{\circ}_{\alpha}}(\partial E_{1}, \partial E_{2})>0$.  Then the function
\[
\delta(t)=\mathrm{dist}_{\Phi^{\circ}_{\alpha}}(\partial E_{1}(t), \partial E_{2}(t))
\]
satisfies
\[
\delta(t)\geq \delta(0) + (c_{2}-c_{1})t
\]
for each $t \in [0,t_{s}]$, where $t_{s}=\inf\{\tau > 0: \delta(\tau)=0\}$. 
\end{lemma}
\begin{proof}
First, we assume that $c_{1}, c_{2} \leq 0$. According to the Hopf-Lax formula for the Hamiltonian $H_{i}(p)=|c_{i}|\Phi_{\alpha}(p)$, for each $i \in \{1,2\}$, the solution of the system
\[
\begin{cases}
\partial_{t}u_{i}(x, t) + |c_{i}|\Phi_{\alpha}(Du(x,t))=0,\\
u_{i}(x,0)=d^{\eta}_{\partial E_{i}}(x),
\end{cases}
\]
is given by 
\[
u_{i}(x,t)=\inf_{y \in \mathbb{R}^{N}}\left\{d^{\eta}_{\partial E_{i}}(y)+tH^{*}_{i}\left(\frac{x-y}{t}\right)\right\},
\]
where $H^{*}_{i}$ denotes the Legendre-Fenchel conjugate of $H_{i}$, namely
\[
H^{*}_{i}(\xi)=\begin{cases}
0 \,\  & \text{if}\,\ \Phi^{\circ}_{\alpha}(\xi)\leq |c_{i}|, \\
+\infty \,\ & \text{otherwise}
\end{cases}
\]
(see, for instance, \cite{Bardi_Evans}).
Thus, 
\begin{equation}\label{eq_formulaforviscsolutionhlah}
u_{i}(x,t)=\inf\{d^{\eta}_{\partial E_{i}}(y): \Phi^{\circ}_{\alpha}(y-x) \leq |c_{i}|t,\,\ y \in \mathbb{R}^{N}\}.
\end{equation}
In view of \eqref{eq_formulaforviscsolutionhlah} and the fact that $d^{\eta}_{\partial E_{i}}(y)\leq 0$ for each $y \in E_{i}^{c}$, we have the following
\begin{equation}\label{eq_supelevelusetuniq1}
\{x \in \mathbb{R}^{N}: u_{i}(x,t)>0\}=\{x \in E_{i}:  \mathrm{\dist}_{\Phi^{\circ}_{\alpha}}(x, \partial E_{i}) > |c_{i}|t\}.
\end{equation}
Let $t<t_{s}$ and $x_{i} \in \partial \{x \in \mathbb{R}^{N}: u_{i}(x,t)>0\}$, $i \in \{1,2\}$ satisfy $\delta(t)=\Phi^{\circ}_{\alpha}(x_{1}-x_{2})$. Denote by $\xi$ the unique point of the intersection of $\partial E_{1}$ and $[x_{1}, x_{2}]$. Let $z$ be the projection of $x_{2}$ onto $\partial E_{2}$. Then $\Phi^{\circ}_{\alpha}(z-x_{2})=|c_{2}|t$ and the following holds
\begin{align*}
\delta(t)&=\Phi^{\circ}_{\alpha}(x_{1}-x_{2})\\
&=\Phi^{\circ}_{\alpha}(x_{2}-\xi)+\Phi^{\circ}_{\alpha}(\xi-x_{1})\\
&\geq \Phi^{\circ}_{\alpha}(x_{2}-\xi)+|c_{1}|t\\
&\geq \Phi^{\circ}_{\alpha}(z-\xi)-\Phi^{\circ}_{\alpha}(z-x_{2})+ |c_{1}|t\\
&\geq \delta(0)-|c_{2}|t+|c_{1}|t\\
&=\delta(0)+(c_{2}-c_{1})t.
\end{align*}
This proves Lemma~\ref{lem estimondistdual1} in the case where $c_{1}, c_{2} \leq 0$. The proof in the case where $c_{1}, c_{2} \geq 0$ is similar. Indeed, if $c_{1}, c_{2} \geq 0$ and $u_{i}$ is a viscosity solution of the equation $\partial_{t}u-|c_{i}|\Phi_{\alpha}(Du)=0$ supplemented with the initial condition $u_{i}(x,0)=d^{\eta}_{\partial E_{i}}(x)$, then $v_{i}=-u_{i}$ is a viscosity solution of the equation $\partial_{t}v+|c_{i}|\Phi_{\alpha}(-Dv)=0$ supplemented with the initial condition $v_{i}=-d^{\eta}_{\partial E_{i}}$. In this case, the set $\{x \in \mathbb{R}^{N} :u_{i}(x,t)>0\}$ is the interior of the set $\{x \in \mathbb{R}^{N}: v_{i}(x,t)>0\}^{c}$. If $c_{1}<0$ and $c_{2}>0$, reasoning similarly, we have
\[
E_{1}(t)=\{x \in E_{1}: \mathrm{dist}_{\Phi^{\circ}_{\alpha}}(x, \partial E_{1})\geq |c_{1}|t\}\,\ \text{and}\,\ E_{2}(t)=\{x \in \mathbb{R}^{N}: \mathrm{dist}_{\Phi^{\circ}_{\alpha}}(x, E_{2})\leq c_{2}t\}.
\]
Let $x_{1} \in \partial E_{1}(t)$ and $x_{2} \in \partial E_{2}(t)$ be such that $\delta(t)=\Phi^{\circ}_{\alpha}(x_{1}-x_{2})$. Denote by $y_{i} \in \partial E_{i}$ the unique point of the intersection of $\partial E_{i}$ and $[x_{1}, x_{2}]$ for each $i \in \{1,2\}$. Since 
\begin{align*}
\Phi^{\circ}_{\alpha}(x_{1}-x_{2})&= \Phi^{\circ}_{\alpha}(x_{2}-y_{2})+\Phi^{\circ}_{\alpha}(y_{2}-y_{1})+\Phi^{\circ}_{\alpha}(y_{1}, x_{1}) \\
&\geq \delta(0)+|c_{2}|t+|c_{1}|t\\
&=\delta(0)+(c_{2}-c_{1})t,
\end{align*}
we deduce that $\delta(t)\geq \delta(0)+(c_{2}-c_{1})t$. This completes our proof of  Lemma~\ref{lem estimondistdual1}.
\end{proof}
\begin{prop}\label{prop flowdistanceestim}
Let $\alpha \in [1,2)$, $g_{1}, g_{2} \in C([0,+\infty))$, $E_{1}\subset E_{2}$ be two nonempty subsets of $\mathbb{R}^{N}$ and for each $i \in \{1,2\}$, $u_{i}$ be a unique viscosity solution to the problem
\[
\begin{cases}
\partial_{t}u=\mu_{\alpha}(Du)F_{\alpha}(D^{2}u, Du)+\Phi_{\alpha}(Du)g_{i},\\
u(x, 0)=d^{\eta}_{\partial E_{i}}(x).
\end{cases}
\]
Let for each $t \in [0,+\infty)$, $E_{i}(t)=\{x \in \mathbb{R}^{N}: u_{i}(x,t)\geq 0\}$. Assume that $\dist_{\Phi^{\circ}_{\alpha}}(\partial E_{1}, \partial E_{2})>0$. Then the function $\delta(t)=\dist_{\Phi^{\circ}_{\alpha}}(\partial E_{1}(t), \partial E_{2}(t))$ satisfies
\[
\delta(t)\geq \delta(0)+\int_{0}^{t}(g_{2}(\tau)-g_{1}(\tau)) \diff \tau
\]
for each $t \in [0, t_{s}]$, where $t_{s}=\inf\{\tau > 0: \delta(\tau)=0\}$.
\end{prop}
\begin{proof}
Without loss of generality, we can assume that $\partial E_{i}(t)=\partial \{x \in \mathbb{R}^{N}: u_{i}(x,t)<0\}$, namely, that the front does not develop an interior. 
For each $i \in \{1,2\}$, let $u_{\varepsilon, i}$ and $c_{\varepsilon, i}$ be the functions defined in Section~\ref{Section 6} for $g=g_{i}$. Denote $E_{\varepsilon, i}(t)=\{x \in \mathbb{R}^{N}: u_{\varepsilon, i}(x,t) \geq 0\}$ and $\delta_{\varepsilon}(t)=\dist_{\Phi^{\circ}_{\alpha}}(\partial E_{\varepsilon, 1}(t), \partial E_{\varepsilon, 2}(t))$. By Proposition~\ref{prop approachingmcf}, $\delta_{\varepsilon}(t)\to \delta(t)$ for each $t \in [0,t_{s})$. Fix $t \in [0,t_{s})$ and define $n_{*}=\max\{n\in \mathbb{N}: n\varepsilon <t\}$. Then the following holds
\begin{equation}\label{eq_representationofdistance}
\delta_{\varepsilon}(t)=\delta_{\varepsilon}(0)+(\delta_{\varepsilon}(\varepsilon)-\delta_{\varepsilon}(0))+(\delta_{\varepsilon}(2\varepsilon)-\delta_{\varepsilon}(\varepsilon))+(\delta_{\varepsilon}(3\varepsilon)-\delta_{\varepsilon}(2\varepsilon))+\dotsc+(\delta_{\varepsilon}(t)-\delta_{\varepsilon}(n_{*}\varepsilon)).
\end{equation}
Since the $u_{\varepsilon,i}$'s solve in $(0, \varepsilon]$ the equation $\partial_{t} u=2\Phi_{\alpha}(Du)c_{\varepsilon, i}$, then, according to Lemma~\ref{lem estimondistdual1}, 
\begin{equation}\label{eq_evolutlemmnormalforce}
\delta_{\varepsilon}(\varepsilon)\geq \delta_{\varepsilon}(0)+2\varepsilon(c_{\varepsilon,2}(\varepsilon)-c_{\varepsilon,1}(\varepsilon)). 
\end{equation}
Since the $u_{\varepsilon,i}$'s solve in $(\varepsilon, 2\varepsilon]$ the geometric and translation-invariant equation \[\partial_{t}u=2\mu_{\alpha}(Du)F_{\alpha}(D^{2}u, Du),\] the distance $\delta_{\varepsilon}$ is nondecreasing on $[\varepsilon, 2\varepsilon]$ and hence $\delta_{\varepsilon}(2\varepsilon)\geq \delta_{\varepsilon}(\varepsilon)$.
Using this, \eqref{eq_evolutlemmnormalforce} and repeating the procedure, we obtain 
\begin{equation}\label{eq_concludingcasesprocedure}
\begin{cases}
\delta_{\varepsilon}(k\varepsilon)-\delta_{\varepsilon}((k-1)\varepsilon) \geq 2\varepsilon(c_{\varepsilon, 2}(k\varepsilon)-c_{\varepsilon, 1}(k\varepsilon)) \,\ &\text{if} \,\ k\,\ \text{is odd},\\
\delta_{\varepsilon}(k\varepsilon)-\delta_{\varepsilon}((k-1)\varepsilon) \geq 0 \,\ &\text{otherwise}.
\end{cases}
\end{equation}
Thus, summing over $k \in \{1,\dotsc, n_{*}\}$ and taking into account \eqref{eq_representationofdistance} and \eqref{eq_concludingcasesprocedure}, we have
\[
\delta_{\varepsilon}(t)\geq \delta_{\varepsilon}(0)+2\varepsilon \sum_{l=0}^{\lfloor{\frac{n_{*}-1}{2}\rfloor}} (c_{\varepsilon,2}((2l+1)\varepsilon)-c_{\varepsilon,1}((2l+1)\varepsilon)) \geq \delta_{\varepsilon}(0)+\int_{0}^{2\varepsilon\lfloor{\frac{n_{*}+1}{2}\rfloor}}(g_{2}(\tau)-g_{1}(\tau))\diff \tau.
\]
Letting $\varepsilon\to0+$, yields $\delta(t)\geq \delta(0)+\int_{0}^{t}(g_{2}(\tau)-g_{1}(\tau))\diff \tau$ and completes our proof of Proposition~\ref{prop flowdistanceestim}.
\end{proof}
Using Proposition~\ref{prop flowdistanceestim} and taking into account the proof of \cite[Theorem~8.4]{Bellettini_Caselles_Chambolle_Novaga}, we deduce the next result. 
\begin{cor}\label{cor uniqboundconvexf}
Let $\alpha \in [1,2)$, $g \in C_{b}([0,+\infty))$, $E_{1} \subset E_{2}$ be two compact convex subsets of $\mathbb{R}^{N}$, and $X_{t}(E_{1})$ and $X_{t}(E_{2})$ be the generalized evolutions (see Definition~\ref{def ofgenflow}) corresponding to \eqref{weighted MC equation}. Then $X_{t}(E_{1})\subset X_{t}(E_{2})$ for each $t \geq 0$.
\end{cor}
\begin{proof}[Proof of Corollary~\ref{cor uniqboundconvexf}]
If $E_{1}$ has an empty interior, then $X_{t}(E_{1})=\emptyset$ for all $t>0$ and the proof follows. Thus, we can assume that $E_{1}$ has a nonempty interior and that the origin belongs to the interior of $E_{1}$. Let us fix $s>1$. For each $t \in [0,+\infty)$, we define $g_{1}(t)=g(t)$ and $g_{2}(t)=g(t/s^{2})/s$.  If $u$ is a unique viscosity solution of the equation \eqref{weighted MC equation} supplemented with the initial condition $u(\cdot, 0)=d^{\eta}_{\partial E_{2}}(\cdot)$ in $\mathbb{R}^{N}$, 
then the function $u_{s}(x,t)=su(x/s, t/s^{2})$ is a unique viscosity solution of the equation $\partial_{t} u=\mu_{\alpha}(Du)(F_{\alpha}(D^{2}u, Du)+g_{2}|Du|)$ supplemented with the initial condition $u(\cdot,0)=sd^{\eta}_{\partial E_{2}}(\cdot)$ in $\mathbb{R}^{N}$. Notice that the generalized evolution corresponding to the solution of the latter equation is defined by $sX_{t/s^{2}}(E_{2})$. Setting $\delta_{s}(t)=\dist_{\Phi^{\circ}_{\alpha}}(\partial X_{t}(E_{1}), s\partial X_{t/s^{2}}(E_{2}))$, we observe that $\delta_{s}(0)>0$. According to Proposition~\ref{prop flowdistanceestim},
\[
\delta_{s}(t)\geq \delta_{s}(0)+\int_{0}^{t}\left(\frac{1}{s}g\left(\frac{\tau}{s^{2}}\right)-g(\tau)\right)\diff \tau
\]
for each $t \in [0, \inf\{\tau>0: \delta_{s}(\tau)=0\})$, where
\begin{align*}
\int_{0}^{t}\left(\frac{1}{s}g\left(\frac{\tau}{s^{2}}\right)-g(\tau)\right)\diff \tau &= (s-1)\int_{0}^{t/s^{2}}g(\tau)\diff \tau - \int_{t/s^{2}}^{t}g(\tau)\diff \tau \\
& \leq \frac{(s-1)t\|g\|_{\infty}}{s^{2}}+\frac{t(s-1)(s+1)\|g\|_{\infty}}{s^{2}} \\
&\leq t(s-1)\|g\|_{\infty}+2t(s-1)\|g\|_{\infty}\\
&=3t(s-1)\|g\|_{\infty}
\end{align*}
and $\delta_{s}(0)\geq c(s-1)$, where $c>0$ depends only on $E_{1}$, $E_{2}$ and $\Phi_{\alpha}$. Thus, $\delta_{s}(t)\geq 0$ while $t\leq \frac{c}{3\|g\|_{\infty}}$, which does not depend on $s$. This implies that $X_{t}(E_{1})\subset X_{t}(E_{2})$, which completes our proof of Corollary~\ref{cor uniqboundconvexf}.
\end{proof}
\begin{rem} The same proof shows that a strictly star-shaped domain with respect to a center point $x_{0}$ will have a unique evolution for a positive time as long as no line emanating from $x_{0}$ becomes tangent to its boundary.
\end{rem}
\begin{rem}\label{remark about IP}
It is worth noting that our scheme is monotone. Indeed, if $\Omega_{1}\subset \Omega_{2}$, then $(\Omega_{1})^{h}_{nh}\subset (\Omega_{2})^{h}_{nh}$ for each $h>0$ and for each $n \in \mathbb{N}$, which comes from the definition (see \eqref{eq_nsetdiscreteh}). Let $(X_{t}(\overline{\Omega}_{1}))_{t\geq 0}$  and $(X_{t}(\overline{\Omega}_{2}))_{t\geq 0}$ be the generalized evolutions of $\Omega_{1}$ and $\Omega_{2}$ with uniformly continuous initial conditions $u_{1,0}(\cdot)$ and $u_{2,0}(\cdot)$ such that $u_{1,0}(\cdot) \leq u_{2,0}(\cdot)$ in $\mathbb{R}^{N}$ (recall that $\Omega_{1}=\{u_{1,0}(\cdot)>0\}$ and $\Omega_{2}=\{u_{2,0}(\cdot)>0\}$), respectively (see Definition~\ref{def ofgenflow}). Then the monotonicity of the scheme (in combination with the application of the comparison principle) yields the inclusion $X_{t}(\overline{\Omega}_{1}) \subset X_{t}(\overline{\Omega}_{2})$ for each $t \geq 0$. In order to prove this inclusion principle, assume by contradiction that for some $t>0$ there exists $x \in X_{t}(\overline{\Omega}_{1})$ such that $x \not \in X_{t}(\overline{\Omega}_{2})$. Then there exists $\delta>0$ such that $u_{2}(x,t)<-\delta$, where $u_{i}$ is the unique viscosity solution of \eqref{weighted MC equation} satisfying the initial condition $u_{i}(\cdot, 0)=u_{i,0}(\cdot)$ in $\mathbb{R}^{N}$ for each $i \in \{1,2\}$. Since $u_{1,0}(\cdot)\leq u_{2,0}(\cdot)$, defining $\Omega^{\delta/2}_{i}=\{u_{i,0}(\cdot)+\delta/2>0\}$ for each $i \in \{1,2\}$, we have $\Omega^{\delta/2}_{1}\subset \Omega^{\delta/2}_{2}$. Inasmuch as $u_{1}(x,t)\geq 0>-\delta/2$, $x \in (\Omega^{\delta/2}_{1})_{t}$. Then, according to Theorem~\ref{mainthm}, $u_{1,h}\to 1$  locally uniformly around $x$ as $h\to 0+$. This, together with the monotonicity of the scheme and the fact that $\Omega^{\delta/2}_{1}\subset \Omega^{\delta/2}_{2}$, implies that there exists $\varepsilon>0$ such that $\overline{B}_{\varepsilon}(x) \subset (\Omega^{\delta/2}_{1})^{h}_{nh} \subset (\Omega^{\delta/2}_{2})^{h}_{nh}$ for each $h>0$ small enough, where $nh \to t$ as $h \to 0+$. Next, taking into account that $\sign^{*}(u_{2}(x,t)+\delta/2)$ is the maximal upper semicontinuous subsolution of \eqref{weighted MC equation} supplemented with the initial datum $\mathbbm{1}_{\overline{E}}-\mathbbm{1}_{\overline{E}^{c}}$ (see \cite{Barles-Soner-Souganidis}), where $E=\Omega^{\delta/2}_{2}$, using Proposition~\ref{prop visc subandsupersol}, we deduce that $1=\limsup^{*}u_{2,h}(x,t)\leq \sign^{*}(u_{2}(x,t)+\delta/2)$. Thus, $u_{2}(x,t)+\delta/2\geq 0$ and $u_{2}(x,t)\geq -\delta/2>-\delta$, which leads to a contradiction with the fact that $u_{2}(x, t)<-\delta$. This completes our proof of the inclusion principle. 

\end{rem}
\section{Large times asymptotics}
In this section, we describe the asymptotic behavior of the generalized evolutions corresponding to \eqref{weighted MC equation}, in the limit $t \to +\infty$, in the case where $g\equiv c$ is a positive constant function. Namely, if the initial set $\Omega_{0}$ is bounded and contains a sufficiently  large ball $B_{R}$, then the generalized front propagation is asymptotically similar to the Wulff shape $\mathcal{W}$ of the energy function $c\,\Phi_{\alpha}$, where 
\[
\mathcal{W}=\{x \in \mathbb{R}^{N}: \langle x, p\rangle \leq c\,\Phi_{\alpha}(p)\,\ \text{for all}\,\ p \in \mathbb{S}^{N-1}\}
\]
$c>0$ is a constant and $\Phi_{\alpha}$ is the mobility defined in \eqref{eq_explmobilitycomputah1}. We recall that $\Phi_{\alpha}$ is a norm on $\mathbb{R}^{N}$ (see Lemma~\ref{lem convexitymobility}). It is worth noting that $\mathcal{W}=\{(c\,\Phi_{\alpha})^{\circ}\leq 1\}$ is the unit ball of the dual norm $(c\,\Phi_{\alpha})^{\circ}$ of $c\,\Phi_{\alpha}$. Furthermore, $\mathcal{W}$ is a compact convex subset of $\mathbb{R}^{N}$ with the origin as its interior point (see \cite[Section~5]{Ishii-Pires-Souganidis}) and the following result holds. 
\begin{theorem}\label{theorem ltasymptotic behavior}
Let $\alpha \in [1,2)$ and $c>0$. Then there exists $R=R(\alpha, c, N, \norm)>0$ such that if $\varepsilon>0$ and $\Omega_{0} \subset \mathbb{R}^{N}$ is open, bounded and contains $B_{R}$, then for some $T>0$ and for each $t \geq T$, 
\[
\{x \in \mathcal{W}: \dist(x, \partial \mathcal{W})>\varepsilon\} \subset t^{-1}O_{t}(\Omega_{0})\,\ \text{and}\,\ t^{-1}X_{t}(\smash{\overline{\Omega}}_{0}) \subset \{x \in \mathbb{R}^{N}: \dist(x, \mathcal{W})<\varepsilon\},
\]
where $(X_{t}(\smash{\overline{\Omega}}_{0}), O_{t}(\Omega_{0}))_{t \geq0}$ is the generalized evolution corresponding to \eqref{weighted MC equation} with $g\equiv c$ (see Definition~\ref{def ofgenflow}).  In particular, $t^{-1}(X_{t}(\smash{\overline{\Omega}}_{0})\setminus O_{t}(\Omega_{0})) \to \partial \mathcal{W}$ in the Hausdorff distance as $t \to +\infty$.
\end{theorem}
\begin{proof}
Define $u=\mathbbm{1}_{E}$, where $E=\bigcup_{t\geq 0} O_{t}(\Omega_{0})\times \{t\}$. It is well known (see, for instance, \cite{Barles-Soner-Souganidis}) that $u$ is a viscosity supersolution of the equation 
\begin{equation}\label{eq_m0jtj45jt405j0tj450}
\partial_{t}u=\mu_{\alpha}(Du)F_{\alpha}(D^{2} u, Du) + c\,\Phi_{\alpha}(Du) \,\ \text{in}\,\ \mathbb{R}^{N}\times (0,+\infty).
\end{equation}
According to \cite[Lemma~6.3]{Ishii-Pires-Souganidis}, there exist $R=R(\alpha, c, N, \norm)>0$ and $\delta=\delta(\alpha, c, N, \norm)>0$ such that if $u=1$ on $B_{R}\times \{0\}$ (or, equivalently, $B_{R}\subset \Omega_{0}$), then $u(tx,t)=1$ for each pair $(x,t) \in B_{\delta}\times [0,+\infty)$. This defines our $R>0$. Let $\underaccent{\bar}u:\mathbb{R}^{N}\to \{0,1\}$ be a lower semicontinuous function defined by 
\[
\underaccent{\bar}u(x):=\lim_{\varepsilon \to 0+}\inf\{u(sy, s):s>\varepsilon^{-1},\,\ y \in B_{\varepsilon}(x)\}.
\]
Then $\underaccent{\bar}u$ is a viscosity supersolution of the equation
\begin{equation}\label{eq_n4j05j04j90540900j045jj0j}
-\langle x, Dv \rangle-c\,\Phi_{\alpha}(Dv)=0 \,\ \text{in}\,\ \mathbb{R}^{N}
\end{equation}
(see \cite[Lemma~6.1]{Ishii-Pires-Souganidis}). Indeed, defining the function $f(x, t)=u(tx, t)$, we can show that
\begin{equation}\label{viscossupersolassocequtomcf}
t\partial_{t}f\geq \langle Df, x\rangle + t^{-1}\mu_{\alpha}(Df)F_{\alpha}(D^{2}f, Df)+ c\Phi_{\alpha}(Df) \,\ \text{in}\,\ \mathbb{R}^{N}\times (0,+\infty)
\end{equation}
holds in the viscosity sense. Assume that $\varphi \in C^{2}(\mathbb{R}^{N})$, $\underaccent{\bar}u - \varphi$ has a strict minimum at $\hat{x}$ and \[\lim_{|x|\to +\infty} \varphi(x)=-\infty, \,\ \underaccent{\bar}u(\hat{x})=\varphi(\hat{x}).\] Next, we can define a sequence $(\alpha_{n})_{n \in \mathbb{N}} \subset (0,+\infty)$ such that the function
\[
f_{*}(x,t)-\varphi(x)-(\varepsilon_{n}-\frac{1}{n})e^{-\frac{1}{n}(t-n)}+\alpha_{n}t,
\]
where \[\varepsilon_{n}=\inf_{x\in \mathbb{R}^{N}, t\geq n}(f_{*}(x,t)-\varphi(x))\to 0\] as $n\to+\infty$, achieves its minimum over $\mathbb{R}^{N}\times [n,+\infty)$ at some point $(x_{n}, t_{n})\in \mathbb{R}^{N}\times (n,+\infty)$, where, up to a subsequence (not relabeled), $x_{n} \to \hat{y}$ as $n\to +\infty$. We can choose $\alpha_{n}$ so that 
\[
\varepsilon_{n}-(\varepsilon_{n}-\frac{1}{n})e^{-\frac{1}{n}(t_{n}-n)}+\alpha_{n}t_{n}\leq f_{*}(x_{n},t_{n})-\varphi(x_{n})-(\varepsilon_{n}-\frac{1}{n})e^{-\frac{1}{n}(t_{n}-n)}+\alpha_{n}t_{n}\leq \frac{1}{2n}
\]
and hence $\alpha_{n}t_{n}\to 0$ and $\underaccent{\bar}u(\hat{y})-\varphi(\hat{y})=0$, which implies that $\hat{y}=\hat{x}$, since $\hat{x}$ is a strict minimum of $\underaccent{\bar}u-\varphi$ and $\underaccent{\bar}u(\hat{x})=\varphi(\hat{x})$. Next, since $f_{*}$ is a viscosity solution of \eqref{viscossupersolassocequtomcf}, $\frac{t_{n}}{n}e^{-\frac{1}{n}(t_{n}-n)} \leq 1$ (where $t_{n}>n$) and for some $r>0$, $(x_{n})_{n\in \mathbb{N}} \subset B_{r}(\hat{x})$, there exists a constant $C>0$ independent of $n$ such that
\[
\langle x, D\varphi(x_{n})\rangle +c \Phi_{\alpha}(D\varphi(x_{n})) -\frac{C}{t_{n}} \leq \varepsilon_{n}-\frac{1}{n}.
\]
Letting $n\to +\infty$, we deduce that $-\langle x, D\varphi(\hat{x})\rangle - c\Phi_{\alpha}(D\varphi(\hat{x})) \geq 0$. This proves that $\underaccent{\bar}u$ is a viscosity supersolution of \eqref{eq_n4j05j04j90540900j045jj0j}. Since $\underaccent{\bar}u=1$ in $B_{\delta}$ and $\underaccent{\bar}u$ is a viscosity supersolution of \eqref{eq_n4j05j04j90540900j045jj0j}, \cite[Theorem~5.3]{Ishii-Pires-Souganidis} implies that $\underaccent{\bar}u=1$ in $\text{int}(\mathcal{W})$. This yields that for each $\varepsilon>0$ there exists $T>0$ such that for each $t \geq T$ and for each $x\in \mathcal{W}$ satisfying $\dist(x, \partial \mathcal{W})>\varepsilon$, it holds $u(tx, t)=1$ (since $u$ takes values in $\{0,1\}$). Observing that for each pair $(x,t) \in \mathbb{R}^{N}\times (0,+\infty)$, $u(tx,t)=\mathbbm{1}_{t^{-1}O_{t}(\Omega_{0})}(x)$, we obtain
\[
\{x \in \mathcal{W}: \dist(x, \partial \mathcal{W})>\varepsilon\} \subset t^{-1} O_{t}(\Omega_{0})
\]
for each $t \geq T$. Next, define $w=\mathbbm{1}_{\Sigma}$, where $\Sigma=\bigcup_{t\geq 0} X_{t}(\smash{\overline{\Omega}}_{0})\times \{t\}$. Then $w$ is an upper semicontinuous viscosity subsolution of the equation \eqref{eq_m0jtj45jt405j0tj450} (see \cite{Barles-Soner-Souganidis}). Let $\bar{w}:\mathbb{R}^{N}\to \{0,1\}$ be defined by 
\[
\bar{w}(x)=\lim_{\varepsilon\to 0+}\sup\{w(sy, s): s>\varepsilon^{-1},\,\ y \in B_{\varepsilon}(x)\}.
\]
By \cite[Lemma~6.1]{Ishii-Pires-Souganidis}, $\bar{w}$ is an upper semicontinuous viscosity subsolution of the equation \eqref{eq_n4j05j04j90540900j045jj0j}. According to \cite[Lemma~6.2]{Ishii-Pires-Souganidis}, for some $L=L(\alpha, c, N, \mathcal{N})>0$, $\bar{w}=0$ in $B_{L}^{c}$. Then, applying \cite[Theorem~5.3]{Ishii-Pires-Souganidis}, we deduce that $\bar{w}=0$ in $\mathcal{W}^{c}$, which implies that for each $\varepsilon>0$ there exists $T>0$ such that for each $t \geq T$, 
\[
t^{-1}X_{t}(\smash{\overline{\Omega}}_{0}) \subset \{x \in \mathbb{R}^{N}: \dist(x, \mathcal{W})<\varepsilon\}.
\]
This completes the proof of Theorem~\ref{theorem ltasymptotic behavior}. 
\end{proof}
\bibliography{bib.1}
\bibliographystyle{amsplain}
\end{document}